\documentclass[twoside,11pt]{article}
\usepackage{amsfonts}
\usepackage{amsmath}
\usepackage{multirow}
\usepackage{cases}
\usepackage{color}
\usepackage{algorithm}
\usepackage{algorithmic}
\usepackage{blindtext}

\usepackage{jmlr2e}

\usepackage{mathtools}
\newtheorem{Assumption}{Assumption}
\usepackage{lastpage}
\jmlrheading{}{}{1-\pageref{LastPage}}{; Revised }{}{21-0000}{Shujun Bi, Yonghua Yang and Shaohua Pan}

\ShortHeadings{Sparse Linear Regression: Sequential Convex Relaxation}{Shujun Bi, Yonghua Yang and Shaohua Pan}
\firstpageno{1}

\begin{document}

\title{Sparse Linear Regression: Sequential Convex Relaxation, Robust Restricted Null Space Property, and Variable Selection}

\author{\name Shujun Bi \email bishj@scut.edu.cn \\
       \addr School of Mathematics\\
       South China University of Technology\\
       Guangzhou, 510641, China
       \AND
       \name Yonghua Yang \email  15938292656@163.com \\
       \addr School of Mathematics\\
       South China University of Technology\\
       Guangzhou, 510641, China
       \AND
       \name Shaohua Pan \email shhpan@scut.edu.cn \\
       \addr School of Mathematics\\
       South China University of Technology\\
       Guangzhou, 510641, China}

\editor{}

\maketitle

\begin{abstract}%   <- trailing '%' for backward compatibility of .sty file
 For high dimensional sparse linear regression problems, we propose a sequential convex relaxation algorithm (iSCRA-TL1) by solving inexactly a sequence of truncated $\ell_1$-norm regularized minimization problems, in which the working index sets are constructed iteratively with an adaptive strategy. We employ the robust restricted null space property and sequential restricted null space property (rRNSP and rSRNSP) to provide the theoretical certificates of iSCRA-TL1. Specifically, under a mild rRNSP or rSRNSP, iSCRA-TL1 is shown to identify the support of the true $r$-sparse vector by solving at most $r$ truncated $\ell_1$-norm regularized problems, and the $\ell_1$-norm error bound of its iterates from the oracle solution is also established. As a consequence, an oracle estimator of high-dimensional linear regression problems can be achieved by solving at most $r\!+\!1$ truncated $\ell_1$-norm regularized problems. To the best of our knowledge, this is the first sequential convex relaxation algorithm to produce an oracle estimator under a weaker NSP condition within a specific number of steps,
 provided that the Lasso estimator lacks high quality, say, the supports of its first $r$ largest (in modulus) entries do not coincide with those of the true vector. 
 % ${\rm supp}(x^{\rm Lasso}, r)\neq {\rm supp}(\overline{x})$, $\|x^{\rm Lasso}-\overline{x}\|\nleqslant O(\sqrt{r}\lambda)$
 %  or $\|x^{\rm Lasso}-\overline{x}\|_1\nleqslant O(r\lambda)$
 %  where $x^{\rm Lasso}$ is the Lasso estimator associate to the regularized parameter $\lambda$.
\end{abstract}

% For sparse linear regression problem,
 % many sequential convex relaxation algorithms are proposed,
 % but their theoretical guarantees all require the initial point (e.g., the Lasso estimator) to be close to the true regression vector.
 % Now it is still unknown whether a sequential convex relaxation algorithm can obtain a high
 % quality estimator when the Lasso estimator is bad.


\section{Introduction}

 We are interested in the high dimensional sparse linear regression problem, which aims at estimating the unknown $r$-sparse regression vector $\overline{x}\in\mathbb{R}^n$
 (i.e., $\overline{x}$ is assumed to have only $r$ nonzero components) from the following linear observation model
  \begin{equation}\label{LinearModel}
  b:=A \overline{x}+e,
 \end{equation}
 where $A\in\mathbb{R}^{m \times n}$ with $m<n$ is the design matrix, $b\in \mathbb{R}^m$ is the response vector, and $e\in\mathbb{R}^m$ is a noise vector. Throughout this paper, we assume that $2\le r\ll n$. To resolve this problem, it is quite natural to consider the zero-norm regularized least squares model
 \begin{equation}\label{CS}
  \min_{x\in\mathbb{R}^{n}}\ \frac{1}{2m}\|A x-b\|^2+\lambda\|x\|_0,
 \end{equation}
 where $\|\cdot\|_0$ denotes the zero-norm or cardinality of vectors,  and $\lambda>0$ is the regularization parameter. Such a model appears in the literature \citep[e.g.,][]{Blumensath2008Iterative,Candes-Lasso,Zhang-L0}. Among others, under the sparse eigenvalue condition (SEC) of $\frac{1}{m}A^\top A$ and the condition that $\min_{i\in {\rm supp}(\overline{x})} |\overline{x}_i|$ is not too small, \citet{Zhang-L0} proved that model \eqref{CS} has the variable selection consistency and its unique optimal solution coincides with the oracle estimator of \eqref{LinearModel}, i.e., an optimal solution of 
 \begin{equation}\label{oracle-esti}
  \min_{{\rm supp}(x)\subset{\rm supp}(\overline{x})}\|Ax-b\|^2.
 \end{equation}
 
 Due to the combinatorial property of the zero-norm, it is not an easy task to achieve an optimal solution of \eqref{CS}. A common way to handle this class of NP-hard problems is to use convex relaxation strategy, typically yielding a desirable solution via a single or a sequence of tractable convex optimization problems. The popular $\ell_1$-norm convex relaxation method, also known as the Lasso \citep{Tibshirani-Lasso}, is solving a single convex minimization problem
  \begin{equation}\label{Lone}
   x^{\rm Lasso}\in\mathop{\arg\min}_{x\in\mathbb{R}^n}\ \frac{1}{2m}\|A x-b\|^2+\lambda\|x\|_1.
  \end{equation}
 A host of important theoretical results are available for this estimator, and we here just to name a few. \citet{Knight-Lasso} studied asymptotic distribution of the Lasso estimator. \citet{Bickel-Lasso,Bunea-Lasso,van-Lasso,Ye-Lasso} investigated the oracle property of this estimator under restricted eigenvalue conditions (RECs).
 \citet{Lahiri-Lasso,Meinshausen-Lasso,Zhao-Lasso} proved that the irrepresentable condition is sufficient and essentially necessary for variable selection consistency of this estimator. While \citet{Negahban-Lasso} and \citet{ZTSEV} established the performance bounds for this estimator under REC or SEC. Recently, \citet{Lederer-Lasso} and \citet{Taheri-Lasso}  focused on the calibration of the Lasso's tuning parameter by developing a bootstrap-based estimator of the quantiles of the effective noise. In the sequel, we say that the Lasso estimator is of high quality if its $\ell_1$-norm or $\ell_2$-norm error from the true vector $\overline{x}$ satisfies $\|x^{\rm Lasso}-\overline{x}\|_1\leq O(r\lambda)$ or $\|x^{\rm Lasso}-\overline{x}\|\leq O(\sqrt{r}\lambda)$. Such error bounds were got under a REC in \citep{Bickel-Lasso,Negahban-Lasso}.

 %  The Lasso problem \eqref{Lone} is equivalent to the following $\ell_1$-norm minimization problem
 %  \begin{equation}\label{Lone-LSC}
 %  \mathop{\min}_{x\in\mathbb{R}^n}\big\{\|x\|_1\ \ {\rm s.t.}\ \ \|A x-b\|\leq \varpi\big\}
 % \end{equation}
 % in the sense that
 % for appropriate constants $\varpi$ and $\lambda$, they have the same optimal solution set \citep[see][]{Foucart-CS}.
 % Problem \eqref{Lone-LSC} is known as the basis pursuit problem \citep{Donoho98}.
 % The papers \citep{CaiTRIP,CTao,CohenNSP,Foucart-CS} shown that \eqref{Lone-LSC}
 % is robust and can reconstruct $\overline{x}$ stably
 % under a certain restricted isometry property (RIP) or null space property (NSP) condition.
 % It is well known that under noiseless scenario (i.e. $e=0$), the RIP of order $2r$ is sufficient and
 % the NSP of order $r$ is sufficient and necessary for any $r$-sparse vector $\overline{x}$
 % to be the unique solution of \eqref{Lone-LSC} with $\varpi=0$ \citep{CTao,Foucart-CS};
 % while under noisy scenario (i.e., $e\ne 0$), a robust NSP or RIP condition is common for the robustness of \eqref{Lone-LSC}
 % with respect to measurement errors \citep{CaiTRIP,Foucart-CS}.

 Unfortunately, the Lasso estimator is often biased due to the remarkable difference between the $\ell_1$-norm and the zero-norm, and lacks the consistency for variable selection because the required conditions are not satisfied \citep[see][]{Fan-SCAD,Lahiri-Lasso,Wainwright-Lasso,ZCH-Lasso,ZH-aLasso}, so one cannot expect that the index set of the first $r$ largest (in modulus) entries of the Lasso estimator coincides with the support of the true sparse regression vector, especially when the sample size is small. Inspired by this, some researchers pay their attention to nonconvex surrogates of the zero-norm. Many nonconvex surrogates have received active exploration, including the $\ell_p\,(0<\!p<\!1)$-norm function \citep{LMJ-Lp,LuZS-Lp,ZT-Multi-a,ZH-aLasso}, the logarithm surrogate function \citep{Candesre}, the transformed $\ell_1$-norm function \citep{TL1}, the capped $\ell_1$ function \citep{Zhang-L0,ZT-Multi-b}, the ratio of $\ell_1$ and $\ell_2$ norms \citep{TML1/L2,PTKL1/L2}, the difference of $\ell_1$ and $\ell_2$ norms \citep{WenL1-L2,XinL1-L2}, and the folded concave penalty functions \citep{Fan-SCAD,Fan-FCP,GPcap,Zhang-MCP}. Among others, \citet{Zhang-L0} proved that under appropriate conditions, a global optimal solution of nonconvex surrogate models has a desirable recovery performance, and \citet{Loh15,Loh17} proved that any stationary point of nonconvex surrogates lie within statistical precision of the true vector and achieved the variable selection consistency of nonconvex surrogates, under a restricted strong convexity (RSC) condition and suitable regularity conditions on the surrogates.
 
 For the nonconvex surrogates of the zero-norm regularized problem \eqref{CS}, many tractable sequential convex relaxation algorithms were developed and demonstrated to have better performance than the Lasso in capturing sparse solutions \citep[see][]{Candesre,Fan-FCP,Fan-LAMM,LMJ-Lp,LuZS-Lp,XinL1-L2,TL1,ZT-Multi-a,ZH-aLasso,ZH-oneLasso}. Most of these algorithm are solving at the $k$th iteration a reweighted $\ell_1$-norm regularized problem of the following form
 \begin{equation}\label{wLone}
 \min_{x\in\mathbb{R}^n}\ \frac{1}{2m}\|A x-b\|^2+\lambda\sum_{i=1}^nw_i^{k-1}|x_i|
 \end{equation}
 where $w^{k-1}\!:=(w_1^{k-1}, w_2^{k-1}, \ldots, w_n^{k-1})^\top$ is a weight vector associated with the solution of the $(k\!-\!1)$th iteration and the adopted surrogate function. When $w_{T^{k-1}}^{k-1}=\textbf{1}$ for some $T^{k-1}\!\subset[n]:=\{1,\ldots,n\}$ and $w_{(T^{k-1})^c}^{k-1}\!=\textbf{0}$ for $(T^{k-1})^c:=[n]\backslash T^{k-1}$, where $\textbf{1}$ (resp. \textbf{0}) denotes the vector of all ones (resp. zeros) with dimension known from the context, the above model \eqref{wLone} becomes the truncated $\ell_{1}$-norm regularized minimization problem:
  \begin{equation}\label{TwLone}
 \min_{x\in\mathbb{R}^n}\ \frac{1}{2m}\|A x-b\|^2+\lambda\sum_{i\in T^{k-1}}|x_{i}|.
 \end{equation} 
 \citet{ZH-aLasso} and \citet{ZH-oneLasso} proved that the one-step estimator yielded by solving problem \eqref{wLone} with $k=1$ enjoys the oracle property if the initial estimator $x^0$ is close to the true regression vector $\overline{x}$. \citet{ZT-Multi-a,ZT-Multi-b} proposed a multi-stage convex relaxation algorithm by solving a sequence of subproblems of the form \eqref{wLone}, established the error bound for the solution of every subproblem from the true vector $\overline{x}$, and achieved the variable selection consistency property under an SEC and a premise that the Lasso estimator is of high quality. \citet{Fan-FCP} proposed a local linear approximation (LLA) algorithm by solving a sequence of subproblems of the form \eqref{wLone}, and proved that the LLA algorithm based on the folded concave regularizer can produce the oracle estimator after two iterations if the initial Lasso estimator is close to the true vector $\overline{x}$. \citet{Fan-LAMM} proposed an iterative local adaptive majorization-minimization (I-LAMM) algorithm, an inexact sequential convex relaxation algorithm, and established its strong oracle property for the folded concave regularizer under a localized SEC by requiring that the Lasso estimator 
 is of high quality.

 % Motivated by some gradient-type algorithms for convex optimization \citep{Beck-APG,Ben-CO},
 % some researchers try to extend these
 % algorithms to solve the nonconvex surrogates of problem \eqref{CS} directly \citep[e.g.,][]{Gong-PG,Lin-APG,Loh15}.
 Consider that many nonconvex surroates of the zero-norm still have closed-form proximal mappings. Some researchers apply directly gradient-type methods to solve nonconvex surrogate models of \eqref{CS} \citep[e.g.,][]{Gong-PG,Lin-APG,Loh15}, but a few works provide theoretical certificates for the estimator yielded by this class of algorithm except  \citep{Loh15,Loh17}. Observe that many nonconvex surrogates of the zero-norm such as the capped $\ell_1$ function, the transformed $\ell_1$-norm function, the SCAD and MCP functions, are difference of convex (DC) functions.
 Some sequential convex relaxation algorithms were also developed to seek a stationary point of these nonconvex surrogates by solving a sequence of subproblems of the form
 \begin{equation}\label{DCwLone}
  \min_{x\in\mathbb{R}^n}\ \frac{1}{2m}\|A x-b\|^2+\lambda\Big(\|x\|_1-\sum_{i=1}^nv_i^{k-1}x_i\Big),
 \end{equation}
 where $v^{k-1}\in[-\textbf{1},\textbf{1}]\in\mathbb{R}^n$ is a vector associated with the solution of the $(k\!-\!1)$th iteration and the used surrogate \citep[e.g.,][]{DCPang,DCsp,DC30,XinL1-L2,TL1}. To the best of our knowledge, there are few works to discuss the parameter estimation performance of DC algorithms. We also notice that many iterative hard thresholding (IHT) algorithms are proposed for sparse estimation problem.   Under certain conditions such as RIP or RSC along with restricted Lipschitz smoothness, the parameter estimation accuracy of IHT can be established \citep[see, e.g.,][]{Axiotis-IHT,Blumensath-IHT,Meng-IHT,Yuan-b,Yuan-c,Zhou-IHT}. However, these conditions are too strong and often do not hold.    For more details on the optimization problems involving the zero-norm
 constraint or regularization term, the reader may refer to the recent survey paper \citep{Survey-OPTZN}.

 From the above discussions, the estimation performance of the existing sequential convex relaxation algorithms, developed by solving a sequence of problems \eqref{wLone} or \eqref{TwLone} with $x^0=x^{\rm Lasso}$, depends heavily on the quality of the Lasso estimator and also requires a REC or SEC for the theoretical guarantees. The Lasso estimator is often bad, especially when the sample size is low, and the index set of the first $r$ largest (in modulus) entries even does not coincide with that of the true regression vector $\overline{x}$. In addition, it is straightforward to prove that the REC is stronger than the null space property (NSP) \citep[see][]{SEC-NSP}, a key condition to achieve the exact reconstruction of the following model 
 \begin{equation}\label{Lone-LSC}
 \mathop{\min}_{x\in\mathbb{R}^n}\big\{\|x\|_1\ \ {\rm s.t.}\ \ \|A x-b\|\leq \varpi\big\},
 \end{equation}
 which is equivalent to the Lasso model \eqref{Lone} for an appropriate $\lambda$. However, the NSP is also very restricted, and a toy example is provided in Section \ref{sec3} to show that the existing (robust) NSP and then REC does not hold. Furthermore, a toy example is given in Section \ref{sec4} to show that the LLA algorithm \citep{Fan-FCP} and multi-stage convex relaxation algorithms \citep{ZT-Multi-a,ZT-Multi-b} do not possess variable selection consistency and fail to yield an oracle estimator.  Thus, it is still necessary to develop new sequential convex relaxation algorithms that can produce an oracle estimator under weaker conditions, which is one of the main motivation of this work. Now it is still unclear whether a sequential convex relaxation algorithm can produce a high-quality solution to sparse linear regression problem when the Lasso estimator lacks good quality, and if yes, how many convex relaxation subproblems are enough. Resolving this open question becomes another motivation of this work. In addition, the current research on the variable selection property and error bounds mainly focuses on exact sequential convex relaxation algorithms for the zero-norm regularized problem \eqref{CS} or sparse linear regression problem. Due to errors, in practical computation only inexact solutions are available. This inspires us to study inexact sequential convex relaxation algorithms in this work along the line of \citep{Fan-LAMM}. 
 % These examples also demonstrate that for almost all nonconvex surrogates of \eqref{CS}, including $\ell_{0.5}$-norm, capped $\ell_1$, SCAD and MCP surrogates, bad stationary points always exist, and the existing sequential convex relaxation algorithms would suffer from bad stationary points.
 %---------------------------------------------------------------------------------------
 \subsection{Main contributions}\label{sec1.1}

  We propose an inexact sequential convex relaxation algorithm (iSCRA-TL1, for short) to seek a high-quality solution for the 
  high dimensional sparse linear regression problem, 
  and provide its theoretical certificates by studying the variable selection property and the $\ell_1$-norm estimation error bound.
  The main contributions are listed as follows.
  \begin{itemize}
  \item Inspired by the difference between the first-order optimality condition of \eqref{TwLone} and the stationary point condition of the zero-norm regularized problem \eqref{CS}, we propose the iSCRA-TL1 by  solving inexactly a sequence of truncated $\ell_1$-norm regularized problems \eqref{TwLone} with a simple bound constraint on those entries with indices from $(T^k)^c$, where the index set $T^{k}\!:=T^{k-1}\backslash I^{k}$ is constructed with $I^{k}\!:=\{i\in T^{k-1}\,|\,|x_i^{k}|\geq\varrho\|x_{T^{k-1}}^{k}\|_\infty\}$ for a constant $\varrho\in(0, 1]$ and $T^0:=[n]$.
  As far as we know, this is the first sequential convex relaxation algorithm based on problem \eqref{TwLone} with the index sets $T^k$ satisfying the nonincreasing property $T^k\subset T^{k-1}$ for all $k$. Though the multi-stage convex relaxation method proposed by \citet{ZT-Multi-b} with the capped $\ell_1$-surrogate also solves a sequence of problems \eqref{TwLone}, the involved index sets $T^k$ do not have such a property.

  \item We use the robust (sequential) restricted null space property of $A$, rRNSP (rSRNSP) introduced recently in \citet{Bi-iSCRA}, to provide the theoretical guarantees of iSCRA-TL1. To the best of our knowledge, there is no work to use the null space property of $A$ to study the variable selection property and the $\ell_1$-norm error bound of sequential convex relaxation methods. We also explore the relations among the rRNSP, rSRNSP, robust NSP and REC, and show that the rRNSP and rSRNSP are strictly weaker than the robust NSP in \citet{Foucart-CS}
  and the REC in \citet{Bickel-Lasso}. Among others, the latter two conditions are required to ensure the solution of basis pursuit problem \eqref{Lone-LSC} and the Lasso estimator to have good quality, respectively.   

  \item Under a mild rRNSP or rSRNSP, we prove that iSCRA-TL1 can identify
  the support of the true $r$-sparse regression vector via at most $r$ truncated $\ell_1$-norm regularized subproblems. This means that if one more truncated $\ell_1$-norm regularized problem is solved exactly,
  then the oracle estimator can be achieved and the consistent variable selection property can be established. We also derived the $\ell_1$-norm estimation error bound of each iterate of iSCRA-TL1 from the oracle estimator. As far as we know, this is the first theoretical certificate for a sequential convex relaxation algorithm to achieve the oracle estimator via at most $r\!+\!1$ convex relaxation subproblems under a rSRNSP condition. In particular, the proposed iSCRA-TL1 does not require the Lasso estimator to be of high quality, even allowing the Lasso estimator to be non-unique, ${\rm supp}(x^{\rm Lasso},r)\neq {\rm supp}(\overline{x})$ or $\|x^{\rm Lasso}-\overline{x}\|_1\nleqslant O(r\lambda)$, which is the significant difference between our results and the existing ones in \citep[e.g.,][]{Fan-FCP,ZT-Multi-a}.
  Two toy examples are given in Section \ref{sec4.1} to show that the sequential convex relaxation algorithms
  LLA-SCAD (or MCP) by \citet{Fan-FCP} and the MSCR-cL1 by \citet{ZT-Multi-b}, starting from the Lasso estimator, are inconsistent for variable selection and cannot find the oracle estimator,
  but the exact version of iSCRA-TL1 is successful.
 \end{itemize}
%------------------------------------------------------------------------------------------
 \subsection{Notation}\label{sec1.3}
	
 Throughout this paper, $E$ represents an identity matrix whose dimensional is known from the context. For an integer $k\ge 1$, write $[k]\!:=\{1,\ldots,k\}$ and $[k]_{+}\!:=\{0,1,\ldots,k\}$. For a real number $a$, let $a_+\!:=\max(a, 0)$. For a vector $x\in\mathbb{R}^n$, $\|x\|_1,\|x\|$ and $\|x\|_{\infty}$ denote its $\ell_1$-norm, $\ell_2$-norm and $\ell_\infty$-norm, respectively, ${\rm sign}(x)\!:=({\rm sign}(x_1), {\rm sign}(x_2), \ldots, {\rm sign}(x_n))^{\top}$ denotes its sign vector, ${\rm supp}(x)$ represents its support, ${\rm supp}(x,k)$ denotes the index set of the first $k$ largest entries of the vector $|x|:=(|x_1|, |x_2|,\ldots, |x_n|)^\top$, and $|x|_k^{\downarrow}$ denotes the $k$th largest component of $|x|$. For an index set $J\subset[n]$, $J^c$ means the complement of $J$ in $[n]$, $|J|$ denotes the cardinality of $J$, and $x_J\in\mathbb{R}^{|J|}$ is the subvector of $x\in\mathbb{R}^n$ obtained by removing those entries $x_j$ with $j\notin J$. For a matrix $B\in\mathbb{R}^{m \times n}$, $\|B\|$ and $\|B\|_\infty$
 denote the spectral norm and the elementwise maximum norm of $B$, respectively, and for an index set $J\subset[n]$, $B_J\in\mathbb{R}^{m\times|J|}$ denotes
 the matrix obtained by removing those columns $B_j$ of $B$ with $j\notin J$. For a closed set $\Omega\subset\mathbb{R}^n$, $\delta_{\Omega}$ denotes the indicator function of $\Omega$, i.e., $\delta_{\Omega}(x)=0$ if $x\in\Omega$, otherwise $\delta_{\Omega}(x)=\infty$; while $\mathbb{I}_{\Omega}$ means the characteristic function of $\Omega$, i.e., $\mathbb{I}_{\Omega}(x)=1$ if $x\in \Omega$, otherwise $\mathbb{I}_{\Omega}(x)=0$. We always write $\overline{S}:={\rm supp}(\overline{x})$ and $\overline{S}^{c}=[n]\backslash\overline{S}$.
  
 %------------------------------------------------------------------------------------------
 \subsection{Outline of this paper}\label{sec1.2}

 The rest of this paper is organized as follows. Section \ref{sec2}  characterizes uniformly positive lower bounds of the first $r$ largest
 (in modulus) entries of feasible solutions to the least-squares constraint. Section \ref{sec3} introduces two new types of NSPs, i.e. rRNSP and rSRNSP, and clarify their relation with the existing robust NSP and REC. In Section \ref{sec4}, we describe the iteration steps of iSCRA-TL1 for the zero-norm regularized problem \eqref{CS} or sparse linear regression problem, and provide its theoretical certificates.   Section \ref{sec5} shows the effectiveness of iSCRA-TL1 via numerical tests on synthetic and real data examples. The conclusions are made in Section \ref{sec6}. The proofs of the main results are deferred to the Appendix \ref{secC}.
 
 %------------------------------------------------------------------------------------------

\section{Lower bound of solutions for least-squares constraint}\label{sec2}

 To use the new robust restricted null space properties of $A$ in the next section, we need a uniformly positive lower bound for the first $r$ largest (in modulus) entries of feasible solutions to the least-squares constraint $\|Ax-b\|\le\! \sqrt{\|e\|^2+2m\lambda(1\!+\!\varsigma_0)\|\overline{x}\|_1}$ for some $\varsigma_0\in[0,1)$. This section aims at providing such a bound. Clearly, the most tight one is the value
 \begin{equation}\label{seqbeta1} 	
  \beta_k(\lambda):=\min_{x\in\mathbb{R}^n}\Big\{|x|_{k+1}^{\downarrow}\ \text { s.t.}\ \|Ax-b\|\le\! \sqrt{\|e\|^2+2m\lambda(1\!+\!\varsigma_0)\|\overline{x}\|_1}\Big\}\quad{\rm for}\ \ k\in [r\!-\!1].
 \end{equation} 
 As will be shown in Section \ref{sec4}, the new null space properties of $A$ often involve the array
 $\beta(\lambda):=\{\beta_0(\lambda),\beta_1(\lambda),\ldots,\beta_{r-1}(\lambda)\}$ with $\beta_0(\lambda)$ defined as follows:
 \begin{equation}\label{seqbeta0} 	
 \beta_0(\lambda):=\!\!\min_{z\in \bigcup_{\|\xi\|_\infty\leq \varsigma_0}\mathcal {S}(\xi)}\ \|z\|_{\infty}\ \text { with }\ \mathcal {S}(\xi)\!:=\mathop{\arg\min}_{x\in \mathbb{R}^n}\Big\{\frac{1}{2m}\|Ax-b\|^2+\lambda(\|x\|_1-\langle \xi, x\rangle)\Big\}.
 \end{equation}
 Here we omit the dependence of $\beta_k(\lambda)$ on $\varsigma_0$ for simplicity. For any $\|\xi\|_\infty\!\le\!\varsigma_0$ and $x^*\in\mathcal{S}(\xi)$, using the feasibility of $\overline{x}$ to the minimization problem in the definition of $\mathcal{S}(\xi)$
 leads to 
 \begin{align*}%\label{boundx0x}
  \frac{1}{2m}\|Ax^*-b\|^2+\lambda\sum_{i\in J^c}|x_i^*|-\lambda\langle \xi, x^*\rangle
  &\le\frac{1}{2m}\|A\overline{x}-b\|^2+\lambda\sum_{i\in J^c}|\overline{x}_i|-\lambda\langle \xi, \overline{x}\rangle\nonumber\\
  &\leq {\|e\|^2}/{(2m)}+\lambda(1+\|\xi\|_\infty)\|\overline{x}\|_1\\
  &\le {\|e\|^2}/{(2m)}+\lambda(1\!+\!\varsigma_0)\|\overline{x}\|_1,
  \end{align*}
 which by  $\|x^*\|_1\!-\!\langle \xi, x^*\rangle\!\ge\! 0$ implies that $\|Ax^*-b\|\le\!\sqrt{\|e\|^2+2m\lambda(1\!+\!\varsigma_0)\|\overline{x}\|_1)}$. Hence,  
  $\mathcal {S}(\xi)\subset\{x\in \mathbb{R}^n\,|\,\|Ax-b\|\leq \sqrt{\|e\|^2+2m\lambda(1\!+\!\varsigma_0)\|\overline{x}\|_1}\}$. This, by the definition of $\beta_0(\lambda)$, implies that 
\(
 \beta_{0}(\lambda)\geq\min\limits_{x\in\mathbb{R}^n}\big\{\|x\|_\infty\ \text { s.t.}\ \|Ax-b\|\leq \sqrt{\|e\|^2+2m\lambda(1\!+\!\varsigma_0)\|\overline{x}\|_1}\big\}.
 \)
 Consequently, 
 \[
  \beta_0(\lambda)\ge\beta_1(\lambda)\ge\cdots\ge\beta_{r-1}(\lambda).
 \] 
 
 Note that any optimal solution of problems \eqref{Lone} and \eqref{TwLone} is feasible to the minimization problem in \eqref{seqbeta1}. Together with the above order relation, $\beta_{r-1}(\lambda)>0$ is a sufficient condition for any optimal solution of problems \eqref{Lone} and \eqref{TwLone} to have at least $r$ nonzero components. This means that, to achieve a high-quality solution of the $r$-sparse linear regression problem, $\beta_{r-1}(\lambda)>0$ is generally required. Due to the nonconvexity of the objective function of \eqref{seqbeta1} and the feasible set of \eqref{seqbeta0}, the exact computation of $\beta_{k}(\lambda)$ for $k\in[r\!-\!1]_{+}$ is not an easy task, so we shall derive a positive lower bound for them under suitable conditions. Define
 \begin{equation}\label{sigmaA}
  \sigma_{\!A}(l):=\min _{\|x \|_0\le l,\,\|x\|=1}\frac{1}{\sqrt{m}}\|Ax\|.
 \end{equation}
 Clearly, $(\sigma_{\!A}(l))^2$ is the smallest $l$-sparse eigenvalue of $\frac{1}{m}A^{\top}A$ introduced in \citet{ZTSEV}, $\sigma_{\!A}(l_1)\leq\sigma_{\!A}(l_2)$ when $l_1\ge l_2$, and when $\sigma_{\!A}(l)>0$, it holds that ${\rm rank}(A)\ge l$.
 The following proposition provides a positive lower bound of
 $\beta_{k}(\lambda)$ for all $k\!\in[r\!-\!1]_{+}$ under suitable restrictions on $\sigma_{\!A}(2r\!-\!1)$ and $|\overline{x}|_{r}^{\downarrow}$, whose proof is included in Appendix \ref{secB}. This lower bound will be used in Section \ref{sec4.2} to provide theoretical guarantees of iSCRA-TL1. 
 %--------------------------------------------------------------------------------------------------
 \begin{proposition}\label{bound-feasol}
 Suppose that $\sigma_{\!A}(2r\!-\!1)>0$ and $|\overline{x}|_{r}^{\downarrow}>\frac{2\|e\|+\sqrt{2m\lambda(1+\varsigma_0)\|\overline{x}\|_1}}{\sqrt{m}\sigma_{\!A}(2r-1)}$.
 Then, for any vector $x$ such that $\|Ax-b\|\leq\sqrt{\|e\|^2+2m\lambda(1\!+\!\varsigma_0)\|\overline{x}\|_1}$ and any $k\in[r\!-\!1]_{+}$, 
 \begin{equation}
  |x|_{k+1}^{\downarrow}\ge\vartheta_k:=
  \frac{\sigma_{\!A}(r+k) \sqrt{m(r\!-\!k)}\,|\overline{x}|_{r}^{\downarrow}-\|e\|-\sqrt{\|e\|^2+2m\lambda(1\!+\!\varsigma_0)\|\overline{x}\|_1}}{\sqrt{n-k}\|A\|}>0,\nonumber
 \end{equation}
 and consequently $\beta_k(\lambda)\ge\vartheta_k>0$ for all $k\in[r\!-\!1]_{+}$.
 \end{proposition} 
 \begin{remark}
 Suppose that $e\sim N(\textbf{0},\sigma^2E)$. Then, when $m$ is suitably large, the event $\|e\|^2\leq \sigma^2(m+2\sqrt{2m})$ will occur with high probability.
 In addition, if choosing $\lambda=O(\sigma\sqrt{\frac{\log n}{m}})$ \citep[see][]{Bickel-Lasso,Negahban-Lasso}, then $\lambda\|\overline{x}\|_1=O(\sigma r\sqrt{\frac{\log n}{m}})$. Thus, the assumption on $|\overline{x}|_{r}^{\downarrow}$ in Proposition \ref{bound-feasol} will hold with high probability when $|\overline{x}|_{r}^{\downarrow}\geq c(\sigma\!+\!\sqrt{\sigma})$ for some constant $c>0$ (depends on the smallest $(2r\!-\!1)$-sparse singular value $\sigma_A(2r\!-\!1)$). As the sample size $m$ increases, it is more likely for $\vartheta_{r-1}>0$ to hold by noting that $\lambda$ becomes small  and $\frac{\vartheta_{r-1}\sqrt{n-r+1}\|A\|}{\sqrt{m}}=\sigma_{\!A}(2r\!-\!1)\,|\overline{x}|_{r}^{\downarrow}-\|e\|/m-\sqrt{\|e\|^2/m+2\lambda(1\!+\!\varsigma_0)\|\overline{x}\|_1}$. 
 \end{remark} 
%----------------------------------------------------------------------------------------------
 \section{Robust restricted null space properties}\label{sec3}

 Recently, \citet{Bi-iSCRA} introduced two new restricted null space properties of $A$, and proved that they are weaker than the common stable NSP, while their robust versions are not stronger than the existing robust NSP. In this section, we show that the robust versions of the new restricted null space properties are strictly weaker than the existing robust NSP \citep{Foucart-CS} and REC \citep{Bickel-Lasso,Negahban-Lasso}. It is well known that the robust NSP is the key condition to achieve the robustness of the $\ell_1$-norm minimization problem \eqref{Lone-LSC} with respect to measurement errors, while the REC is the key condition to derive $\ell_1$-norm and $\ell_2$-norm error bounds for the Lasso estimator.
%--------------------------------------------------------------------------------------------
 \begin{definition}\label{NSP-def}
  The matrix $A$ is said to have the robust null space property of order $r$ with constants $\gamma\in(0,1)$ and $\tau>0$
  if for any $S\subset[n]$ with $|S|=r$ and $d\in\mathbb{R}^n$,
  \begin{equation}\label{ineq1-NSP}
  \sum_{i\in S} |d_i| \le \gamma\sum_{i \in S^c}|d_i|+\tau\sqrt{{r}/{m}}\|A d\|.
  \end{equation}
 \end{definition}

 It was shown by \citep[][Theorem 6.13]{Foucart-CS} that if the $2r$th restricted isometry constant of $\frac{1}{\sqrt{m}}A$ is smaller than $\frac{4}{\sqrt{41}}$, then $A$ satisfies the robust null space property of order $r$ with constants $\gamma\in(0,1)$ and $\tau>0$ depending only on its $2r$th restricted isometry constant. As is well known, the robust NSP of order $r$ is very strong and usually does not hold due to the remarkable difference between the $\ell_1$-norm and the zero-norm. This inspires us to introduce two novel robust restricted null space properties of $A$ in \citep[][]{Bi-iSCRA}.
 %---------------------------------------------------------------------------------------------
 \begin{definition}\label{RNSP-def}
  Fix any $l\in[r]$ and $\eta\ge 0$. The matrix $A$ is said to have the $(l,\eta)$-robust restricted null space property with constants $M>0,\,\gamma\in(0,1)$ and $\tau>0$, denoted by $(l,\eta)$-rRNSP($M,\gamma,\tau$), if for any $I\subset S\subset[n]$ with $|I|=l,\,|S|=r$ and $d\in\mathbb{R}^n$ with $\|d_{S^c}\|_{\infty} \ge\eta$,
  \begin{equation}\label{ineq1-RNSP}
   \sum_{i\in I}\min\big\{|d_i|,\,2M\!-\!|d_i|\big\}
   \le\gamma\sum_{i \in S^c}|d_i|+\tau\sqrt{{r}/{m}}\|Ad\|.
  \end{equation}
 \end{definition}
 %-----------------------------------------------------------------------------------
 \begin{remark}\label{RNSP-remark1}
 As the left side of \eqref{ineq1-RNSP} is no more than that of \eqref{ineq1-NSP}, if inequality \eqref{ineq1-RNSP} does not hold for some $d\in\mathbb{R}^n$ with $\|d_{S^c}\|_{\infty} \ge\eta$, then inequality \eqref{ineq1-NSP} does not either. This shows that the $(l,\eta)$-rRNSP($M,\gamma,\tau$) of $A$ for any $l\in[r]$ is not stronger than its robust NSP of order $r$ with constants $\gamma\in(0,1)$ and $\tau>0$. In fact, due to Example \ref{Exam1-rSRNSP} later, the $(l,\eta)$-rRNSP($M,\gamma,\tau$) of $A$ is strictly weaker than its robust NSP of order $r$.
\end{remark}
 %----------------------------------------------------------------------------------------------
 \begin{definition}\label{SRNSP-def}
  Let $\alpha=\{\alpha_0,\alpha_1,\ldots,\alpha_{r-1}\}\ne 0$ be a nonincreasing ordered nonnegative array, i.e.
  $\alpha_0\ge\alpha_1\ge\cdots\geq \alpha_{r-1}\geq0$.
  The matrix $A$ is said to have the robust sequential RNSP of order $r$
  relative to array $\alpha$ with constants $M>0,\,\gamma\in(0,1)$ and $\tau>0$, denoted by $(r,\alpha)$-rSRNSP($M,\gamma,\tau$),
  if for any $\emptyset\neq I \subset S$ with $|S|=r$ and $d\in\mathbb{R}^n$ with $\|d_{S^c}\|_{\infty} \ge \alpha_{r-|I|}$,
  \begin{equation}\label{ineq1-SRNSP}
  \sum_{i \in I} \min\{|d_i|,\,2M\!-|d_i|\}
  \le\gamma\sum_{i \in S^c}|d_i|+\tau\sqrt{{r}/{m}}\|Ad\|.
  \end{equation}
 \end{definition}
  %-----------------------------------------------------------------------
 \begin{remark}\label{SRNSP-remark1}
 Compared with Definition \ref{RNSP-def}, the $(r,\alpha)$-rSRNSP($M,\gamma,\tau$) of $A$ is equivalent to saying that
 $A$ has the $(l,\alpha_{r-l})$-rRNSP($M,\gamma,\tau$) for all $l\in[r]$. Together with Remark \ref{RNSP-remark1}, the $(r,\alpha)$-rSRNSP of $A$ with constants $M>0$, $\gamma\in(0,1),\,\tau>0$ is not stronger than its robust NSP of order $r$ with constants $\gamma$ and $\tau$. Thus, from Example \ref{Exam1-rSRNSP} later, we conclude that the $(r,\alpha)$-rSRNSP($M,\gamma,\tau$) of $A$
 is strictly weaker than its robust NSP of order $r$.
 \end{remark}
 \begin{remark}\label{Rem-rSRNSP}
  By Definition \ref{RNSP-def}, for a given $l\in[r]$, when $\eta$ gets larger, $M$ gets smaller or $\tau$ gets larger, the corresponding $(l,\eta)$-rRNSP$(M,\gamma,\tau)$ becomes weaker;
  while by Definition \ref{SRNSP-def}, when the nonincreasing ordered nonnegative array $\alpha=\{\alpha_0,\alpha_1,\ldots,\alpha_{r-1}\}$ becomes larger, $M$ gets smaller or $\tau$ gets larger, the corresponding $(r,\alpha)$-rSRNSP$(M,\gamma,\tau)$ becomes weaker. Of course, the values of $\eta,M$ and $\tau$ appearing in $(l,\eta)$-rRNSP$(M,\gamma,\tau)$ depend on the property of $A$, so do those of $\alpha,M$ and $\tau$ appearing in $(r,\alpha)$-rSRNSP$(M,\gamma,\tau)$. Now it is unclear which type of matrices satisfies the  $(l,\eta)$-rRNSP$(M,\gamma,\tau)$ with larger $\eta$ or smaller $M$ and the $(r,\alpha)$-rSRNSP$(M,\gamma,\tau)$ with larger $\alpha$ or smaller $M$. From \citep[][Theorem 6.13 and Theorem 9.2]{Foucart-CS}, an $m\times n$ Gaussian matrix $A$ satisfies the robust NSP with a high probability when $m$ is suitably large, so it satisfies $(l,\eta)$-rRNSP$(M,\gamma,\tau)$ for any $l\in[r]$ and $\eta>0,M>0$, and the $(r,\alpha)$-rSRNSP$(M,\gamma,\tau)$ for any $M>0$ and nonincreasing ordered nonnegative array $\alpha$ with a high probability for such $m$.
  \end{remark}

 The REC, first introduced by \citet{Bickel-Lasso}, a common condition to achieve the $\ell_1$-norm and $\ell_2$-norm error bound of the Lasso estimator. It takes the following form
 \begin{equation}\label{REC-ineq}
  0<\chi(c):=\min_{S\subset[n], |S|=r}\min_{ \|d_{S^c}\|_1\leq c\|d_S\|_1, \|d\|=1}\ \frac{1}{\sqrt{m}}\|Ad\|\ \ {\rm for}\ c>0.
 \end{equation}
 By comparing with the definition of $\sigma_{\!A}(r)$, for any $c>0$, if $\chi(c)>0$, then $\sigma_{\!A}((1+c)r)>0$. 
The following proposition along with Example \ref{Exam1-rSRNSP} shows that the REC with $c>1$ is strictly stronger than the $(r,\beta(\lambda))$-rSRNSP$(\|\overline{x}\|_{\infty},\frac{1}{c},\frac{1}{\chi(c)})$ and $(r,\beta_0(\lambda))$-rRNSP$(\|\overline{x}\|_{\infty},\frac{1}{c},\frac{1}{\chi(c)})$, and is not weaker than the robust NSP with constants $\frac{1}{c}$ and $\frac{1}{\chi(c)}$.
%---------------------------------------------------------------------------
 \begin{proposition}\label{REC-rNSP}
  If the REC holds with constant $c>1$, then the matrix $A$ satisfies the robust NSP of order $r$ with constants $\gamma=\frac{1}{c}$ and $\tau=\frac{1}{\chi(c)}$, and consequently, $A$ has the $(r,\beta(\lambda))$-rSRNSP($M,\gamma,\tau$) and $(r,\beta_{r-1}(\lambda))$-rRNSP($M,\gamma,\tau$) for $\beta(\lambda)=\{\beta_0(\lambda),\beta_1(\lambda),\ldots,\beta_{r-1}(\lambda)\}$ and any $M>0$, where $\beta_j(\lambda)$ for $j\in[r\!-\!1]_{+}$ is defined by \eqref{seqbeta1}-\eqref{seqbeta0}.
  \end{proposition}

 The proof of Proposition \ref{REC-rNSP} is put in Appendix \ref{app-3-3}.
 The following example shows that the $(r,\beta(\lambda))$-rSRNSP($M,\gamma,\tau$) of the matrix $A$ with $\beta(\lambda)=\{\beta_0(\lambda),\beta_1(\lambda),\ldots,\beta_{r-1}(\lambda)\}$ and constants $M=\|\overline{x}\|_\infty$, $\gamma\in(0, 1),\,\tau>0$ is strictly weaker than the robust NSP and REC.
%-----------------------------------------------------------------------------------------
 \begin{example}\label{Exam1-rSRNSP}
  Let the true vector $\overline{x}\in\mathbb{R}^5$, the design matrix $A\in\mathbb{R}^{4\times 5}$, and the response vector $b\in\mathbb{R}^4$ (without noise for simplicity) be given as follows
  \begin{equation}\label{ExampAb-NSP-s}
  \overline{x}=\left(
   \begin{array}{c}
    2\\
   	10 \\
   	0\\
   	0 \\
   	0 \\
   \end{array}
   \right),\ A=\frac{1}{2}\left(
  	\begin{array}{ccccc}
  	    1 &    0      & -2         & 0          & 0\\
  		1 &    0      & 0          & -2         & 0\\
        1 &    0      & 0          & 0          & -2\\
  	   -1 &    2      & 0          & 0          & 0\\
  	\end{array}
  	\right),\
   b=A\overline{x}=\left(
  	\begin{array}{c}
  		1\\
        1\\
  		1 \\
  		9 \\
  	\end{array}\right).
  \end{equation}
  Obviously, $r=2$. As ${\rm rank}(A)=4$, the null space of $A$ is ${\rm Null}(A)=\{( 2t,\, t,\, t,\, t,\,t)^\top\ |\ t\in \mathbb{R}\}$.

 \noindent
 {\bf(1) The REC for any $c\geq2$ is false.}	For any $d^0\in {\rm Null}(A)$ and $S\subset[5]$ with $|S|=r$, we have $\|d^0_{S^c}\|_1\leq c\|d^0_S\|_1$ for any $c\ge 2$, so $\chi(c)=\min\limits_{S\subset[5], |S|=r}\min\limits_{ \|d_{S^c}\|_1\leq c\|d_S\|_1, \|d\|=1}\ \frac{1}{2}\|Ad\|=0$. Hence, the REC for any $c\!\geq\! 2$ is false. Generally, $c=3$ is needed for
  $\|x^{\rm Lasso}-\overline{x}\|\leq O(\sqrt{s}\lambda)$ and $\|x^{\rm Lasso}-\overline{x}\|_1\leq O(s\lambda)$ \citep[see][]{Bickel-Lasso,Negahban-Lasso}.

 \medskip
 \noindent
 {\bf(2) The robust NSP of order $r$ of $A$ is false.} For any $0\neq d\in {\rm Null}(A)$, choosing $S=\{1, 2\}$ can show that inequality \eqref{ineq1-NSP} is false for any $\gamma\in (0,1)$ and $\tau>0$, so it is impossible for the robust NSP of order $r$ of $A$ to hold.

 % Without (robust) NSP condition, we cannot expect a high quality estimation for $\overline{x}$ by solving the basis pursuit problem \eqref{Lone-LSC}. Indeed, from \eqref{ExampAb-NSP-s}, we have
 %  \(
 %  \{x\in  \mathbb{R}^5\ |\ Ax=b\}=\{(2-2t,\,10-t,\,-t,\,-t,\,-t)^\top\ |\  t\in \mathbb{R}\}.
 %  \)
 %  Then, the optimal solution set of problem \eqref{Lone-LSC} with $\varpi=0$ is $ \{(2-2t,\,10-t,\,-t,\,-t,\,-t)^\top\ |\ t\in [0, 1]\}$. Obviously, when an optimal solution $x^*$ of \eqref{Lone-LSC} with $x_5^*=1$, $\|x^*-\overline{x}\|$

  \medskip
  \noindent
  {\bf(3) The $(2, \beta(\lambda))$-rSRNSP($M, \gamma, \tau$) of $A$ holds with $0<\lambda\leq\frac{1}{16\|\overline{x}\|_1},\,M=10,\,\gamma=0.7$ and $\tau=200$.} Here, $\beta(\lambda)=\{\beta_0(\lambda),\, \beta_1(\lambda)\}$ with $\beta_1(\lambda)$ defined by \eqref{seqbeta1} and $\beta_0(\lambda)$ defined by \eqref{seqbeta0} for $\varsigma_0=0$. The solution set of the Lasso problem \eqref{Lone} with $A$ and $b$ from \eqref{ExampAb-NSP-s} is
  \[
  \mathcal {S}^{\rm Lasso}=\{(2+2t-8\lambda,\,10+t-8\lambda,\,t,\,t,\,t)^\top\ |\ 4\lambda-1\leq t\leq 0\}.
  \]
  By the expression of $\mathcal {S}(\cdot)$ in \eqref{seqbeta0}, we have $\mathcal {S}(0)=\mathcal {S}^{\rm Lasso}$, so $\beta_0(\lambda)=9-4\lambda$. As $\sigma_{\!A}(3)> 0.25$,
  Proposition \ref{bound-feasol} with $e=0,r=2,\,m=4,n=5$ implies that 
  \(
  \beta_1(\lambda)>\frac{1-\sqrt{8\lambda\|\overline{x}\|_1}}{2\|A\|}>\frac{\sqrt{2}-1}{2\sqrt{2}\|A\|}.
 \)
 
 Next we show that the matrix $A$ satifies $(2, \beta')$-rSRNSP($M,\gamma,\tau$) for $\beta'\!=\!\{8, 0\},\,M\!=\!10$, $\gamma\!=\!0.7$ and $\tau\!=\!200$, and the conclusion then follows  Remark \ref{Rem-rSRNSP}. Fix any $d\in \mathbb{R}^5\backslash\{0\}$. After calculation, its orthogonal projection onto ${\rm Null}(A)$ is 
  \(
  d^0:=\frac{1}{8}(2d_1\!+\!\sum_{i=2}^5d_i)(2, 1, 1, 1, 1)^\top.
  \)
  Let $d':=d-d^0$. As $d'\in{\rm Range}(A^{\top})$ and the smallest nonzero singular value of $A$ is $1$, we have
  \(
  \|Ad\|=\|Ad'\|\geq \|d'\|.
  \)
  We proceed the arguments by the two cases as shown below.

  \noindent
  {\bf Case I: $\|d'\|\geq 0.01\|d\|$.} For any $\emptyset\neq I\subset S\subset[5]$ with $|S|=2$, it holds that
   \[
  \sum_{i \in I} \min\{|d_i|,\,2M\!-|d_i|\}\leq \sum_{i \in S}|d_i|\leq\sqrt{2}\|d_S\|\leq\sqrt{2}\|d\|\leq 100\sqrt{2}\|d'\|
  \le \tau\sqrt{{r}/{m}}\|Ad\|,
  \]
  where the last inequality is due to $\|d'\|\leq \|Ad\|$ and $\tau=200$. Then, inequality \eqref{ineq1-SRNSP} holds.

 \noindent
  {\bf Case II: $\|d'\|< 0.01\|d\|$.} Now
  $\|d^0\|=2\sqrt{2}|d^0_2|=\|d-d'\|\geq \|d\|-\|d'\|\geq0.99\|d\|$.
  Fix any $S \subset[5]$ with $|S|=2$. We proceed the arguments by two subcases: $1\in S^c$ and $1\in S$.

  \noindent
 {\bf Subcase II.1: $1\in S^c$.} For any $\emptyset\neq I\subset S$, using the expression of $d^0$ leads to
 \[
  \sum_{i \in I} \min\{|d_i|,\,2M\!-|d_i|\}\leq \sum_{i \in I}|d_i|\leq\sum_{i \in I}(|d^0_i|+|d'_i|)\leq 2|d^0_2|+0.02\|d\|,
  \]
  where the last inequality is due to $1\leq |I|\leq 2$ and $1\notin I$.
  Using $|S^c|=3$ and $1\in S^c$ yields
  \[
  \sum_{i \in S^c}|d_i|\geq\sum_{i \in S^c}(|d_i^0|-\|d'\|_\infty)\geq 2(|d^0_2|-\|d'\|_\infty)+|d^0_1|-\|d'\|_\infty\geq 4|d^0_2|-0.03\|d\|.
  \]
  Combining the above two inequalities and $|d^0_2|\geq \frac{0.99}{2\sqrt{2}}\|d\|$ shows that inequality \eqref{ineq1-SRNSP} holds.

 \noindent
 {\bf Subcase II.2: $1\in S$.} In this subcase, $S=\{1, k\}$ for some $k\in\{2,3,4,5\}$. We first consider any $\emptyset\neq I\subset S$ with $I\neq S$. Then, $I=\{k\}$ or $I=\{1\}$, which implies that
 \[
  \sum_{i \in I} \min\{|d_i|,\,2M\!-|d_i|\}\leq\max(|d_k|, |d_1|)\leq \max(|d^0_k|+|d'_k|,|d^0_1|+|d'_1|)\leq 2|d^0_2|+0.01\|d\|.
 \]
 On the other hand, using $1\notin S^c$ and $|S^c|=3$ results in
 \[
  \sum_{i \in S^c}|d_i|\geq \sum_{i \in S^c}(|d_i^0|-\|d'\|_\infty) = 3|d_2^0|-3\|d'\|_\infty\geq 3|d^0_2|-0.03\|d\|.
 \]
 From the above two inequalities and $|d^0_2|\geq \frac{0.99}{2\sqrt{2}}\|d\|$, it follows that
   \begin{align}\label{NSP-RECp1-s}
  \sum_{i \in I} \min\{|d_i|,\,2M\!-|d_i|\}\le\gamma\sum_{i \in S^c}|d_i|+\tau\sqrt{{r}/{m}}\|Ad\|.
   \end{align}
 Now we consider $I=S$. For any $d\in \mathbb{R}^5$ with $\|d_{S^c}\|_\infty\geq \beta'_0=8$, as $S=\{1, k\}$, there exists an index $j\in[5]\backslash\{1, k\}$ such that
 $|d_j|\geq8$. Using $|d^0_j|=|d^0_2|\geq \frac{0.99}{2\sqrt{2}}\|d\|$ leads to
  \[
  |d_j^0|=|d_j-d_j'|\geq |d_j|-\|d'\|_\infty\geq 8-0.01\|d\|\geq 8-\frac{2\sqrt{2}}{99}|d^0_j|\ \Rightarrow\ |d_j^0|=|d_2^0|\geq \frac{8}{1+2\sqrt{2}/99},
  \]
  which implies that
  \(
  |d_1|\geq  |d_1^0|-\|d'\|\geq |d_1^0|-0.01\|d\| \geq 2|d_2^0| - \frac{2\sqrt{2}}{99}|d^0_2|>10=M.
  \)
  Then,
   \[
  \sum_{i \in S} \min\{|d_i|,\,2M-|d_i|\}\leq |d^0_k|+|d'_k|+20-|d_1|\leq |d^0_k|+20-2|d_2^0|+ 2\|d'\|_\infty\leq 20-|d_2^0|+ 2\|d'\|_\infty,
  \]
  where the second inequality is due to $|d_1|\geq |d_1^0|-|d_1'|\geq 2|d_2^0|-\|d'\|_\infty$. On the other hand, it holds that
  \(
 \sum_{i \in S^c}|d_i|\geq \sum_{i \in S^c}(|d_i^0|-\|d'\|_\infty) = 3|d_2^0|-3\|d'\|_\infty,
  \)
  so we have
 \begin{equation}\label{NSP-RECp2-s}
  \sum_{i \in S} \min\{|d_i|,\,2M\!-|d_i|\}\le\gamma\sum_{i \in S^c}|d_i|+\tau\sqrt{{r}/{m}}\|Ad\|.
  \end{equation}
  The above inequalities \eqref{NSP-RECp1-s}-\eqref{NSP-RECp2-s}
  show that inequality \eqref{ineq1-SRNSP} holds in this subcase.

 % From the above arguments, we conclude that the matrix $A$ in \eqref{ExampAb-NSP-s} satisfies the $(2, \beta')$-rSRNSP($M, \gamma, \tau$), so has the $(2, \beta(\lambda))$-rSRNSP($M, \gamma, \tau$).
 \end{example}
%----------------------------------------------------------------------------------
 \section{iSCRA-TL1 for sparse estimation problems}\label{sec4}

 Note that a vector $\widetilde{x}\in\mathbb{R}^n$ is a stationary point of the zero-norm regularized problem \eqref{CS} iff
 \begin{equation}\label{scond}
   0\in m^{-1}A^{\top}(A\widetilde{x}-b)+\lambda D_{\widetilde{x}}\ \ {\rm with}\ \ D_{\widetilde{x}}:=\big\{d\in \mathbb{R}^{n}\ |\ d_i=0\ {\rm if}\ \widetilde{x}_i\neq0\big\};
 \end{equation}
 while $x^{k}$ is an optimal solution of the truncated $\ell_1$-norm regularized problem \eqref{TwLone} iff
  \begin{align*}
  & 0\in m^{-1}A^{\top}(Ax^{k}-b)+\lambda\widehat{D}_{\!x^{k}}\ \ {\rm with} \nonumber\\
  \widehat{D}_{x^{k}}:=\big\{d\in [-\textbf{1},\textbf{1}]\ |\ &\ d_i={\rm sign}(x^{k}_i)\ {\rm if}\ i\in T^{k-1}\cap{\rm supp}(x^k);\ d_{(T^{k-1})^c}=0\big\}.
 \end{align*}
 By comparing with \eqref{scond}, an optimal solution $x^{k}$ of \eqref{TwLone} becomes a stationary point of \eqref{CS} iff
 \begin{equation}\label{opti-cond}
 T^{k-1}\cap{\rm supp}(x^k)=\emptyset.
 \end{equation}
 This inspires us to reduce the number of elements in $T^{k-1}\cap{\rm supp}(x^k)$ iteratively by solving inexactly a sequence of truncated $\ell_1$-norm regularized problems \eqref{TwLone}, and develop an inexact sequential convex relaxation algorithm for the zero-norm regularized problem \eqref{CS}.
%---------------------------------------------------------------------------------
 \subsection{Inexact sequential convex relaxation algorithm}\label{sec4.1}

 The iterations of our inexact sequential convex relaxation algorithm are described as follows.

 %-------------------------------------------------------------------------------------------------
  \begin{algorithm}[H]
  \caption{\label{Alg1}{(iSCRA-TL1 for sparse linear regression problem)}}
  	
  \textbf{Input:}\ $\epsilon\ge 0,\,\mu>0,\,\varrho\in(0, 1],\,T^{0}\!=[n]$ and a nonincreasing $\{\varsigma_{k}\}_{k\geq0}\subset\mathbb{R}_{+}$ with $\varsigma_0<1$.
  	
 \noindent
 \textbf{For} $k=1,2,\ldots$ \textbf{do}
  	
 \quad\ \textbf{1)}\ \ Compute the following problem with an error vector $\xi^{k-1}$ satisfying $\|\xi^{k-1}\|_\infty\leq\varsigma_{k-1}$:
  \begin{equation}\label{inexact-subprob}
  x^{k}\in\mathop{\arg\min}_{x\in\mathbb{R}^{n}}\bigg\{\frac{1}{2m}\|Ax\!-\!b\|^2+\lambda\!\!\sum_{i\in T^{k-1}}\!\!|x_{i}|-\lambda\langle \xi^{k-1}, x\rangle
  \ \ {\rm s.t.}\ \ \|x_{(T^{k-1})^c}\|_{\infty} \le \mu\bigg\}.
  \end{equation}
  	
  \quad\ \textbf{2)}\ \ If $\|x_{T^{k-1}}^{k}\|_\infty\le\epsilon$ or $T^{k-1}=\emptyset$, stop; else go to Step 3.
  	
  \quad\ \textbf{3)}\ \ Let $I^{k}\!:=\big\{i\in T^{k-1}\ |\ |x_i^{k}|\geq\varrho\|x_{T^{k-1}}^{k}\|_\infty\big\}$.
  Set $ T^{k}\!:=\!T^{k-1}\backslash I^{k}$.
  	
 \textbf{End (for)}
 \end{algorithm}
%-----------------------------------------------------------------------------------------------
 \begin{remark}\label{Remark-Alg1}
 {\bf(a)}
 An $\ell_\infty$-norm constraint is added to every subproblem \eqref{inexact-subprob} so as to ensure its well-definedness. Indeed, subproblem \eqref{inexact-subprob} is a convex quadratic program with a nonempty feasible set, and its objective function is lower bounded on the feasible set because $\|\xi^k\|_{\infty}\le\varsigma_{k}\le\varsigma_0<1$ for each $k\in\mathbb{N}$, which implies that every subproblem has an optimal solution and Algorithm \ref{Alg1} is well defined. For the exact version of iSCRA-TL1, i.e., $\xi^k\equiv0$, one can remove the $\ell_\infty$-norm constraint from subproblem \eqref{inexact-subprob} because now it has a nonempty optimal solution set. As will be shown by Theorem \ref{theorem1-nearly} later, the value of $\mu$ has no influence on the theoretical results as long as it is no less than $\|\overline{x}\|_\infty+\frac{\sqrt{r}}{m\sigma_{\!A}(r)}\|A^\top e\|_\infty$. Figure \ref{fig0} in Section \ref{sec5.1} also validates this conclusion. Thus, a large $\mu$ is enough in practical computation, so its choice cannot bring any trouble to the implementation of Algorithm \ref{Alg1}.

 \noindent
 {\bf (b)} Let $f_k(x)\!:=\!\lambda\sum_{i\in T^{k-1}} |x_{i} |+\sum_{i \in (T^{k-1})^c}\delta_{[-\mu, \mu]}(x_i)$ for $x\in\mathbb{R}^n$. According to the first-order optimality condition of subproblem \eqref{inexact-subprob}, $x^k$ is an optimal solution of \eqref{inexact-subprob} iff
 \[
  {\rm dist}_{\infty}\big(m^{-1}A^{\top}(b-\!Ax^k)+\lambda\xi^{k-1},\,\partial\!f_k(x^{k})\big)=0,
 \]
 where $\partial\!f_{k}(x^k)$ is the subdifferential of $f_{k}$ at $x^k$ and ${\rm dist}_{\infty}(\cdot,\partial\!f_k(x^{k}))$ is the distance from the closed convex set $\partial\!f_k(x^{k})$ induced by the $\ell_\infty$-norm. The above equality implies that
 \begin{equation}\label{inexact-rule}
   \lambda^{-1}{\rm dist}_{\infty}(m^{-1}A^{\top}(b-Ax^k),\,\partial\!f_k(x^{k}))\leq \|\xi^{k-1}\|_\infty\leq\varsigma_{k-1}.
 \end{equation}
 That is, $x^k$ is actually an inexact optimal solution to the following convex program
 \begin{equation}\label{exact-subproblem}
 \min_{x\in\mathbb{R}^n}\bigg\{\frac{1}{2m}\|Ax-b\|^2+\lambda\!\! \sum_{i\in T^{k-1}}\!|x_{i}|
 \ \ {\rm s.t.}\ \  \|x_{(T^{k-1})^c}\|_{\infty} \leq \mu\bigg\}.
 \end{equation}
 The error vector $\xi^{k-1}$ is introduced into each iteration of Algorithm \ref{Alg1} by considering that numerical computation always involves errors. In practical computation, there is no need to choose in advance the error vector sequence $\{\xi^{k}\}_{k\ge 0}$ because they are automatically generated during the iterations of solvers to problem \eqref{exact-subproblem}. Many efficient algorithms are proposed to solve \eqref{exact-subproblem} \citep[see, e.g.,][]{Beck-APG,SDFLasso,YuanCS}.

  \noindent
  {\bf(c)} When $\xi^k\equiv 0$, if $\max_{i \in T^{k-1}} |x_{i}^{k} |=0$, then the condition in \eqref{opti-cond} holds, which implies that $x^k$ is a stationary point of problem \eqref{CS}. This inspires us to adopt $\max_{i \in T^{k-1}} |x_{i}^{k} |\le\epsilon$ as the termination condition of Algorithm \ref{Alg1}, and the final iterate $x^{\overline{k}}$ of Algorithm \ref{Alg1} can be viewed as an approximate stationary point of \eqref{CS}. With such $x^{\overline{k}}$, one can achieve an exact one by solving $\min_{{\rm supp}(x)\subset (T^{\overline{k}-1})^c}\|Ax-b\|^2$ or equivalently $\min_{z\in\mathbb{R}^{|(T^{\overline{k}-1})^{c}|}}\|A_{(T^{\overline{k}-1})^{c}}z-b\|^2$. Since $|(T^{\overline{k}-1})^{c}|$ is small, its computation cost is cheap. That is, by imposing a cheap post-processing on the final iterate of Algorithm \ref{Alg1}, one can achieve a stationary point of \eqref{CS}.
 \end{remark}

 Before conducting theoretical analysis of Algorithm \ref{Alg1}, we give two interesting examples to show that Algorithm \ref{Alg1} with $\varsigma_k\equiv 0,\, \epsilon=0,\, \varrho\in[\frac{1}{2}, 1]$ and $\mu=+\infty$ can produce an oracle estimator of \eqref{LinearModel} within $r+1$ iterations. From the perspective of statistics, an oracle estimator is a good stationary point of problem \eqref{CS}. However, the popular LLA algorithm by \citet{Fan-FCP} and the multi-stage convex relaxation algorithm by \citet{ZT-Multi-b} fail to do it for $\lambda\in[\frac{1}{3}\|e\|_{\infty}, \frac{2}{6.7}]$. Thus, with $\lambda\in[\frac{1}{3}\|e\|_{\infty}, \frac{2}{6.7}]$, these two sequential convex relaxation methods, starting from the Lasso estimator, cannot yield a good estimator.
%-----------------------------------------------------------------------------------------
 \begin{example}\label{Exam1}
  The true $\overline{x}\in\mathbb{R}^4$, the design matrix $A\in\mathbb{R}^{3\times 4}$, and the vector $b$ are given by
   \begin{equation}\label{ExampAb}
  \overline{x}=\left(
   \begin{array}{c}
   	0 \\
   	0\\
   	2 \\
   	10 \\
   \end{array}
   \right),\ A=\left(
  	\begin{array}{cccc}
  		1 & -1 & 0  & 0  \\
  		1 & 0  & 1  & 0   \\
  		2 & 0  & 0  & 1  \\
  	\end{array}
  	\right)\ \ {\rm and}\ \
  	b=A\overline{x}+e=\left(
  	\begin{array}{c}
  		e_1\\
  		2+e_2 \\
  		10+e_3 \\
  	\end{array}\right),
  \end{equation}
  where $e\in \mathbb{R}^3$ is a noise vector satisfying $\|e\|_\infty\leq 0.1$. From $\overline{S}={\rm supp}(\overline{x})=\{3, 4\}$, the oracle estimator
  $x^{\rm o}\in{\rm arg}\min_{{\rm supp}(x)\subset \overline{S}}\|Ax-b\|^2$ has the following explicit form
  \[
  x^{\rm o}=(0,\ 0,\ b^\top A_{\overline{S}}(A_{\overline{S}}^\top A_{\overline{S}})^{-1})^\top=(0,\,0,\,2+e_2,\,10+e_3)^\top.
  \]
  It is easy to calculate that the optimal solution of the Lasso problem \eqref{Lone} with $0<\lambda\leq \frac{1}{3}$ is
  \[
  x^{\rm Lasso}(\lambda,e)=(2+e_2,\   2-e_1+ e_2 - 3\lambda, \  0, \ 6-2e_2+e_3-3\lambda)^\top.
  \]
  Obviously, the Lasso estimator is unique but bad because ${\rm supp}(x^{\rm Lasso}, 2)\neq {\rm supp}(\overline{x})$ for all $0<\lambda\leq \frac{1}{3}$ and $\lim_{\lambda\to 0^{+}, e\to 0}\frac{\|x^{\rm Lasso}(\lambda,e)-\overline{x}\|_1}{\max(\|e\|, \lambda)}=\infty$. For simplicity, in the following discussions, we assume that $e_1=e_2=e_3>0$ and pick any
  $\lambda\in[\frac{1}{3}\|e\|_\infty, \frac{2}{6.7}]$.

   \noindent
  {\bf(1) Algorithm \ref{Alg1} with $\varsigma_k\equiv 0,\, \epsilon=0,\,\mu=+\infty$ and $\varrho\in[\frac{1}{2},1]$.}	
  As $x^1=x^{\rm Lasso}$, we have $I^{1}=\{4\}$ and $T^{1}\!=T^{0}\backslash\{4\}\!=\{1, 2, 3\}$. Then, 
  the second iteration of Algorithm \ref{Alg1} solves
  \[
    x^{2}\in\mathop{\arg\min}_{x \in\mathbb{R}^4}\frac{1}{6}(x_1\!-\!x_2-e_1)^2+\frac{1}{6}(x_1+x_3-2-e_1)^2+\frac{1}{6}(2x_1+x_4-10\!-\!e_1)^2 + \lambda\sum_{i=1}^3|x_i|.
  \]
  An elementary calculation yields the unique optimal solution
  $x^2\!=\!(e_1,\,0,\,2-3\lambda,\,10-e_1)^\top$. Now $I^2=\{3\}$ and $T^2=T^1\backslash I^2=\{1, 2\}$, so the third iteration of Algorithm \ref{Alg1} solves
  \[
    x^{3}\in\mathop{\arg\min}_{x \in\mathbb{R}^4}\frac{1}{6}(x_1\!-\!x_2\!-\!e_1)^2+\frac{1}{6}(x_1+x_3-2-e_1)^2+\frac{1}{6}(2x_1+x_4\!-\!10\!-\!e_1)^2 + \lambda\sum_{i=1}^2|x_i|.
  \]
  We calculate that $x^{3}=(0,\,0,\,2+e_1,\,10+e_1)^\top=x^{\rm o}$. That is,  Algorithm \ref{Alg1}  with $\varsigma_k\equiv 0,\,\epsilon=0,\,\mu=+\infty$ and $\varrho\in[\frac{1}{2},1]$ finds the oracle estimator within $r+1$ iterations.

 \noindent
 {\bf(2) The LLA algorithm by \citet{Fan-FCP} with SCAD or MCP surrogate.}	The SCAD function \citep{Fan-SCAD} is defined by $\phi(t,\lambda):=\int_0^{|t|}\min(\lambda, \frac{(a\lambda-z)_+}{a-1})dz$ for $t\in\mathbb{R}$ and $\lambda>0$, where we choose $a=3.7$ as suggested in \citet{Fan-SCAD}. Its derivative w.r.t. the first variable at $|t|$ is $\lambda \mathbb{I}_{\{|t|\leq\lambda\}}+\frac{(a\lambda-|t|)_+}{a-1}\mathbb{I}_{\{|t|>\lambda\}}$. With $x^{1}=x^{\rm Lasso}$, the second iteration of the LLA-SCAD is solving the following convex minimization problem
  \begin{equation}\label{exp-SCAD}
   x^{2}\in\mathop{\arg\min}_{x\in\mathbb{R}^4}\ \frac{1}{6}(x_1-x_2-e_1)^2+\frac{1}{6}(x_1+x_3-2-e_1)^2+\frac{1}{6}(2x_1+x_4-10-e_1)^2+\lambda|x_3|.
  \end{equation}
  An elementary calculation yields the unique optimal solution $x^{2} = (2+e_1,\ 2,\ 0,\ 6-e_1)^\top\ne x^{\rm o}$. Using the similar arguments leads to $x^k=(2+e_1,\,2,\, 0,\,6-e_1)^\top\neq x^{\rm o}$ for all $k\geq 3$. Thus, the LLA-SCAD with $\lambda\in[\frac{1}{3}\|e\|_{\infty},\frac{2}{6.7}]$ fails to produce the oracle estimator. Similarly, the LLA-MCP with $1<a\leq 3$ and $\lambda\in[\frac{1}{3}\|e\|_{\infty},\frac{2}{6.7}]$ also fails to produce the oracle estimator.

  % The MCP function \citep{Zhang-MCP} is defined by $\phi(t,\lambda):=\int_0^{|t|}(\lambda-\frac{z}{a})_+dz$ for $t\in\mathbb{R}$ and $\lambda>0$, where we choose $a=3.0$ as suggested in \citet{Zhang-MCP} for sparsity recovery. With $x^{1}=x^{\rm Lasso}$, the second iteration of the LLA-MCP is solving the convex program \eqref{exp-SCAD}. We calculate that $x^k\!=\! (2+e_1, 2, 0, 6-e_1)^\top\!\neq\! x^{\rm o}$ for all $k\!\geq\! 2$. Thus, the LLA-MCP with $\lambda\in[\frac{1}{3}\|e\|_{\infty},\frac{2}{6.7}]$ fails to generate the oracle estimator.

  \medskip
 \noindent
  {\bf(3) The multi-stage convex relaxation (MSCR) algorithm by \citet{ZT-Multi-b} with the capped $\ell_1$ surrogate.}
 The capped $\ell_1$ function is defined by $\phi(t,\lambda):=\lambda\min(|t|,\varepsilon)$ with $\varepsilon>0$ for $t\in\mathbb{R}$ and $\lambda>0$. Its derivative w.r.t. the first variable at $|t|$ is $\lambda \mathbb{I}_{\{|t|\leq\varepsilon\}}$.
  This algorithm solves a sequence of truncated $\ell_1$-norm regularized problems \eqref{TwLone} with $T^{k-1}:=\{j\in[n]\,|\,|x_{j}^{k-1}|\leq \varepsilon_{k-1}\}$ at the $k$th iteration,
  where $\varepsilon_{k-1}>0$ is a constant depending on $x^{k-1}$.
  Clearly, when $\varepsilon_{1}\!<\! 1$ and $x^{1}\!=\!x^{\rm Lasso}$, we have $T^{1}\!=\!\{3\}$, so that $x^2$ is the same as in  \eqref{exp-SCAD}. Using the similar arguments, when $\varepsilon_{k}\!< \!1$, we have $x^k\!=\!(2+e_1, 2, 0, 6-e_1)^\top\!\neq\! x^{\rm o}$ for all $k\ge 3$. Thus, this algorithm with $\lambda\in[\frac{1}{3}\|e\|_{\infty},\frac{2}{6.7}]$ fails to yield the oracle estimator. In the sequel, this algorithm is abbreviated as the MSCR-cL1. 

 \end{example}
%-----------------------------------------------------------------------------
 \begin{example}\label{Exam1-SCRA-s}
 The true vector $\overline{x}\in\mathbb{R}^5$, the design matrix $A\in\mathbb{R}^{4\times 5}$ and the response vector $b\in\mathbb{R}^4$ are given by \eqref{ExampAb-NSP-s}. Note that ${\rm supp}(\overline{x})=\{1, 2\}$. It is not difficult to check that $x^{\rm o}\in\mathop{\arg\min}_{{\rm supp}(x)\subset \overline{S}}\|Ax-b\|^2=\{\overline{x}\}$. From Example \ref{Exam1-rSRNSP}, the solution set of the Lasso problem \eqref{Lone} with $0<\lambda<\frac{1}{4}$ is
  \(
  \mathcal {S}^{\rm Lasso}=\{(2+2t-8\lambda,\,10+t-8\lambda,\,t,\, t,\,t)^\top\ |\ 4\lambda-1\leq t\leq 0\}.
  \)
  Since the Lasso estimator is not unique, a high-quality solution is not necessarily obtained via the sequential convex relaxation methods mentioned in Example \ref{Exam1}. In fact, when $t$ is close to $4\lambda-1$, the support of $x^{\rm Lasso}$ is remarkably different from that of the true $\overline{x}$. Next we demonstrate that Algorithm \ref{Alg1} with $\varsigma_k\!\equiv\! 0,\,\epsilon=0,\,\mu\!=+\infty$ and $\varrho\in[0.2,1]$ can find the oracle estimator within three iterations. Indeed, with any Lasso estimator, we have $I^1=\{2\}$ and $T^1=\{1,3,4,5\}$, so the second iteration of
 Algorithm \ref{Alg1} solves
 \[
    x^{2}\in\mathop{\arg\min}_{x \in\mathbb{R}^5}\ \frac{1}{32}\big[(x_1-2x_3-2)^2+(x_1-2x_4-2)^2+(x_1-2x_5-2)^2+(x_1-2x_2+18)^2\big]+\lambda\sum_{i\neq 2}|x_i|.
 \]
 An elementary calculation yields the unique solution $x^2=(2-\frac{16}{3}\lambda, 10-\frac{8}{3}\lambda, 0,0,0)^\top$. Now $I^2=\{1\}$ and $T^2=T^1\backslash I^2=\{3, 4, 5\}$, so the third iteration of Algorithm \ref{Alg1} solves
  \[
     x^{3}\in\mathop{\arg\min}_{x \in\mathbb{R}^5}\ \frac{1}{32}\big[(x_1-2x_3-2)^2+(x_1-2x_4-2)^2+(x_1-2x_5-2)^2+(-x_1+2x_2-18)^2\big]+\lambda\sum_{i=3}^5|x_i|,
  \]
  which has the unique optimal solution $x^{3}\!=\!(2,\,10,\,0,\, 0,\, 0)^\top=x^{\rm o}$. This shows that Algorithm \ref{Alg1}  with $\varsigma_k\equiv 0, \epsilon=0,\mu=+\infty$ and $\varrho\in[0.2,1]$ yields the oracle estimator within $r+1$ iterations.
 \end{example}
%-----------------------------------------------------------------------------
 \begin{remark}\label{remark2-alg}
 Examples \ref{Exam1}-\ref{Exam1-SCRA-s} show that the performance of the LLA-SCAD or LLA-MCP  by \citet{Fan-FCP} and the MSCR-cL1 by \citet{ZT-Multi-b} depends much on the quality of the Lasso estimator (or the initial estimator). If the Lasso estimator is bad, say, ${\rm supp}(x^{\rm Lasso}, r)\neq \overline{S}$, it is very likely for them to fail in yielding a high-quality estimator for the zero-norm regularized problem \eqref{CS} because now the weight vector obtained from the Lasso estimator is not enough to penalize correctly the entries in the support of the true sparse regression vector. In contrast, the dependence of Algorithm \ref{Alg1} on the Lasso estimator is not so much, which allows the Lasso estimator to be non-unique even ${\rm supp}(x^{\rm Lasso},r)\neq \overline{S}$.
 \end{remark}
%----------------------------------------------------------------------------
\subsection{Theoretical certificates of Algorithm \ref{Alg1}}\label{sec4.2}

 First, we take a look at the properties of the sequence $\{I^k\}_{k\ge 1}$ generated by Algorithm \ref{Alg1}.
%--------------------------------------------------------------------------------------------------------
  \begin{proposition}\label{finite-Alg1}
  If Algorithm \ref{Alg1} does not stop at the $k$th iteration, then
  $\emptyset\neq I^{k}\subset{\rm supp}(x^{k})$, $I^{k}\cap I^{l}=\emptyset$ for any $l\in[k\!-\!1]$,
  and $T^{k}=\big(\bigcup_{j=1}^kI^{j}\big)^c$. Consequently, Algorithm \ref{Alg1} with $\epsilon=0$
 stops within a finite number of steps.
  \end{proposition}
 \begin{proof}
 As Algorithm \ref{Alg1} does not stop at the $k$th step, we have $\|x_{T^{k-1}}^{k}\|_\infty>\epsilon\geq0$. This, along with $\varrho\in(0, 1]$, implies that $\emptyset\neq I^{k}\subset{\rm supp}(x^{k})$. Pick any $l\in[k\!-\!1]$. From step 3, it follows that $I^{k}\subset T^{k-1}\subset T^{k-2}\subset\cdots\subset T^{l}=T^{l-1}\backslash I^{l}$, which implies that $I^{k}\cap I^{l}=\emptyset$. By the definition of $T^{k}$, we have $T^{k}=\big(\bigcup_{j=1}^kI^{j}\big)^c$. Combining $\emptyset\neq I^{k}\subset{\rm supp}(x^{k})$ with $I^{k}\cap I^{l}=\emptyset$ for any $l\in[k-1]$ and $T^{k}=T^{k-1}\backslash I^{k}$, we conclude that $T^k=\emptyset$ within a finite number of steps, i.e., Algorithm \ref{Alg1} with $\epsilon=0$ stops within a finite number of steps.
 \end{proof}

 Proposition \ref{finite-Alg1} implies that Algorithm \ref{Alg1} has the finite termination. Unless otherwise stated, we always let $\overline{k}$ denote the total number of iterations of Algorithm \ref{Alg1}. 
 
 The finite termination of Algorithm \ref{Alg1} implies the boundedness of the sequence $\{x^{k}\}_{k\in[\overline{k}]}$. In fact, for any $k\in[\overline{k}]$, it is easy to achieve an upper bound for $\|x^k_{T^{k-1}}\|_1$. Fix any $k\in[\overline{k}]$. From the feasibility of $x=\textbf{0}$ to subproblem \eqref{inexact-subprob}, it follows that
  \begin{align}\label{boundxk}
  \frac{1}{2m}\|Ax^k\!-\!b\|^2+\lambda\!\!\sum_{i\in T^{k-1}}\!\!|x^k_{i}|-\lambda\langle \xi^{k-1}, x^k\rangle\le\frac{1}{2m}\|b\|^2,
  \end{align}
  which implies that $\lambda(1-\|\xi^{k-1}\|_\infty)\|x^k_{T^{k-1}}\|_1-\lambda\|\xi^{k-1}\|_\infty\|x^k_{(T^{k-1})^c}\|_1\le\frac{1}{2m}\|b\|^2$. Then,
  \[
  \|x^k_{T^{k-1}}\|_1\leq\frac{1}{\lambda(1-\|\xi^{k-1}\|_\infty)}\Big(\frac{1}{2m}\|b\|^2+\lambda\|\xi^{k-1}\|_\infty\|x^k_{(T^{k-1})^c}\|_1\Big),
  \]
  which means that $\|x^k_{T^{k-1}}\|_1\leq\frac{1}{\lambda(1-\varsigma_0)}(\frac{\|b\|^2}{2m}+\lambda\varsigma_0n\mu)$ due to  $\|\xi^{k-1}\|_\infty\leq\varsigma_{k-1}\leq\varsigma_0<1$ and $\|x^k_{(T^{k-1})^c}\|_\infty\leq \mu$, so $\|x^k\|_1\leq\|x^k_{T^{k-1}}\|_1+\|x^k_{(T^{k-1})^c}\|_1\leq\frac{1}{\lambda(1-\varsigma_0)}(\frac{\|b\|^2}{2m}+\lambda\varsigma_0n\mu)+n\mu$. This bound is bad because it is related to the dimension $n$. The following lemma, which will be used to establish the main results of this section, implies a quantitative upper bound independent of $n$ for the iterate $x^{k}$ with $(T^{k-1})^c\subset\overline{S}$. It requires the notation
 \begin{align}{}\label{kappa12}
 \kappa:=\max_{\emptyset\neq S \subset\overline{S}}\max_{j\in S^c}\|(A_S^{\top}A_S)^{-1}A_S^{\top}A_{j}\|_1\ \ {\rm and}\ \
 \widehat{M}:=\frac{(5+\kappa)\sqrt{r}\|b\|}{2\sqrt{m}\sigma_{\!A}(r)}+\frac{5\|b\|^2(1+\kappa)}{8m\lambda}.
 \end{align}
%------------------------------------------------------------------------------------------------------
 \begin{lemma}\label{bound-xhat1}
 Suppose that $\sigma_{\!A}(r)>0$. Let $\kappa$ and $\widehat{M}$ be defined by \eqref{kappa12}. Consider the problem 
 \begin{equation}\label{equa-JTL1}
  \min_{x\in\mathbb{R}^n}\bigg\{\frac{1}{2m}\|Ax-b\|^2+\lambda\sum_{i\in J^c}|x_i|-\lambda\langle \xi, x\rangle
  \ \ {\rm s.t.}\ \ \|x_J\|_\infty\le\mu\bigg\}
 \end{equation}
 with $J\subset\overline{S},\,\mu>0,\,\|\xi\|_\infty\leq\!\frac{1}{5(1+\kappa)}$. Then any optimal solution $\widehat{x}$ of \eqref{equa-JTL1} satisfies $\|\widehat{x}\|_1\leq \widehat{M}$.
 \end{lemma}
 \begin{remark}
 {\bf(a)} When $A$ has a worse sparse singular value property, say, $\sigma_{\!A}(r)$ is close to $0$, the value of $\kappa$ will become larger, which restricts the error vector $\xi$ to be smaller and leads to a worse $\ell_1$-norm bound of $\widehat{x}$. When the columns of $A_{\overline{S}}$ and those of $A_{\overline{S}^{c}}$ are close to be orthogonal, the value of $\kappa$ is close to zero, and the upper bound of $\widehat{x}$ becomes better. 

 \noindent
 {\bf(b)} It is interesting to note that $\max_{j\in \overline{S}^c}|{\rm sign}(\overline{x}_{\overline{S}})^{\top}(A_{\overline{S}}^{\top}A_{\overline{S}})^{-1}A_{\overline{S}}^{\top}A_{j}|$ is a lower bound of $\kappa$. This term was ever used by \citet{Zhao-Lasso} to propose the irrepresentable condition
  \[
  \max_{j\in \overline{S}^c}|{\rm sign}(\overline{x}_{\overline{S}})^{\top}(A_{\overline{S}}^{\top}A_{\overline{S}})^{-1}A_{\overline{S}}^{\top}A_{j}|\leq 1-\eta\ {\rm\ for\ some }\ \eta> 0,
  \]
  which is almost necessary and sufficient for the Lasso to satisfy variable selection consistency, but here it has no help to improve the error bound of the iterates of Algorithm \ref{Alg1}.

 % When $\overline{x}$ is $r$ sparse, we have
 % \[
 % \kappa\geq \max_{j\in \overline{S}^c}\|(A_{\overline{S}}^{\top}A_{\overline{S}})^{-1}A_{\overline{S}}^{\top}A_{j}\|_1
 %  \geq\max_{j\in \overline{S}^c}|{\rm sign}(\overline{x}_{\overline{S}})^{\top}(A_{\overline{S}}^{\top}A_{\overline{S}})^{-1}A_{\overline{S}}^{\top}A_{j}|.
 % \]

 % Although Recall that the irrepresentable condition introduced by \citet{Zhao-Lasso} has the form

 %  which is almost necessary and sufficient for the Lasso to satisfy variable selection consistency property. Generally, as ${\rm sign}(\overline{x}_{\overline{S}})$ is unknown, the irrepresentable condition needs to hold for all possible
 %  signs, this leads to
 %  \[
 %   \max_{j\in \overline{S}^c}\|(A_{\overline{S}}^{\top}A_{\overline{S}})^{-1}A_{\overline{S}}^{\top}A_{j}\|_1\leq 1-\eta\ {\rm\ for\ some }\ \eta> 0.
 %  \]
 % It was shown by \citep[Theorem 2]{Wainwright-Lasso} that
 % the probability of variable selection consistency of the Lasso is bounded away from one when $\max_{j\in \overline{S}^c}\|(A_{\overline{S}}^{\top}A_{\overline{S}})^{-1}A_{\overline{S}}^{\top}A_{j}\|_1=1+\delta$ for some $\delta>0$.
 % However, we in this paper do not require $\kappa<1$, and allow the irrepresentable condition to be false. Indeed, for the matrix $A$ in Examples \ref{Exam1-rSRNSP} and \ref{Exam1} does not satisfy the irrepresentable condition.
 \end{remark}
 
 Now we are in a position to establish the theoretical certificates of Algorithm \ref{Alg1}. Specifically, under mild rRNSP or rSRNSP, we prove that $I^{k}\subset\!\overline{S}= {\rm supp}(\overline{x})$ for all $k\in[\overline{k}-1]$, and obtain an $\ell_1$-norm error bound of each iterate from the oracle estimator. Then, an oracle estimator of model \eqref{LinearModel} can be obtained by solving at most $r+1$ truncated $\ell_1$-norm regularized problems. These results (see Appendix \ref{app-4-2}-\ref{app-4-4} for their proofs) affirmatively answer which indices belong to $\overline{S}$ when one cannot achieve a high-quality solution of the zero-norm regularized problem \eqref{CS} or sparse linear regression problem by solving a single (truncated) $\ell_1$-norm regularized problem. As far as we know, when the Lasso estimator is bad, i.e., $\|x^{\rm Lasso}-\overline{x}\|_1\nleqslant O(r\lambda)$ and ${\rm supp}(x^{\rm Lasso}, r)\neq \overline{S}$, there is no work to study error bound and variable selection of each iterate of sequential convex relaxation methods. 
 
 We first apply Lemma \ref{bound-xhat1} to prove that under a suitable rRNSP, for any optimal solution of \eqref{equa-JTL1}, the indices of its entries with larger modulus will belong to $\overline{S}$. Note that \eqref{equa-JTL1} with $\xi=\xi^k$ and $J=T^{k-1}$ reduces to the $k$th subproblem of Algorithm \ref{Alg1}. Thus, for every iterate of Algorithm \ref{Alg1}, the indices of its entries with larger modulus will belong to $\overline{S}$. 
%----------------------------------------------------------------------------------------------
 \begin{proposition}\label{prop1-nearly}
  Suppose that $\sigma_{\!A}(r)>0$ and $\mu\ge M:=\|\overline{x}\|_\infty+\frac{\sqrt{r}}{m\sigma_{\!A}(r)}\|A^\top e\|_\infty$. Consider any optimal solution $\widehat{x}$ of problem \eqref{equa-JTL1} with $J\subsetneqq\overline{S},\,\|\xi\|_\infty\leq \frac{1-\rho}{5(1+\kappa)}$ for some $\rho\in(0,1)$, and $\lambda \geq \frac{2}{m(1-\rho)}\|A_{\overline{S}^c}^\top(E-A_{\overline{S}}(A_{\overline{S}}^\top A_{\overline{S}})^{-1}A_{\overline{S}}^\top) e\|_\infty$. Then, under the $(r\!-\!|J|,\eta)$-rRNSP($M,\rho,\nu$) of $A$ with $\nu>0$ and $\eta>\frac{18\widehat{M}\| \xi\|_\infty+5\lambda r\nu^2}{3(1-\rho)}$, it holds that 
 $\{j\in J^c\ |\ |\widehat{x}_j|\ge\eta\}\subset\overline{S}$ and
  \begin{align}
  \|\widehat{x}-x^{\rm o}\|_1&\leq (1\!+\!\kappa)\min\Big\{(10/3)\big[(r\!-\!|J|)M+1.8\widehat{M}\| \xi\|_\infty\big],\, (n-r)\eta\Big\}\nonumber\\
  &\quad\ +[\sigma_{\!A}(r)]^{-1}\sqrt{2\lambda r}\sqrt{(r\!-\!|J|)M+1.8\widehat{M}\| \xi\|_\infty}.\nonumber
 \end{align}
 \end{proposition}
 %------------------------------------------------------------------------------------
 \begin{remark}\label{remark1-nearly}
 {\bf (a)} If the entries $e_1,\ldots,e_m$ of $e$ are i.i.d. sub-Gaussian random variable with parameter $\sigma>0$, i.e., $\mathbb{E}_{e_i}(\exp(te_i))\leq \exp(\sigma^2t^2/2)$ for all $i\in[m]$ and $t\in \mathbb{R}$, then by \citep[][Lemma 3]{ZT-Multi-b} picking the regularization parameter $\lambda=O(\sigma\sqrt{\log n/m})$ can ensure that $\lambda \geq \frac{2}{m(1-\rho)}\|A_{\overline{S}^c}^\top(E-A_{\overline{S}}(A_{\overline{S}}^\top A_{\overline{S}})^{-1}A_{\overline{S}}^\top) e\|_\infty$ holds with a high probability.

 \noindent
 {\bf(b)} Notice that \eqref{equa-JTL1} with $J=\emptyset$ and $\xi=0$ reduces to the Lasso problem \eqref{Lone}. Then, under the conditions of Proposition \ref{prop1-nearly}, using  Proposition \ref{prop1-nearly} and Lemma \ref{Lemma-oracle} (i) leads to
 \begin{equation}\label{errorb-Lasso}
  \|x^{\rm Lasso}-\overline{x}\|_1\leq (1+\kappa)\min\Big\{\frac{10}{3}rM,\,(n-r)\eta\Big\}+\frac{rM }{\sigma_{\!A}(r)}+\frac{r \lambda}{\sigma_{\!A}(r)}+\frac{r\|A^\top e\|_\infty}{m\sigma_{\!A}(r)}.
 \end{equation}
 Since $\|A^\top e\|_\infty/m=O(\sigma\sqrt{\log n/m})$ by \citep[Lemma 5]{ZT-Multi-a}, the sum of the last two terms on the right hand side of \eqref{errorb-Lasso} is $O(r\sigma\sqrt{\log n/m})$, while the sum of the first two terms is $O(r\!+\!r^{3/2}\sigma\sqrt{\log n/m})$. Thus, the right hand side of \eqref{errorb-Lasso} is $O(r\!+\!r^{3/2}\sigma\sqrt{\log n/m})$. However, it was shown in \citep[][]{Bickel-Lasso,Negahban-Lasso} that under $\chi(3)>0$, $\|x^{\rm Lasso}-\overline{x}\|_1=O(r\sigma\sqrt{\log n/m})$. This implies that the condition $\chi(3)>0$ is very strong, which precisely inspires  us to introduce the weaker rRNSP and rSRNSP in Section 3. 
 \end{remark}

 Henceforth, we let $\beta(\lambda)\!:=\{\beta_0(\lambda),\beta_1(\lambda),\ldots,\beta_{r-1}(\lambda)\}$ with  $\beta_j(\lambda)$ for $j\in[r\!-\!1]_{+}$ defined by \eqref{seqbeta1}-\eqref{seqbeta0} with $\varsigma_0$ of Algorithm \ref{Alg1}, and write $M\!:=\|\overline{x}\|_\infty+\frac{\sqrt{r}}{m\sigma_{\!A}(r)}\|A^\top e\|_\infty$.
 %----------------------------------------------------------------------------------------------
 \begin{corollary}\label{Corollary-nearly-rho}
  Suppose $\sigma_{\!A}(r)>0$ and $\mu\ge M$. Consider any optimal solution $\widehat{x}$ of \eqref{equa-JTL1} with $J\subsetneqq\overline{S},\,\|\xi\|_\infty\le\varsigma_0\leq \frac{1-\rho}{5(1+\kappa)}$ for some $\rho\in(0,1)$, and $\lambda \geq \frac{2\|A_{\overline{S}^c}^\top(E-A_{\overline{S}}(A_{\overline{S}}^\top A_{\overline{S}})^{-1}A_{\overline{S}}^\top) e\|_\infty}{m(1-\rho)}$. If the matrix $A$ has the $(r\!-|J|,\varrho\beta_{|J|}(\lambda))$-rRNSP($M,\rho,\nu$) with $\nu>0$, 
   $\beta_{|J|}(\lambda)\!>\!\frac{18\widehat{M}\| \xi\|_\infty+5\lambda r\nu^2}{3\varrho(1-\rho)}$ and $\varrho\in(0, 1]$, then 
  $\{j\in J^c\,|\, |\widehat{x}_j|\ge\varrho\|\widehat{x}_{J^c}\|_\infty\}\subset\overline{S}$ and
 \begin{align}\label{pine-errorb1-result-rho}
  \|\widehat{x}-x^{\rm o}\|_1&\leq (1\!+\!\kappa)\min\Big\{(10/3)\big[(r\!-\!|J|)M+1.8\widehat{M}\| \xi\|_\infty\big],\,(n-r)\varrho\beta_{|J|}(\lambda)\Big\}\nonumber\\
  &\quad\ +[\sigma_{\!A}(r)]^{-1}\sqrt{2r\lambda}\sqrt{(r\!-\!|J|)M+1.8\widehat{M}\| \xi\|_\infty}.
 \end{align}
 \end{corollary}
 \begin{proof}
  By using the feasibility of $\overline{x}$ to problem \eqref{equa-JTL1} and $b=A\overline{x}+e$, we have that
  \begin{align}\label{boundx0x}
  \frac{1}{2m}\|A\widehat{x}-b\|^2+\lambda\sum_{i\in J^c}|\widehat{x}_i|-\lambda\langle \xi, \widehat{x}\rangle
  &\le\frac{1}{2m}\|A\overline{x}-b\|^2+\lambda\sum_{i\in J^c}|\overline{x}_i|-\lambda\langle \xi, \overline{x}\rangle\nonumber\\
  &\leq {\|e\|^2}/{(2m)}+\lambda(1+\|\xi\|_{\infty})\|\overline{x}\|_1
  \end{align}
  which implies that $\|A\widehat{x}-b\|\leq \sqrt{\|e\|^2+2m\lambda(1\!+\!\|\xi\|_{\infty})\|\overline{x}\|_1}$. From the definition of $\beta_{|J|}(\lambda)$ in \eqref{seqbeta1}-\eqref{seqbeta0} and $\varsigma_0\leq\frac{1-\rho}{5(1+\kappa)}$, we obtain that
  $\|\widehat{x}_{J^c}\|_{\infty}\!\geq\! |\widehat{x}|^\downarrow_{|J|+1}\!\ge\!\beta_{|J|}(\lambda)>\frac{18\widehat{M}\| \xi\|_\infty+5\lambda r\nu^2}{3\varrho(1-\rho)}>0$, which implies that 
  $\{j\in J^c\,|\,|\widehat{x}_{j}|\!\geq\! \varrho\|\widehat{x}_{J^c}\|_{\infty}\}\!\subset\!\{j\in J^c\,|\,|\widehat{x}_{j}|\ge\varrho\beta_{|J|}(\lambda)\}$. Using Proposition \ref{prop1-nearly} with $\eta=\varrho\beta_{|J|}(\lambda)$ leads to
  $\{j\in J^c\,|\, |\widehat{x}_{j}|\!\geq\! \varrho\|\widehat{x}_{J^c}\|_{\infty}\}\subset\overline{S}$
  and inequality \eqref{pine-errorb1-result-rho}.
 \end{proof}

 By Proposition \ref{bound-feasol}, when $\sigma_{\!A}(2r\!-\!1)>0$ and the $r$th largest (in modulus) entry of the true vector is not too small, the condition on $\beta_{|J|}(\lambda)$ in Corollary \ref{Corollary-nearly-rho} will hold for small error vector $\xi\in \mathbb{R}^n$. When the $(r\!-\!|J|, \varrho\beta_{|J|}(\lambda))$-rRNSP($M,\rho,\nu$) of the matrix $A$ in Corollary \ref{Corollary-nearly-rho} is strengthened to be the $(r\!-\!|J|, \beta_{r-1}(\lambda))$-rRNSP($M,\rho,\nu$),
 we have the following support recovery result for the truncated $\ell_1$-norm regularized problem \eqref{equa-JTL1}.
 \begin{corollary}\label{corollary1-nearly}
 Suppose $\sigma_{\!A}(r)>0$ and $\mu\ge M$. Consider any optimal solution $\widehat{x}$ of \eqref{equa-JTL1} with $J\subsetneqq\overline{S},\,\|\xi\|_\infty\le\varsigma_0\leq \frac{1-\rho}{5(1+\kappa)}$ for some $\rho\in(0,1)$, and $\lambda \geq \frac{2\|A_{\overline{S}^c}^\top(E-A_{\overline{S}}(A_{\overline{S}}^\top A_{\overline{S}})^{-1}A_{\overline{S}}^\top) e\|_\infty}{m(1-\rho)}$. If $A$ satisfies the $(r\!-|J|,\beta_{r-1}(\lambda))$-rRNSP($M,\rho,\nu$) with $\nu>0$ and $
  \beta_{r-1}(\lambda)>\frac{18\widehat{M}\| \xi\|_\infty+5\lambda r\nu^2}{3(1-\rho)}$, then
  ${\rm supp}(\widehat{x}, r)=\overline{S}$ and inequality \eqref{pine-errorb1-result-rho} holds with $\varrho=1$.
  \end{corollary}
 \begin{proof}
  From \eqref{boundx0x}, we have $\|A\widehat{x}-b\|\leq \sqrt{\|e\|^2+2m\lambda\|\overline{x}\|_1}$. From the definition of $\beta_{r-1}(\lambda)$ in \eqref{seqbeta1}, it follows that $|\widehat{x}|_r^{\downarrow}\geq\beta_{r-1}(\lambda)>\frac{18\widehat{M}\| \xi\|_\infty+5\lambda r\nu^2}{3\varrho(1-\rho)}>0$ and $\|\widehat{x}\|_0\geq r$,
  which implies that ${\rm supp}(\widehat{x}, r)\subset{\rm supp}(\widehat{x})$ and $|{\rm supp}(\widehat{x}, r) |=r$. From $|\widehat{x}_{i}|\geq |\widehat{x}|^\downarrow_{r}\geq \beta_{r-1}(\lambda)$ for all $i\in{\rm supp}(\widehat{x}, r)$,
  we have ${\rm supp}(\widehat{x}, r)\subset J\cup \{j\in J^c\ |\ |\widehat{x}_j|\geq\beta_{r-1}(\lambda)\}$.
  Since the matrix $A$ satisfies the $(r\!-|J|,\beta_{r-1}(\lambda))$-rRNSP($M,\rho,\nu$), using Proposition \ref{prop1-nearly} with $\eta=\beta_{r-1}(\lambda)$ leads to \eqref{pine-errorb1-result-rho} for $\varrho=1$ and the inclusion
  $\{j\in J^c\ |\ |\widehat{x}_j|\ge\beta_{r-1}(\lambda)\}\subset\overline{S}$. The inclusion means that ${\rm supp}(\widehat{x}, r)\subset\overline{S}$.
  Along with $|{\rm supp}(\widehat{x}, r) |=|\overline{S}|$, it holds that ${\rm supp}(\widehat{x}, r)=\overline{S}$.
  \end{proof}

 Now we are ready to prove that Algorithm \ref{Alg1} can recover the index set $\overline{S}$ within at most $r$ steps,
 and derive the $\ell_1$-norm error bound of each iterate under the following assumption. As will be discussed at the end of this section, the conditions on $\sigma_{\!A}(2r-1)$ and $\mu$ in this assumption are rather weak, while the condition on $\beta_{r-1}(\lambda)$ is a little restricted but may hold when the $r$th largest (in modulus) entry of the true vector $\overline{x}$ is not too small.
 %-----------------------------------------------------------------------------------------
 \begin{Assumption}\label{ass1}
  $\sigma_{\!A}(2r\!-\!1)>0,\,\mu\ge M,\,\lambda \geq \frac{2}{m(1-\gamma)}\|A_{\overline{S}^c}^\top(E-A_{\overline{S}}(A_{\overline{S}}^\top A_{\overline{S}})^{-1}A_{\overline{S}}^\top) e\|_\infty$ and $\beta_{r-1}(\lambda)>\frac{18\widehat{M}\varsigma_0+5r\lambda\tau^2}{3\varrho(1-\gamma)}$ for some $\varrho\in(0, 1]$, $\gamma\in(0,1),\,\tau>0$ and $\varsigma_0\leq\frac{1-\gamma}{5(1+\kappa)}$.
 \end{Assumption}
%---------------------------------------------------------------------------------------------------
 \begin{theorem}\label{theorem1-nearly}
 Suppose that Assumption \ref{ass1} holds and $A$ has the $(r, \varrho\beta(\lambda))$-rSRNSP($M,\gamma,\tau$) with $\beta(\lambda)\!:=\!\{\beta_0(\lambda),\beta_1(\lambda),\ldots,\beta_{r-1}(\lambda)\}$. Let $\{x^k\}_{k\in[\overline{k}]}$ and $\{I^k\}_{k\in[\overline{k}-1]}$ be generated by Algorithm \ref{Alg1} with $\epsilon\in[6\widehat{M}\varsigma_{0},\beta_{r-1}(\lambda))$. Then, $\overline{k}\in\{2,\ldots,r\!+\!1\}$ is the smallest positive integer such that $\bigcup_{k=1}^{\overline{k}-1}I^k=(T^{\overline{k}-1})^c=\overline{S}={\rm supp}(x^{\overline{k}}, r)$, and for all $1\leq k\leq\overline{k}\!-\!1$,
 \begin{align}\label{result421-noise}
  \|x^{k}\!-\!x^{\rm o}\|_1&\le(10/3)(1+\kappa)\big[(r\!-\!|(T^{k-1})^c|)M+1.8\widehat{M}\varsigma_{k-1}\big]\nonumber\\
   &\quad+[\sigma_{\!A}(r)]^{-1}\sqrt{2r\lambda}\sqrt{(r\!-\!|(T^{k-1})^c|)M+1.8\widehat{M}\varsigma_{k-1}},\\
  \|x^{\overline{k}}-x^{\rm o}\|_1 &\le 6(1+\kappa)\widehat{M}\varsigma_{\overline{k}-1}+[\sigma_A(r)]^{-1}\sqrt{3.6\lambda r\widehat{M}\varsigma_{\overline{k}-1}}. 
 \label{result422-noise}
 \end{align}
 \end{theorem}
 \begin{remark}
 {\bf(a)} From the proof of Theorem \ref{theorem1-nearly}, $\overline{k}$ is the first positive integer such that $\bigcup_{j=1}^{\overline{k}-1}I^{j}=\overline{S}$.
 In the worst case, $|I^{j}|=1$ for each $j\in[\overline{k}]$ and $\overline{k}=r+1$, 
 so the total $r$ (resp. $r+1$) truncated $\ell_1$-norm regularized minimization problems are required to recover $\overline{S}$ (resp. to achieve a high-quality solution to the sparse linear regression problem), which will be expensive when $r$ is large.
 However, to the best of our knowledge, this is the first sequential convex relaxation algorithm to recover the support of the true sparse regression vector and obtain a high-quality solution (the oracle estimator $x^{\rm o}$ when $\varsigma_k\equiv0$) to the sparse linear regression problem
 under a weaker NSP condition within a specific number of steps, even when the Lasso estimator is bad. Indeed, similar to Examples \ref{Exam1} and \ref{Exam1-SCRA-s}, our theoretical results allow ${\rm supp}(x^{\rm Lasso},r)\neq \overline{S}$ and
  $\|x^{\rm Lasso}-\overline{x}\|_1\nleqslant O(r\lambda)$,
  which is the remarkable difference between our results and the existing ones in \citep[e.g.,][]{Fan-FCP,ZT-Multi-a}.

  \noindent
  {\bf(b)} Combining the error bounds in Theorem \ref{theorem1-nearly} with $\|\overline{x}-x^{\rm o}\|_1\leq \frac{r\|A^\top e\|_\infty}{m\sigma_{\!A}(r)}$ by Lemma \ref{Lemma-oracle}, one can achieve the $\ell_1$-norm error bounds of the iterates from the true sparse vector. 
 \end{remark}

 The error bound in \eqref{result421-noise} implies that if $\varsigma_{\overline{k}-\!1}=0$, i.e., the $\overline{k}$th subproblem is exactly solved, the final output of Algorithm \ref{Alg1} is necessarily an oracle estimator of model \eqref{LinearModel} and its variable selection consistency can be established. That is, the following conclusion holds.
%------------------------------------------------------------------------------------
 \begin{theorem}\label{oracle-r}
 Suppose that $\sigma_{\!A}(r)>0,\,\mu\geq\|\overline{x}\|_\infty+\|(A_{\overline{S}}^{\top} A_{\overline{S}})^{-1}A_{\overline{S}}^{\top}e\|_\infty$ and $\lambda> \frac{1}{m}\|A_{\overline{S}^c}^{\top}(E-A_{\overline{S}}(A_{\overline{S}}^\top A_{\overline{S}})^{-1}A_{\overline{S}}^\top) e\|_\infty$. If $T^{\overline{k}-1}=\overline{S}^c$ and $x^{\overline{k}}$ is generated with $\xi^{\overline{k}}=0$, then $x^{\overline{k}}$ is the unique optimal solution of
 $\min_{{\rm supp}(x)\subset \overline{S}}\|Ax-b\|^2$ and is a  local optimal solution of problem \eqref{CS}. If in addition $\min_{i\in \overline{S}}|\overline{x}_i|>\|(A_{\overline{S}}^{\top}A_{\overline{S}})^{-1}A_{\overline{S}}^{\top}e\|_\infty$,
 then ${\rm sign}(x^{\overline{k}})={\rm sign}(\overline{x})$.
 \end{theorem}
 \begin{proof}
 From Lemma \ref{Lemma-oracle}, $x^{\rm o}$ is the unique solution of $\min_{{\rm supp}(x)\subset \overline{S}}\|Ax-b\|^2$, and $x^{\rm o}=((A_{\overline{S}}^{\top}A_{\overline{S}})^{-1}A_{\overline{S}}^{\top}b;\ \textbf{0})
 =(\overline{x}_{\overline{S}}+(A_{\overline{S}}^{\top}A_{\overline{S}})^{-1}A_{\overline{S}}^{\top}e;\ \textbf{0})$.
 Thus, if $\min_{i\in \overline{S}}|\overline{x}_i|>\|(A_{\overline{S}}^{\top}A_{\overline{S}})^{-1}A_{\overline{S}}^{\top}e\|_\infty$,
 then ${\rm sign}(x^{\rm o})={\rm sign}(\overline{x})$.
  Using the definition of $x^{\overline{k}}$, $T^{\overline{k}-1}=\overline{S}^c$, and the feasibility of $x^{\rm o}$ to subproblem \eqref{inexact-subprob} for $k=\overline{k}$ leads to
 \begin{align}
 &\frac{1}{2m}\|Ax^{\overline{k}}\!-\!b\|^2+\lambda\sum_{i\in \overline{S}^c}|x^{\overline{k}}_{i}|\leq \frac{1}{2m}\|Ax^{\rm o}\!-\!b\|^2\nonumber\\
\Leftrightarrow\ &\frac{1}{2m}\|A(x^{\overline{k}}\!-\!x^{\rm o})\|^2-\frac{1}{m}\langle A(x^{\overline{k}}\!-\!x^{\rm o}), b\!-\!Ax^{\rm o}\rangle+\lambda\sum_{i\in \overline{S}^c}|x^{\overline{k}}_{i}|\leq 0\nonumber\\
\Leftrightarrow\ &-\frac{1}{m}\langle A_{\overline{S}^c}x^{\overline{k}}_{\overline{S}^c}, (E\!-\!A_{\overline{S}}(A_{\overline{S}}^\top A_{\overline{S}})^{-1}A_{\overline{S}}^\top) e\rangle+\lambda\sum_{i\in \overline{S}^c}|x^{\overline{k}}_{i}|\leq -\frac{1}{2m}\|A(x^{\overline{k}}\!-\!x^{\rm o})\|^2\nonumber\\
\Rightarrow\ &-\frac{1}{m}\|A_{\overline{S}^c}^T(E\!-\!A_{\overline{S}}(A_{\overline{S}}^\top A_{\overline{S}})^{-1}A_{\overline{S}}^\top) e\|_\infty\|x^{\overline{k}}_{\overline{S}^c}\|_1
+\lambda\sum_{i\in \overline{S}^c}|x^{\overline{k}}_{i}|\leq -\frac{1}{2m}\|A(x^{\overline{k}}\!-\!x^{\rm o})\|^2\leq 0,\nonumber
\end{align}
where the second equivalence is due to $b-Ax^{\rm o}=(E\!-\!A_{\overline{S}}(A_{\overline{S}}^\top A_{\overline{S}})^{-1}A_{\overline{S}}^\top) e$ and $x^{\rm o}_{\overline{S}^c}=\textbf{0}$ by Lemma \ref{Lemma-oracle}. Note that 
$\lambda> \frac{1}{m}\|A_{\overline{S}^c}^{\top}(E-A_{\overline{S}}(A_{\overline{S}}^\top A_{\overline{S}})^{-1}A_{\overline{S}}^\top) e\|_\infty$.
The above inequality implies that $x^{\overline{k}}_{\overline{S}^c}=0$ and $\|A(x^{\overline{k}}\!-\!x^{\rm o})\|=0$. Now, by using the condition $\sigma_A(r)>0$, we have $0=\|A(x^{\overline{k}}\!-\!x^{\rm o})\|=\|A_{\overline{S}}(x^{\overline{k}}_{\overline{S}}\!-\!x^{\rm o}_{\overline{S}})\|\geq \sigma_A(r)\|x^{\overline{k}}_{\overline{S}}\!-\!x^{\rm o}_{\overline{S}}\|$. 
Thus, we have $x^{\overline{k}}= x^{\rm o}$. By Remark \ref{Remark-Alg1} (c), $x^{\overline{k}}$ is a stationary point of problem \eqref{CS}. Invoking \citep[][Remark 3.3]{PanLiu22}, we conclude that $x^{\overline{k}}$ is a (strong) local optimal solution of \eqref{CS}. 
\end{proof}

 Next we improve the conclusion of Theorem \ref{theorem1-nearly} by proving that under a little stronger rSRNSP, a few truncated $\ell_1$-norm regularized problems are enough to identify $\overline{S}$ in many cases and a better $\ell_1$-norm error bound can be achieved. 
 % , which a high-quality solution to the sparse linear regression problem can be achieved under a little stronger rSRNSP. 
  % By comparing with the $\ell_1$-norm error bounds in Theorem \ref{theorem1-nearly}, we see that the error bound in Theorem \ref{theorem2-nearly} is smaller epspecially when $\widetilde{r}$ is close to $r-1$. 
%----------------------------------------------------------------------------------------
 \begin{theorem}\label{theorem2-nearly}
Suppose that Assumption \ref{ass1} holds and $A$ has the $(r, \varrho\beta'(\lambda))$-rSRNSP($M,\gamma,\tau$) with $\beta'(\lambda)\!:=\big\{\beta_0(\lambda),\beta_1(\lambda),\ldots, \beta_{\widetilde{r}-1}(\lambda),\underbrace{\beta_{r-1}(\lambda),\ldots,\beta_{r-1}(\lambda)}_{r-\widetilde{r}}\big\}$ for some $\widetilde{r}\in[r\!-\!1]$. Let $\{x^k\}_{k\in[\overline{k}]}$ and $\{I^k\}_{k\in[\overline{k}-1]}$ be generated by Algorithm \ref{Alg1} with $\epsilon\in[6\widehat{M}\varsigma_{0},\beta_{r-1}(\lambda))$. Then, by letting $\widehat{k}$ be the first integer such that $|\bigcup_{j=1}^{\widehat{k}-1}I^{j}|\ge\widetilde{r}$, it holds that ${\rm supp}(x^{\widehat{k}},r)\!=\!\overline{S}$ and
 \begin{equation}\label{result422-noise}
 \|x^{\widehat{k}}-x^{\rm o}\|_1\leq \frac{10(1\!+\!\kappa)}{3}\big[(r\!-\!\widetilde{r})M+1.8\widehat{M}\varsigma_{\widehat{k}-1}\big] +\frac{\sqrt{2r\lambda}}{\sigma_{\!A}(r)}\sqrt{(r\!-\!\widetilde{r})M+1.8\widehat{M}\varsigma_{\widehat{k}-1}}.
 \end{equation}
 If $\beta'(\lambda)$ is replaced with the array $\big\{\beta_0(\lambda),\beta_1(\lambda),\ldots, \beta_{\widetilde{r}-1}(\lambda),\underbrace{0,\ldots,0}_{r-\widetilde{r}}\big\}$ for some $\widetilde{r}\in[r\!-\!1]$, then under the same conditions, it holds that ${\rm supp}(x^{\widehat{k}},r)\!=\!\overline{S}$ and
 \begin{equation}
 \|x^{\widehat{k}}\!-\!x^{\rm o}\|_1\le\frac{(1\!+\!\kappa)(18\widehat{M} \varsigma_{\widehat{k}-1}+5\lambda r\tau^2)}{3(1-\gamma)} +\frac{\sqrt{2r\lambda}}{\sigma_{\!A}(r)}\sqrt{(r\!-\!\widetilde{r})M+1.8\widehat{M}\varsigma_{\widehat{k}-1}}.\nonumber
 \end{equation}
 \end{theorem}

 To close this section, we take a closer look at the conditions in Assumption \ref{ass1}. First, we notice that the condition $\sigma_{\!A}(2r\!-\!1)>0$ is weaker than $\kappa_{-}(\frac{2r}{1-\rho})>0$ for some $\rho\in(0,1)$ with $\kappa_{-}(s)\!:=\min_{\|d\|_0\leq s, \|d\|=1}\frac{1}{m}\|Ad\|^2$, because $\sigma_{\!A}(2r\!-\!1)= \sqrt{\kappa_{-}(2r\!-\!1)}\geq \sqrt{\kappa_{-}(\frac{2r}{1-\rho})}$. The condition $\kappa_{-}(\frac{2r}{1-\rho})>0$ was used in \cite[Theorem 4]{Zhang-L0} to achieve the model selection consistency of the zero-norm regularized problem \eqref{CS}. In fact, the condition $\kappa_{-}(s)>0$ with $s\geq 3r$ was required in \citep[][Theorem 1]{ZT-Multi-b} to achieve exact recovery of the support set $\overline{S}$ for the MSCR-cL1. We claim that the condition $\sigma_{\!A}(2r\!-\!1)>0$ is also weaker than the restricted eigenvalue condition $\chi(3)>0$ with $\chi$ defined in \eqref{REC-ineq} \citep[see,][]{Bickel-Lasso,Negahban-Lasso}. Indeed, for all nonzero $d\in \mathbb{R}^n$ with $\|d\|_0\leq 2r-1$, by choosing $S={\rm supp}(d, r)$, we have $\|d_{S^c}\|_1\leq \|d_S\|_1$, which implies that $\sigma_{\!A}(2r\!-\!1)\geq \chi(1)\geq \chi(3)$. Second, for the condition $\mu\ge M$, as shown by Theorem \ref{theorem1-nearly} and Figure \ref{fig0} in Section \ref{sec5.2}, the value of $\mu$ has no influence on the theoretical results as long as $\mu$ is suitably larger than $M$, so its choice in practical computation is a piece of cake. Third, the condition $\beta_{r-1}(\lambda)>\frac{18\widehat{M}\varsigma_0+5r\lambda\tau^2}{3\varrho(1-\gamma)}$ for some $\varrho\in(0,1]$, $\gamma\in(0,1),\,\tau>0$ and $\varsigma_0\leq\frac{1-\gamma}{5(1+\kappa)}$ essentially requires smaller $\varsigma_0$ and $\tau$. The former means a high accuracy for the solution of the Lasso problem, while the latter requires $A$ to have better NSP. This restriction is a little stronger, but by Proposition \ref{bound-feasol} it will hold if the $r$th largest (in modulus) entry of the true vector $\overline{x}$ is not too small and $\sigma_{\!A}(2r\!-\!1)>0$.
%--------------------------------------------------------------------------
 \section{Numerical experiments}\label{sec5}

 In this section we validate the theoretical results in Section \ref{sec4.2} by applying Algorithm \ref{Alg1} to solve some synthetic and real-data examples, and compare the performance of Algorithm \ref{Alg1} with that of the LLA-SCAD by \citet{Fan-FCP}, the MSCR-cL1 by \citet{ZT-Multi-b} and the DCA-TrL1 by \citet{TL1}. Among others, every iteration of the LLA-SCAD and MSCR-cL1 needs to solve a weighted $\ell_1$-norm minimization problem of the form
 \begin{equation}\label{sprob-LLA}
  \min_{x,z\in\mathbb{R}^n}\Big\{\frac{1}{2m}\|z\|^2+\lambda\|w^k\circ x\|_1\ \ {\rm s.t.}\ \ Ax-b=z\Big\},
 \end{equation}
 and the semismooth Newton augmented Lagrangian (SSNAL) algorithm in Section \ref{sec5.1} (see \citet{SDFLasso}) can be used to solve it. Every iteration of the DCA-TrL1 needs to solve 
 \begin{equation*}
  \min_{x\in\mathbb{R}^n}\Big\{\frac{1}{2m}\|Ax-b\|^2+c\|x\|^2+(1+a^{-1})\lambda\|x\|_1-\langle x,v^k\rangle\Big\},
 \end{equation*}
 where $a>0$ and $c>0$ are the constants, and $v^k\in\mathbb{R}^n$ is a vector. By following Section \ref{sec5.1}, one can develop the SSNAL algorithm for computing its equivalent reformulation 
 \begin{equation}\label{sprob-TL1}
  \min_{x,s\in\mathbb{R}^n,z\in\mathbb{R}^m}\Big\{\frac{1}{2m}\|z\|^2+c\|x\|^2+(1+a^{-1})\lambda\|s\|_1-\langle x,v^k\rangle\ \ {\rm s.t.}\ \ Ax-b=z,\,x-s=0\Big\}.
 \end{equation}
%------------------------------------------------------------------------------------
 \subsection{SSNAL for solving subproblems}\label{sec5.1}
  
 Let $T\subset[n]$ be an index set and $\mathbb{B}_{\gamma}\!:=\{z\in\mathbb{R}^{|T^c|}\,|\,\|z\|_{\infty}\le\gamma\}$ for $\gamma>0$. The main computation cost of Algorithm \ref{Alg1} in each iteration is to solve inexactly the problem
 \begin{equation}\label{Esubprob}
  \min_{x\in\mathbb{R}^n,z\in\mathbb{R}^m}\Big\{\frac{1}{2m}\|z\|^2+\lambda\|x_{T}\|_1+\delta_{\mathbb{B}_{\mu}}(x_{T^{c}})\ \ {\rm s.t.}\ \ Ax-b=z\Big\}.
 \end{equation}
 Inspired by \citet{SDFLasso}'s work, we develop an SSNAL algorithm for solving \eqref{Esubprob}. Let
 $f(x)\!:=\lambda\|x_{T}\|_1+\delta_{\mathbb{B}_{\mu}}(x_{T^{c}})$ for $x\in\mathbb{R}^n$. A simple calculation results in the dual of \eqref{Esubprob} as 
 \begin{equation}\label{Edprob1}
  \min_{\zeta\in\mathbb{R}^m}f^*(A^{\top}\zeta)-b^{\top}\zeta+\frac{m}{2}\|\zeta\|^2,
 \end{equation}
 where $\mathbb{R}^n\ni z\mapsto f^*(z):=\delta_{\mathbb{B}_{\lambda}}(z_{T})+\mu\|z_{T^c}\|_1$ is the conjugate of $f$. For a given $\sigma>0$, the augmented Lagrangian function of \eqref{Edprob1} takes the following form
 \[
  L_{\sigma}(\zeta,u;x):=-b^{\top}\zeta+\frac{m}{2}\|\zeta\|^2+f^*(u)+\langle x,A^{\top}\zeta-u\rangle+\frac{\sigma}{2}\|A^{\top}\zeta-u\|^2.
 \]
 The basic iterations of the augmented Lagrangian method (ALM) for solving \eqref{Edprob1} are 
 \begin{subnumcases}{}\label{alm-equa1}(\zeta^{l+1},u^{l+1})=\mathop{\arg\min}_{\zeta\in\mathbb{R}^m,u\in\mathbb{R}^n}L_{\sigma_l}(\zeta,u;x^l),\\
 x^{l+1}=x^l+\sigma_l(A^{\top}\zeta^{l+1}-u^{l+1}),\\
 \sigma_{l}\uparrow\sigma_{l+1}\le\infty.
 \end{subnumcases}
 Next we take a closer look at the solution of \eqref{alm-equa1}. An elementary calculation leads to
 \begin{align*}
 \zeta^{l+1}&=\mathop{\arg\min}_{\zeta\in\mathbb{R}^m}\Phi(\zeta)
  \!:=-b^{\top}\zeta+e_{\sigma_l^{-1}}f^*(A^{\top}\zeta+\sigma_l^{-1}x^l)+\frac{m}{2}\|\zeta\|^2,\\ u^{l+1}&\!=\mathcal{P}_{\!\sigma_l^{-1}}f^*\big(A^{\top}\zeta^{l+1}+\sigma_l^{-1}x^l\big),
 \end{align*}
 where $e_{\sigma_l^{-1}}f^*$ and $\mathcal{P}_{\!\sigma_l^{-1}}f^*$ are the Moreau envelope and  proximal mapping of $f^*$ associated with $\sigma_l^{-1}$. Since $\Phi$ is a smooth and strongly convex function and its gradient is Lipschitz continuous, seeking $\zeta^{l+1}$ is equivalent to finding the unique root to the nonsmooth system
 \begin{equation}\label{SNCG}
  0=\nabla\Phi(\zeta):=m\zeta+\sigma_lA\big[A^{\top}\zeta+\sigma_l^{-1}x^l-\mathcal{P}_{\!\sigma_l^{-1}}f^*\big(A^{\top}\zeta+\sigma_l^{-1}x^l)\big]-b,
 \end{equation} which is strongly semismooth by \citep[][Propositions 7.4.4 \& 7.4.5]{Pang03}. Since $\mathcal{P}_{\!\sigma_l^{-1}}f^*$ is globally Lipschitz, by \citep[][Definition 2.1]{Hiriart84} the generalized Hessian matrix of $\Phi$ at $\zeta$ has the following form
 \[
  \widehat{\partial}^2\Phi(\zeta):=mE+\sigma_lA\big[E-\partial_{C}\mathcal{P}_{\!\sigma_l^{-1}}f^*(\mathcal{B}(\zeta))\big]A^{\top}\ \ {\rm with}\ \ \mathcal{B}(\zeta)\!:=A^{\top}\zeta+\sigma_l^{-1}x^l,
 \]  
 where $\partial_{C}\mathcal{P}_{\!\sigma_l^{-1}}f^*(\mathcal{B}(\zeta))$ is the Clarke subdifferential of the Lipschitz continuous mapping $\mathcal{P}_{\!\sigma_l^{-1}}f^*$ at $\mathcal{B}(\zeta)$. By the expression of $\mathcal{P}_{\!\sigma_l^{-1}}f^*$, every element of $\partial_{C}\mathcal{P}_{\!\sigma_l^{-1}}f^*(\mathcal{B}(\zeta))$ is a diagonal matrix, say, $D=(d_1,\ldots,d_i)$ with the $i$th diagonal entry $d_i$ given by 
 \begin{align*}
  d_i=\left\{\begin{array}{cl}
  1 & {\rm if}\ |\mathcal{B}_i(\zeta)|\le \lambda\\
  0 & {\rm otherwise}
  \end{array}\right.\ {\rm for}\ i\in T\ \ {\rm and}\ \  d_i=\!\left\{\begin{array}{cl}
  0 & {\rm if}\ |\mathcal{B}_i(\zeta)|\le \mu\sigma_l^{-1}\\
  1 & {\rm otherwise}
 \end{array}\right. {\rm for}\ i\in T^{c}.
 \end{align*}
 
 For the implementation details of the SSNAL algorithm, the interested reader may refer to \citet{SDFLasso}. In the subsequent numerical tests, when applying the SSNAL to solve \eqref{Esubprob} or \eqref{inexact-subprob}, we terminate its iteration at $\zeta^{l}$ once $\|x^{l}-\mathcal{P}_{1}f(x^l\!-\!A^{\top}(Ax^l-b))\|\le\|b\|10^{-6}$.   
 
 All numerical tests of the subsequent two sections are performed in MATLAB 2020b on a laptop computer running 64-bit Windows Operating System with an Intel(R) Core(TM) i7-7700HQ 2.80GHz and 16 GB RAM. We evaluate the quality of an output by the relative error (relerr) from the true sparse solution, defined by $\|x^{\rm out}-\overline{x}\|/\|\overline{x}\|$.  
%--------------------------------------------------------------------------------
 \subsection{Influence of $\mu$ and $\varrho$ on iSCRA-TL1}\label{sec5.2}

 We explore the influence of parameters $\mu$ and $\varrho$ on the performance of iSCRA-TL1 via the sparse linear regression examples, which are often used to test the efficiency of sparse optimization models and algorithms for solving them \citep[see][]{Fan-FCP,Fan-LAMM}. These examples are generated randomly in a high-dimensional setting via model \eqref{LinearModel} with $e\sim N(\textbf{0},E)$, where every row of $A\in\mathbb{R}^{m\times n}$ obeys the multivariate normal distribution $N(\textbf{0},\Sigma)$.
 \begin{example}\label{test-exam1}
 The true $\overline{x}$ has the form $(\underbrace{\overline{z},\ldots,\overline{z}}_{40})^{\top}$ with $\overline{z}=(3,1.5,0,0,2,\underbrace{0,0,\ldots,0}_{25})$. 
 This example has $n=1200$ predictors, and the pairwise correlation between the $i$th predictor and the $j$th one is set to be $\Sigma_{i,j}=\theta^{|i-j|}$ with $\theta=0.6$. 
 \end{example}
\begin{example}\label{test-exam2}
 The true $\overline{x}$ has the form $(\underbrace{\overline{z},\ldots,\overline{z}}_{150})^{\top}$ with $\overline{z}=(\underbrace{0,\ldots,0}_{7},1)$. This example has $n=1200$ predictors, and the covariance matrix $\Sigma$ is same as in Example \ref{test-exam1} except $\theta=0.6$.
 \end{example}
 When $\theta$ is closer to $1$, the above two examples become more difficult due to high correlation. 
 \begin{example}\label{test-exam3}
 It is same as Example \ref{test-exam1} except that   $\overline{z}=(3,1.5,0,0,2,\underbrace{0,0,\ldots,0}_{20})$ and $\theta=0.75$. 
 \end{example}

 Assumption \ref{ass1} requires $\mu$ to satisfy $\mu\ge M=\|\overline{x}\|_\infty+\frac{\sqrt{r}}{m\sigma_{\!A}(r)}\|A^\top e\|_\infty$ though its value has no influence on the results of Theorems \ref{theorem1-nearly} and \ref{theorem2-nearly}. We validate this conclusion by applying Algorithm \ref{Alg1} with different $\mu$ to solve Example \ref{test-exam1} for $m=400$ and Example \ref{test-exam2} for $m=600$, respectively, and then looking at the relative error and the number of iterations. Figure \ref{fig0} shows the average relative error and number of iterations of the total $10$ running for every $\mu$, with $\lambda=(10/m)\|A^{\top}b\|_{\infty}$ for Example \ref{test-exam1} and $\lambda=(40/m)\|A^{\top}b\|_{\infty}$ for Example \ref{test-exam2}. We see that whether for Example \ref{test-exam1} or for Example \ref{test-exam2}, the relative error and the number of iterations first increase, and then keep unchanged when $\mu$ is more than a threshold $\overline{\mu}$. Among others, the threshold $\overline{\mu}=5>\|\overline{x}\|_{\infty}=3$ for Example \ref{test-exam1} and $\overline{\mu}=1.8>\|\overline{x}\|_{\infty}=1$ for Example \ref{test-exam2}. Furthermore, the relative error and number of iterations for $\mu\in[\|\overline{x}\|_\infty,\overline{\mu})$ have no significant difference from those after stabilizing. This coincides with the conclusions of Theorems \ref{theorem1-nearly} and \ref{theorem2-nearly}. In view of this, we take $\mu=10^3$ for the subsequent numerical tests.    
%--------------------------------------------------------------------------- 
 \begin{figure}[h]
 \centering
 \includegraphics[width=\textwidth]{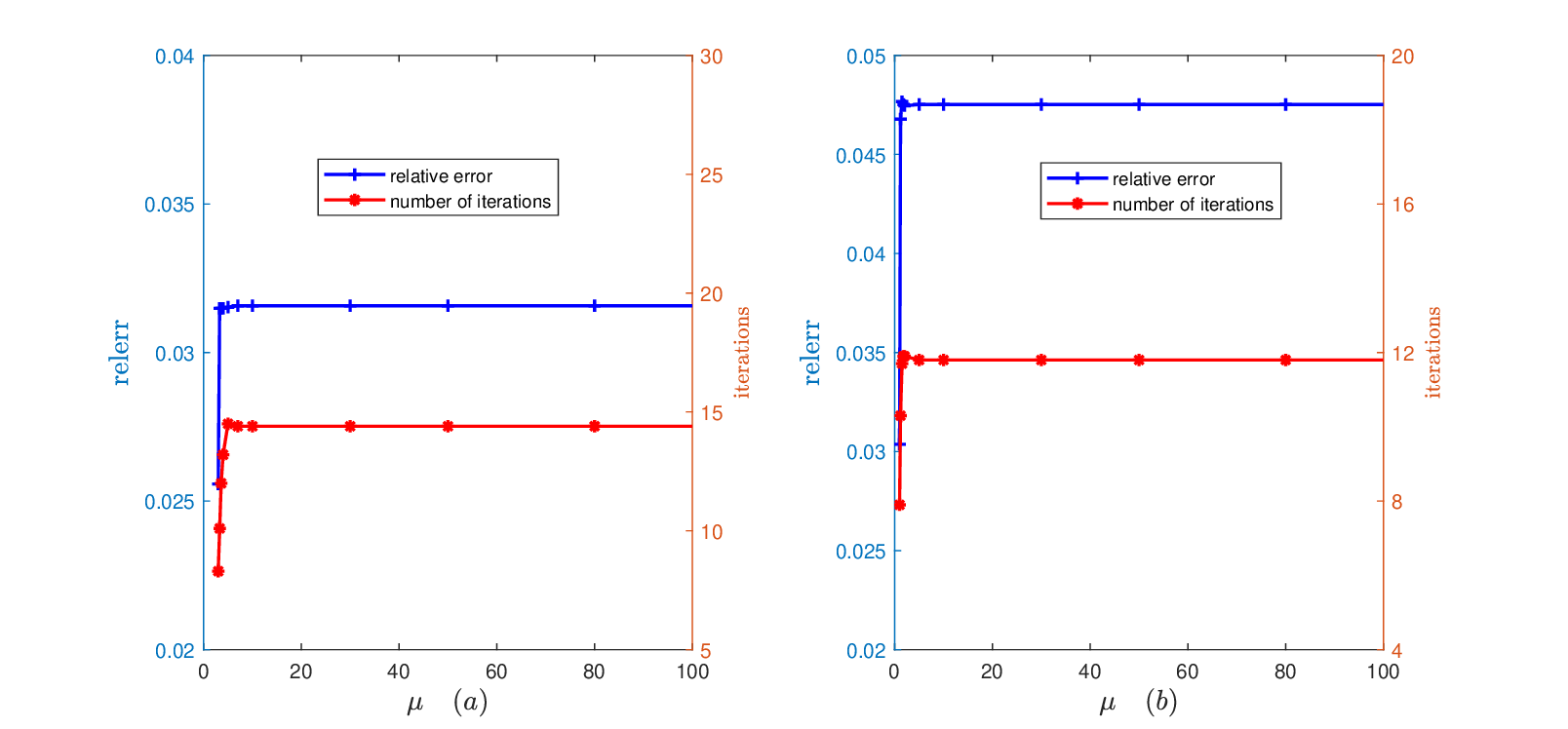}
 \caption{Relative error and number of iterations under different $\mu$ with $\varrho=0.8$: Example \ref{test-exam1} (left) and Example \ref{test-exam2} (right)}
 \label{fig0}
 \end{figure}

 From the iterations of Algorithm \ref{Alg1}, a smaller $\varrho$ implies a larger $I^k$, which by Proposition \ref{finite-Alg1} means the less number of iterations. By Theorems \ref{theorem1-nearly} and \ref{theorem2-nearly}, a smaller $\varrho$ requires $A$ to have a stronger rSRNSP which leads to a larger $\sigma_{\!A}(r)$. Then, a smaller $\varrho$ implies a smaller relative error for the final output of Algorithm \ref{Alg1}. Figure \ref{fig1} below shows the average relative error and number of iterations for the total $10$ running of Algorithm \ref{Alg1} with different $\varrho$ to solve Example \ref{test-exam3} with $m=400$ and $600$, respectively. The parameter $\lambda$ is chosen to be $(10/m)\|A^{\top}b\|_{\infty}$ We see that the smaller $\varrho$ requires the less number of iterations whether for $m=400$ (see Figure \ref{fig1}  (a)) or $m=600$ (see Figure \ref{fig1}  (b)). Figure \ref{fig1} indicates that for $m=400$ the smaller $\varrho$ yields the better relative error, and for $m=600$ the relative errors associated with different $\varrho$ have tiny difference. By Theorem \ref{theorem1-nearly}, when $A$ has the $(r, \varrho\beta(\lambda))$-rSRNSP($M,\gamma,\tau$), the value of $\varrho$ has no influence on the error bound. This interprets why for $m=600$  the relative errors associated with different $\varrho$ have tiny difference. When $m=400$, it is highly possible for the $(r, \varrho\beta(\lambda))$-rSRNSP($M,\gamma,\tau$) of $A$ not to hold. In this case, the above test results tell us that a smaller $\varrho$ is a good choice. Motivated by this, we always use $\varrho=0.2$ for the subsequent numerical tests.
 %-------------------------------------------------------------------- 
 \begin{figure}[h]
 \centering
 \includegraphics[width=\textwidth]{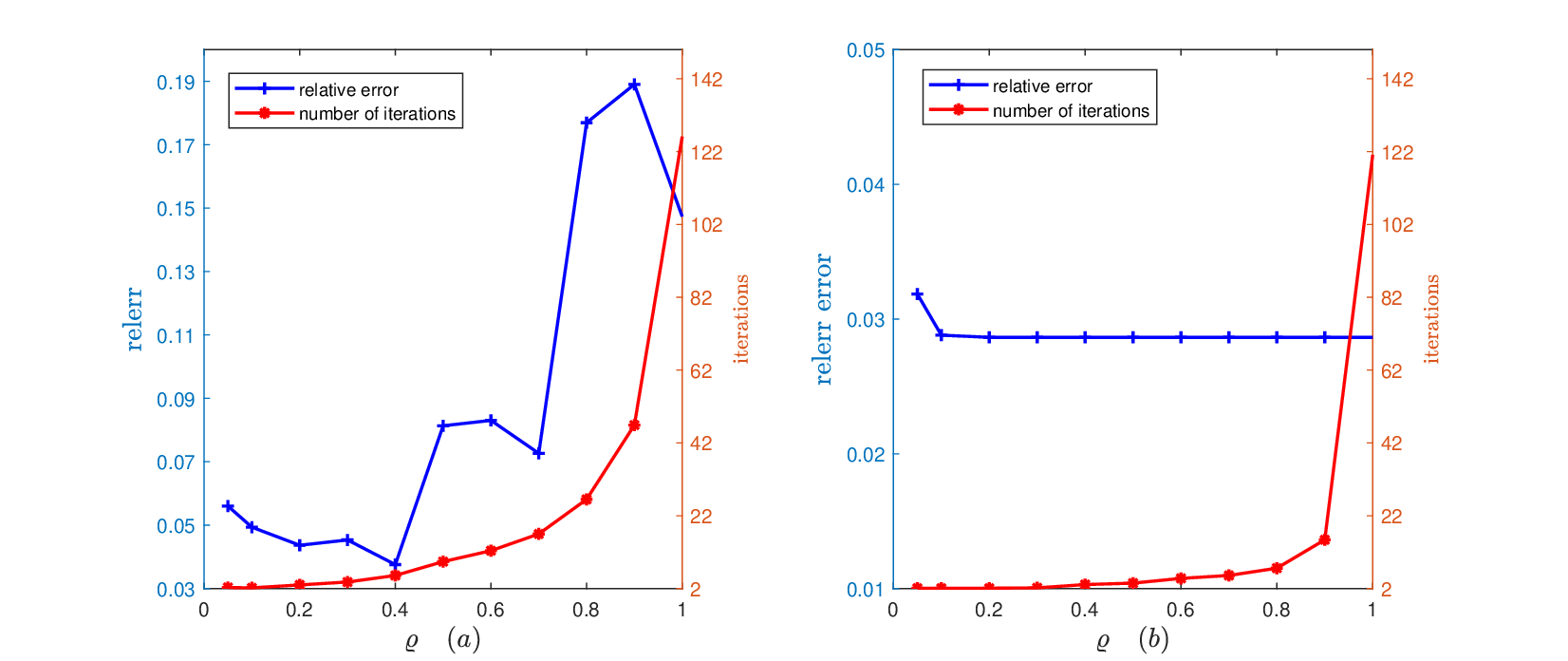}
 \caption{Relative error and number of iterations under different $\varrho$ with $\mu=10^3$}
 \label{fig1}
 \end{figure} 

%--------------------------------------------------------------------------------
 \subsection{Numerical comparison}\label{sec5.3}

 In this section we compare the performance of iSCRA-TL1 with that of three sequential convex relaxation algorithms, LLA-SCAD, MSCR-cL1 and DCA-TrL1, on synthetic and real-data examples. The constant $a$ in LLA-SCAD is chosen to be $3.7$ by \citet{Fan-SCAD}, the parameter $\varepsilon$ in MSCR-cL1 is chosen to be $0.5\sqrt{\ln(n)/m}$ by the theoretical results of \citet{ZT-Multi-b}, and the constant $a$ in DCA-TrL1 is chosen as $1$ by \citet{TL1}. Our preliminary tests indicate that the DCA-TrL1 with a smaller $c$ is faster than the one with a larger $c$, so we choose $c=10^{-8}$. The four solvers stop at the $k$th iteration whenever  
 \[
 {\|x^{k}-x^{k-1}\|}/{\|x^k\|_1}\le 10^{-3}\ \ {\rm or}\ \ k>k_{\rm max}=50.
 \]
%------------------------------------------------------------------------------
 \subsubsection{Synthetic data}\label{sec5.3.1}
 
 Consider that the ranges of $\lambda$ corresponding to better relative errors are different for the four solvers, but all of them can be identified by the form $[\underline{c},\overline{c}](\|A^{\top}b\|_{\infty}/m)$. For fairness, we compare their performance for a group of $\lambda=(c_{\lambda}/m)\|A^{\top}b\|_{\infty}$ with $c_{\lambda}$ from an interval for solving the following two sparse linear regressions, both of which have high correlation. 
 \begin{example}\label{test-exam4}
 It is same as Example \ref{test-exam1} except that $\overline{z}=(3,1.5,0,0,2,\underbrace{0,0,\ldots,0}_{20})$ and $\theta=0.8$. 
 \end{example}
 \begin{example}\label{test-exam5}
 The true $\overline{x}$ has the form $(\underbrace{\overline{z},\ldots,\overline{z}}_{50})^{\top}$ with $\overline{z}=(\underbrace{0,0,\ldots,0}_{18},1.2,1)$. This example has $n=1000$ predictors, and the matrix $\Sigma$ is same as in Example \ref{test-exam1} but with $\theta=0.8$.
 \end{example}  
 
 Figure \ref{fig2} (a) plots the average relative errors of the four solvers for solving Example \ref{test-exam4} with $m=400$ and $c_{\lambda}\in\{0.001,\,0.005,\,0.01,\,0.05,\,0.1,\,0.5,\,1.0,\,3.0,\,5.0,\,7.0,\,10,\,15,\,20\}$, and Figure \ref{fig2} (b) plots the average relative errors for solving Example \ref{test-exam4} with $m=600$ and $c_{\lambda}\in\{0.01,\,0.05,\,0.1,\,0.5,\,1.0,\,3.0,\,5.0,\,7.0,\,10,\,15,\,20,\,25,\,30\}$. For every $\lambda$, the result is the average of the total $10$ running. From Figure \ref{fig2}, as the parameter $\lambda$ increases, the relative errors yielded by iSCRA-TL1 and MSCR-cL1 have smaller variation than those yielded by LLA-SCAD and DC-TL1. When $m=400$, the relative errors returned by iSCRA-TL1 for $c_{\lambda}\in[0.001,10]$ are comparable with those yielded by LLA-SCAD for $c_{\lambda}\in[3.0,10]$, by MSCR-cL1 for $c_{\lambda}\in[0.001,7]$, and DC-TL1 for $c_{\lambda}\in[0.001,5]$; and when $m=600$, the relative errors returned by iSCRA-TL1 for $c_{\lambda}\in[3,25]$ are comparable with those yielded by LLA-SCAD for $c_{\lambda}\in[1,20]$, by MSCR-cL1 for $c_{\lambda}\in[0.5,25]$ and by DC-TL1 for $c_{\lambda}\in[0.5,10]$. Obviously, the interval of $c_{\lambda}$ for iSCRA-TL1 to have better relative errors is the largest for $m=400$, and is only smaller than that of MSCR-cL1 for $m=600$. The interval of $c_{\lambda}$ for DC-TL1 to have better relative errors is the smallest whether for $m=400$ or for $m=600$. This shows that the conditions for iSCRA-TL1 to achieve better relative errors are weaker than those for the other three solvers, agreeing with the results of Theorems \ref{theorem1-nearly} and \ref{theorem2-nearly}.
%---------------------------------------------------------------------  
 \begin{figure}[h]
 \centering
 \includegraphics[width=\textwidth]{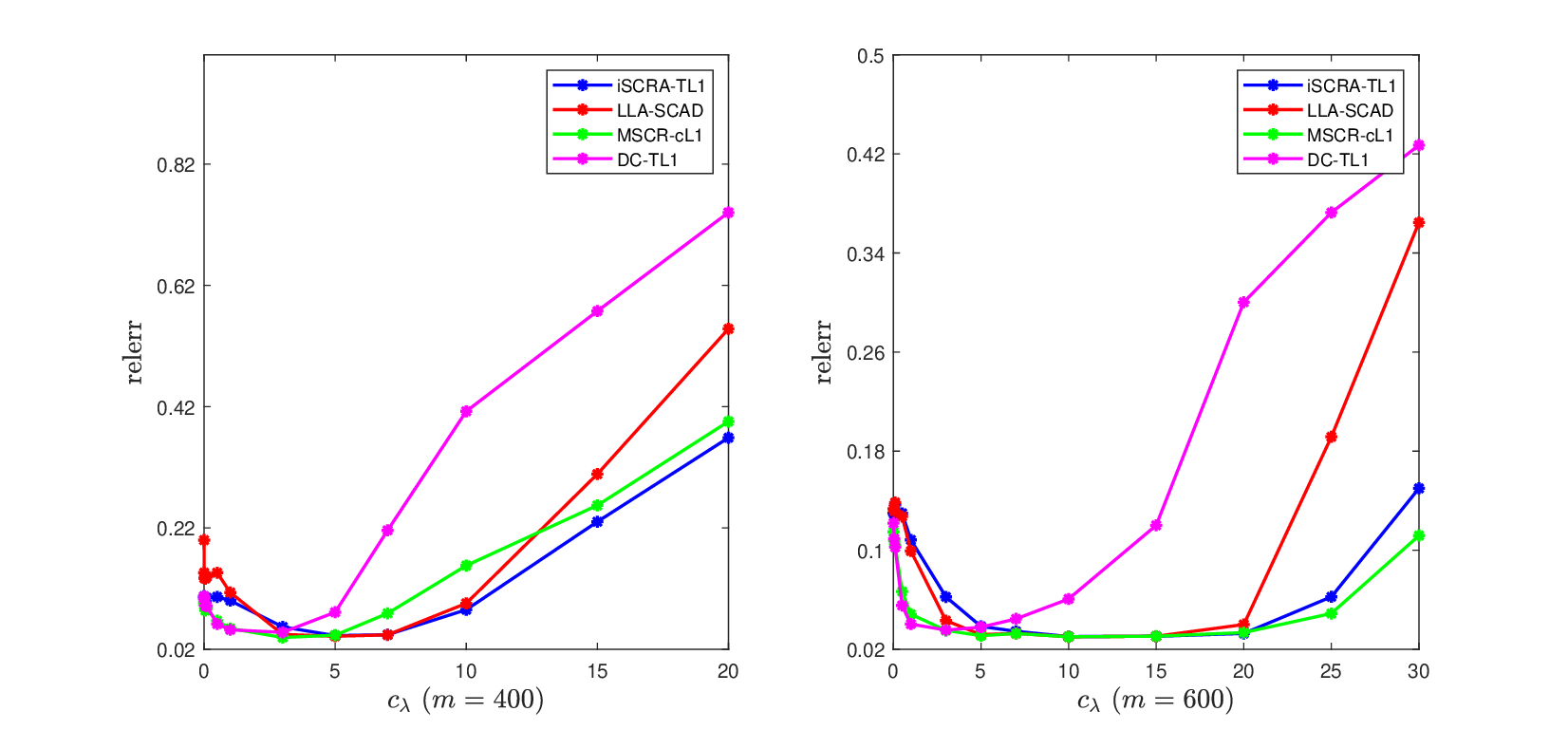}
 \caption{Relative errors of the four solvers with  different $\lambda$ for solving Example \ref{test-exam4}}
 \label{fig2}
 \end{figure} 
\begin{figure}[h]
 \centering
 \includegraphics[width=\textwidth]{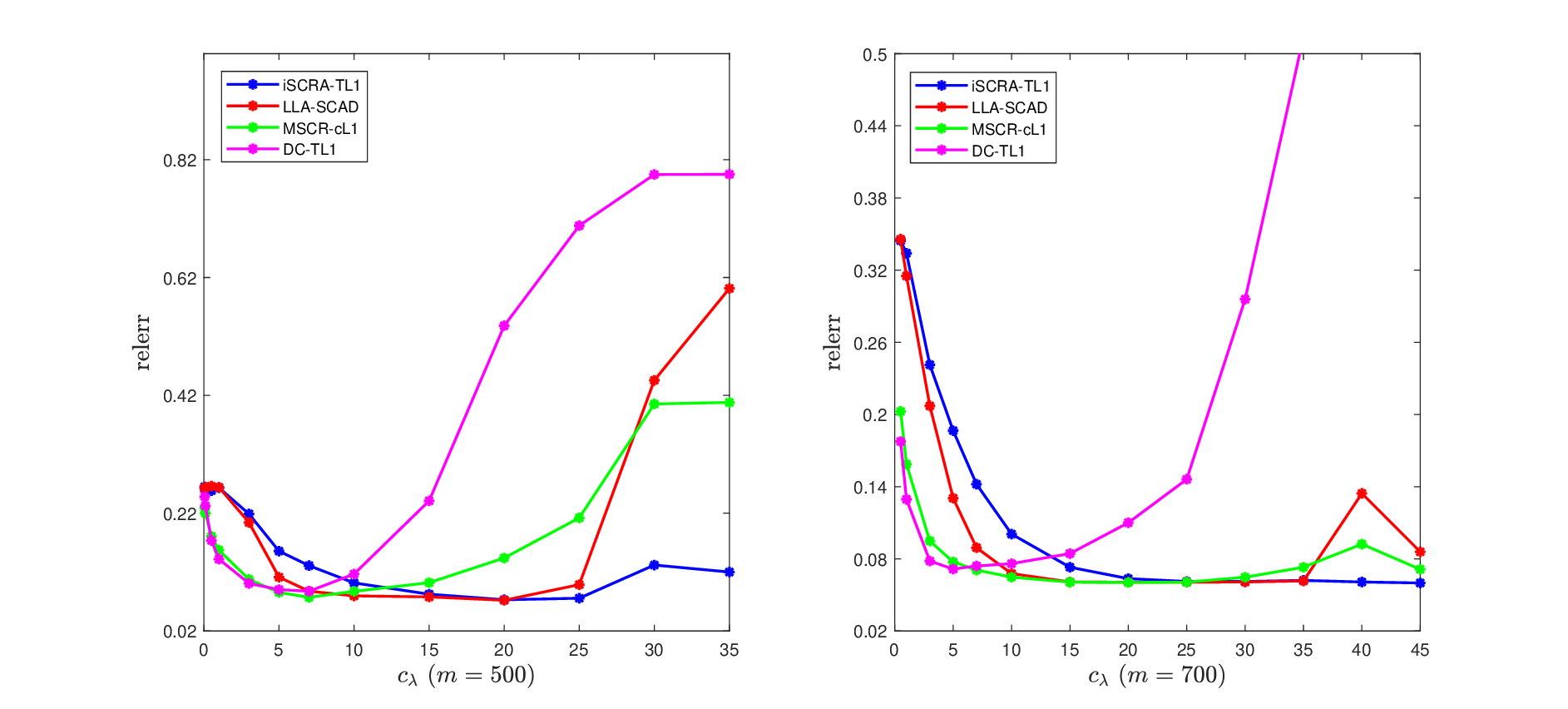}
 \caption{Relative errors of the four solvers with  different $\lambda$ for solving Example \ref{test-exam5}}
 \label{fig3}
 \end{figure} 

 Figure \ref{fig3} (a) plots the average relative errors of the four solvers for solving Example \ref{test-exam5} with $m=500$ and $c_{\lambda}\in\{0.05,\,0.1,\,0.5,\,1.0,\,3.0,\,5.0,\,7.0,\,10,\,15,\,20,\,25,\,30,\,35\}$, and  Figure \ref{fig3} (b) plots the average relative errors for solving Example \ref{test-exam5} with $m=700$ and $c_{\lambda}\in\{0.5,\,1.0,\,3.0,\,5.0,\,7.0,\,10,\,15,\,20,\,25,\,30,\,35,\,40,\,45\}$. For every $\lambda$, the result is the average of the total $10$ running. Figure \ref{fig3} (a) shows that for $m=500$, as the parameter $\lambda$ increases, the relative errors yielded by iSCRA-TL1 has smaller variation than those yielded by the other three solvers, and the relative errors returned by iSCRA-TL1 for $c_{\lambda}\in[7,35]$ are better and almost the same. The interval of $c_{\lambda}$ for iSCRA-TL1 to have better relative errors is the largest. Figure \ref{fig3} (b) shows that for $m=700$ the relative errors yielded by iSCRA-TL1 for $c_{\lambda}\in[15,45]$ have tiny difference, and are comparable with those yielded by LLA-SCAD for $c_{\lambda}\in[10,35]$ and MSCR-cL1 for $c_{\lambda}\in[7,35]$. The range of $\lambda$ for iSCRA-TL1 to yield better relative error is comparable with that of $c_{\lambda}$ for LLA-SCAD and MSCR-cL1, while the range of $\lambda$ for DC-TL1 to yield better relative error is the smallest and lies on the left hand side of the one for the other three solvers. 

 From the above numerical results, iSCRA-TL1 has better robustness under the scenario of low sample or high sample, and the range of $\lambda$ for it to yield better relative errors is larger than that of $\lambda$ for the other three sequential convex relaxation algorithms.  
%------------------------------------------------------------------------------
 \subsubsection{Real data}\label{sec5.3.2} 

 In this part, we compare the performance of the four solvers with the data $(A,b)$ from the LIBSVM data sets \citep{LIBSVM}. These data sets are collected from 10-K Corpus \citep{Corpus} and the UCI data repository \citep{UCI}. For computational efficiency, the zero columns in $A$ are removed. As suggested in \citep{Huang-NIPS}, for the data sets \textbf{pyrim}, \textbf{abalone}, \textbf{bodyfat}, \textbf{housing}, \textbf{mpg}, and \textbf{space\_ga}, we expand their original features by using polynomial basis functions over those features. For example, the last digit in \textbf{pyrim5} indicates that an order 5 polynomial is used to generate the basis functions. This naming convention is also used for the rest of the expanded data sets. As pointed out in \citet{SDFLasso}, these test instances are quite difficult in terms of the problem dimensions and the largest eigenvalue of $AA^{\top}$. 
%------------------------------------------------------------------------------------------
 \begin{table}
  \caption{Performance of the four solvers for the ten data from LIBSVM datasets\\ `a=iSCRA-TL1', `b=LLA-SCAD', `c=MSCR-cL1' and `d=DC-TL1'.\\ ``NZ'' denotes the number of nonzeros in the final output of four solvers.}
  \label{result-lib}
  \centering
 \scalebox{0.7}{
 \begin{tabular}{c|c|cccc|cccc|cccc}
 \hline
\hline
\multirow{2}*{Name of data}  & \multirow{2}*{$m\lambda$} & \multicolumn{4}{ c| }{NZ}       &\multicolumn{4}{ c| }{loss}            & \multicolumn{4}{ c }{Time(s)} \\
   &     & a&b&c&d   &   a&b&c&d          &  a&b&c&d \\
 \hline
 E2006.train   & 1.0 & 2 & 2 & 2 & 2 & 3.839e-1 & 3.839e-1& 3.839e-1& 3.839e-1 & 14.55 & 14.43 & 14.32 & 22.37 \\
16087;4272227 & 0.1 & 111  & 145 & 134 & 128 & 3.630e-1 & 3.611e-1 & 3.617e-1 & 3.613e-1 & 186.9 & 161.9 & 173.7 & 275.2\\
 \hline
E2006.test  & 1.0& 1 & 1 & 1 &1 & 3.809e-1& 3.809e-1& 3.809e-1& 3.809e-1& 7.16 & 6.17 & 7.03 & 12.25\\
3308;150358 &0.1 & 11  & 13 & 12 & 13 & 3.637e-1 & 3.626e-1 & 3.627e-1 & 3.622e-1 & 216.8 & 307.1 & 310.6 & 735.4\\
\hline
log1p.E2006.train  & 100 & 61 & 68 & 87 & 84 & 3.519e-1 & 3.558e-1 & 3.588e-1 & 3.579e-1& 683.4 & 622.7 & 334.1  & 948.0 \\
 16087;4272227  &40 & 376 & 326 & 431 & 426 & 3.303e-1& 3.366e-1 & 3.436e-1 & 3.429e-1 & 1574& 1937 & 527.1 & 948.5\\
 \hline
 log1p.E2006.test & 50 & 21 & 33 & 16 & 19 & 3.432e-1 & 3.487e-1 & 3.571e-1 &3.535e-1 & 272.4 & 239.4 & 228.4 & 328.3 \\
3308;4272226  & 10 & 716  & 573 & 837 & 808 & 1.848e-1 & 2.039e-1 & 2.593e-1 & 2.575e-1 & 759.2 & 683.1 & 449.4 & 705.8 \\
 \hline
 abalone7  & 10 & 31 & 65 & 42 & 37 & 2.069 & 2.056 & 2.063 & 2.064 & 75.28 & 64.54 & 62.28 & 142.8 \\
4177;6435  & 1.0 & 143  & 205  & 168 & 129 & 1.980 & 1.976 & 1.982 & 1.987 & 538.4 & 474.1 & 366.2 & 814.9 \\
 \hline
 bodyfat7  & 1.0 & 2 & 2 & 2 &2 & 2.960e-3 & 2.960e-3 & 1.036e-2 & 1.069e-2  & 2.42 & 1.89 & 1.73 & 2.71\\
252;116280  & 0.01 & 7  & 15 & 10 & 9  & 2.189e-3 & 2.189e-3 & 2.427e-3 & 2.440e-3 &8.55 & 6.39 & 5.91 & 7.24 \\
 \hline
 housing7 & 10 & 90 & 160 & 111 & 75 & 1.934 & 1.908 & 2.121 & 1.926 & 10.35& 7.60 & 7.83 & 23.70 \\
 506;77520  &1.0 & 227 & 343 & 199 & 188 & 9.513e-1 & 9.480e-1 & 1.059 & 9.568e-1 & 22.99 & 24.99 & 20.22 & 34.97 \\
 \hline
mpg7  & 5 & 71 & 74 & 60 & 55 & 2.018 & 2.042 & 2.127 & 2.069 & 0.96 & 0.56 & 0.52 & 1.33\\
392;3432  & 0.5 & 172 & 212 & 174  & 149 & 1.368 & 1.357 & 1.416  & 1.408 & 4.49 & 4.70 & 2.67 & 3.58 \\
 \hline
 pyrim5   & 0.1 & 34  & 37 & 62 & 43 & 1.740e-2  & 1.864e-2 & 2.290e-2 & 2.276e-2 & 5.67 & 3.58 & 2.39 & 6.03\\
74;201376 & 0.01 & 54  & 85 & 78 & 62 & 7.969e-3 & 7.750e-3 & 9.053e-3  & 8.890e-3 & 8.36 & 4.21 & 3.38 & 7.41\\
 \hline
 space ga9 & 1.0 & 19 & 22  & 20  & 18 & 1.094e-1 & 1.074e-1 & 1.113e-1 & 1.111e-1 & 24.71 & 17.04 & 16.82 &  32.98  \\
3107;5005  & 0.1  & 60 & 67 & 55 & 48 & 9.700e-2 & 9.700e-2 & 9.870e-2 & 9.923e-2 & 149.2 & 77.53 & 63.56 & 126.4\\
\hline
\hline
\end{tabular}}
\end{table}
 
 Table \ref{result-lib} reports the numerical results of the four solvers with two different $\lambda$. We see that iSCRA-TL1 yields the sparsest solutions for most of test examples, and the loss values returned by it are comparable with those yielded by the other three solvers. For \textbf{pyrim5} and \textbf{bodyfat7} with $\lambda=0.01/m$, iSCRA-TL1 produces the solutions with the best sparsity and lowest loss. For most of examples, the running time of iSCRA-TL1 is comparable with that of LLA-SCAD and MSCR-cL1, but less than that of DC-TL1. This shows that our iSCRA-TL1 is efficient in dealing with large-scale and difficult examples.   
 
 %  Next we take a closer look at the influence of parameter $\lambda$ on iSCRA-TL1. From Assumption \ref{ass1}, the choice range of parameter $\lambda$ is an interval $[\underline{\lambda},+\infty)$, where the left endpoint $\underline{\lambda}$ depends on the number of samples, the noise vector $e$ and the support $\overline{S}$. Note that $e$ and $\overline{S}$ is unavailable in practical computation, so we replace the left endpoint $\underline{\lambda}$ with  $(2/m)\|A^{\top}b\|_{\infty}$. This means that an appropriate $\lambda$ should be the form of $(2c_{\lambda}/m)\|A^{\top}b\|_{\infty}$ for some $c_\lambda$. To find an appropriate range of $c_{\lambda}$, we apply Algorithm \ref{Alg1} with $\lambda=(2c_{\lambda}/m)\|A^{\top}b\|_{\infty}$ for $c_{\lambda}\in\{0.5,\,1.0,\,5.0,\,10,\,20,\,30,\,40,\,50,\,60,\,70,\,80,\,90,\,100,\,110,\,120\}$ to solve problem \eqref{CS} with data from Example \ref{test-exam4}. Figure \ref{fig2} plots the relative error yielded by Algorithm \ref{Alg1} with different $c_{\lambda}$ with $m=500$ and $m=800$, respectively. We see that the range of $\lambda$ for better relative error under $m=500$ is smaller than the one under $m=800$. 
 
 % \begin{figure}[h]
 % \centering
 % \includegraphics[width=0.9\textwidth]{figure_lam.eps}
 % \caption{Relative error yielded by Algorithm \ref{Alg1} with  different $\lambda$ for $\mu=10^3,\varrho=0.2$}
 % \label{fig44}
 % \end{figure} 
%-------------------------------------------------------------------------------------------
 \section{Conclusions}\label{sec6}
	
 For the high dimensional sparse linear regression problem \eqref{LinearModel} or the zero-norm regularized problem \eqref{CS}, we proposed an inexact sequential convex relaxation algorithm by solving a sequence of truncated $\ell_1$-norm regularized problems \eqref{inexact-subprob}, and provided its theoretical certificates by means of the novel rRNSP and rSRNSP of the matrix $A$, which are shown to be strictly weaker than the existing robust NSP and REC. We showed that under a mild rSRNSP, our iSCRA-TL1 can identify the support of the true $r$-sparse regression vector via at most $r$ truncated $\ell_1$-norm regularized minimization, and established the error bound of its iterates from the oracle solution and the true sparse vector. The oracle estimator can be obtained via at most $r+1$ truncated $\ell_1$-norm regularized minimization. To the best of our knowledge, this is the first sequential convex relaxation algorithm to achieve the oracle estimator via at most $r\!+\!1$ convex subproblems under a condition weaker than the robust NSP and REC, and allow the initial Lasso estimator not to be of good quality.

 % There are three interesting directions of research that are worth pursuing.
 % First, it would be worth to study (inexact) sequential convex relaxation algorithm for
 % the zero-norm regularized minimization problems with nonsmooth loss function, e.g., $l_1$ and $l_2$ loss.
 % Second, as is well known, group structure is one of the most important structures in the field of
 % statistics, machine learning and  signal and image processing.
 % Thus, it is worthwhile to extend our iSCRA-TL1 to group sparse linear regression problems.
 % Third, it is interesting to see which type of matrices satisfies rRNSP and rSRNSP with high probability.

% Acknowledgements and Disclosure of Funding should go at the end, before appendices and references

\acks{The work of Shujun Bi is supported by the National Natural Science Foundation of China under grant number 12371323
   and the Natural Science Foundation of Guangdong Province under Project No. 2024A1515012611.
   The work of Shaohua Pan is supported by the National Natural Science Foundation of China under grant number 12371299.}


\newpage

\appendix
\section{Some lemmas}\label{secA}

This section includes three technical lemmas that will be used to establish the theoretical guarantees of Algorithm \ref{Alg1}. 
%The first lemma provides an upper bound for $\|(A_J^{\top}A_J)^{-1}A_J^{\top}\|$ with $\emptyset\neq J\subseteqq \overline{S}$.  
\begin{lemma}\label{LemmaAJ}
 If $\sigma_{\!A}(r)>0$, then $\sqrt{m}\|(A_J^{\top}A_J)^{-1}A_J^{\top}\|\leq 1/\sigma_{\!A}(r)$ holds for any $\emptyset\neq J\subseteqq \overline{S}$.
\end{lemma}
 \begin{proof}
   Since $\sigma_{\!A}(r)>0$, the smallest eigenvalue of $\frac{1}{m}A_J^{\top}A_J$ is no less than $(\sigma_{\!A}(r))^2$. Then, $m\|(A_J^{\top}A_J)^{-1}\|\leq 1/(\sigma_{\!A}(r))^2$. Let $B:=\sqrt{m}(A_J^{\top}A_J)^{-1}A_J^{\top}$. We have $BB^{\top}=m(A_J^{\top}A_J)^{-1}$. 
   Consequently, $\sqrt{m}\|(A_J^{\top}A_J)^{-1}A_J^{\top}\|=\|B\|= \sqrt{m\|(A_J^{\top}A_J)^{-1}\|} \leq 1/\sigma_{\!A}(r)$.
 \end{proof}

 The following lemma provides the properties of an oracle estimator of model \eqref{LinearModel}.
%----------------------------------------------------------------------------- 
 \begin{lemma}\label{Lemma-oracle}
 Suppose that $\sigma_{\!A}(r)>0$ and $\widehat{M}$ is defined by \eqref{kappa12}. Then, model \eqref{LinearModel} has a unique oracle estimator $x^{\rm o}=(\overline{x}_{\overline{S}}+(A_{\overline{S}}^\top A_{\overline{S}})^{-1}A_{\overline{S}}^\top e; \textbf{0})^\top$, which satisfies the following properties 
 \begin{itemize}
  \item[(i)]
 $\|x^{\rm o}\|_\infty\leq \|\overline{x}\|_\infty+\frac{\sqrt{r}}{m\sigma_{\!A}(r)}\|A^\top e\|_\infty$, $\|x^{\rm o}\|_1\leq 0.8\widehat{M}$ and
 $\|\overline{x}-x^{\rm o}\|_1\leq \frac{r\|A^\top e\|_\infty}{m\sigma_{\!A}(r)}$;
 
 \item[(ii)] $A_{\overline{S}}^\top(Ax^{\rm o}-b)=\textbf{0}$ and $b-Ax^{\rm o}=(E-A_{\overline{S}}(A_{\overline{S}}^\top A_{\overline{S}})^{-1}A_{\overline{S}}^\top) e$.
 \end{itemize}
 \end{lemma}
 \begin{proof}  
 Note that $\min_{{\rm supp}(x)\subset \overline{S}}\|Ax-b\|^2 \Longleftrightarrow\min_{x=(x_{\overline{S}}; \textbf{0})\in \mathbb{R}^n}\|A_{\overline{S}}x_{\overline{S}}-b\|^2$. Due to $\sigma_{\!A}(r)>0$, they have the same unique optimal solution, written as $x^{\rm o}$. Using $b=A_{\overline{S}}\overline{x}_{\overline{S}}+e$ leads to
 \[
 x^{\rm o}=((A_{\overline{S}}^{\top}A_{\overline{S}})^{-1}A_{\overline{S}}^{\top}b;\ \textbf{0})
 =(\overline{x}_{\overline{S}}+(A_{\overline{S}}^{\top}A_{\overline{S}})^{-1}A_{\overline{S}}^{\top}e;\ \textbf{0}).
 \]
 This, together with $\|(A_{\overline{S}}^{\top}A_{\overline{S}})^{-1}\|\leq \frac{1}{m\sigma_{\!A}(r)}$ and $\|A_{\overline{S}}^{\top}e\|\leq\sqrt{r}\|A_{\overline{S}}^{\top}e\|_\infty$, implies that
 \begin{align}
  \|x^{\rm o}\|_\infty&\leq\|\overline{x}\|_\infty+\|(A_{\overline{S}}^{\top}A_{\overline{S}})^{-1}A_{\overline{S}}^{\top}e\|_\infty\leq\|\overline{x}\|_\infty+\|(A_{\overline{S}}^{\top}A_{\overline{S}})^{-1}\|\|A_{\overline{S}}^{\top}e\|\nonumber\\
  &\leq\|\overline{x}\|_\infty +\frac{\sqrt{r}}{m\sigma_{\!A}(r)}\|A_{\overline{S}}^{\top}e\|_\infty.\nonumber
 \end{align}
 Using $\|Ax^{\rm o}-b\|^2\le\|b\|^2$ yields that $\|Ax^{\rm o}\|\leq 2\|b\|$. Together with the definition of $\sigma_{\!A}(r)$ and $\|x^{\rm o}\|_0\leq r$, we have $\sqrt{m}\sigma_{\!A}(r)\|x^{\rm o}\|\leq 2\|b\|$, so that $\|x^{\rm o}\|_1\leq \sqrt{r}\|x^{\rm o}\|\leq \frac{2\sqrt{r}\|b\|}{\sqrt{m}\sigma_{\!A}(r)}\leq 0.8\widehat{M}$.
 Now using the expression of $x^{\rm o}$ results in
 \(
 \|\overline{x}-x^{\rm o}\|_1= \|(A_{\overline{S}}^\top A_{\overline{S}})^{-1}A_{\overline{S}}^\top e\|_1\leq\sqrt{r}\|(A_{\overline{S}}^\top A_{\overline{S}})^{-1}A_{\overline{S}}^\top e\|\leq\sqrt{r}\|(A_{\overline{S}}^\top A_{\overline{S}})^{-1}\|\|A_{\overline{S}}^\top e\|\leq \frac{r\|A_{\overline{S}}^\top e\|_\infty}{m\sigma_{\!A}(r)}.
 \)
 Then part (a) follows. By using the expression of $x^o$ and the first order optimality condition of \eqref{oracle-esti}, it is easy to obtain part (b).
 \end{proof}

 To close this part, we state a technical lemma. Since its proof is direct, we here omit it.
 %-----------------------------------------------------------------------------------------------
 \begin{lemma}\label{PreLemma}
 For any $M>0$ and $\varpi\in\mathbb{R}$, if $|a|\le M$, then  
 $|a|-|a+\varpi| \leq \min\{|\varpi|,2 M\!-|\varpi|\}$.
 \end{lemma}

\section{Proofs}\label{secB}

 \subsection{Proof of Proposition \ref{bound-feasol}}\label{app-prop1}
  \begin{proof}
 The proof is similar to that of \citep[][Theorem II.1]{Bi-iSCRA}, and we include it for completeness. Fix any vector $x$ such that $\|Ax-b\|\leq\sqrt{\|e\|^2+2m\lambda(1\!+\!\varsigma_0)\|\overline{x}\|_1}$ and pick any $k\in[r\!-\!1]_{+}$.
 Let $I=\operatorname{supp}(x,k)$ if $k>0$, otherwise $I=\emptyset$. It holds that
  \[
	A_{I^c} x_{I^c}=A\overline{x}-A_I x_I+(b-A\overline{x})+(Ax-b)
    =A_{\overline{S}} \overline{x}_{\overline{S}}-A_I x_I+(b-A\overline{x})+(Ax-b).
  \]
  Together with 
  $\|b-\!A\overline{x}\|=\|e\|$ and $\|Ax-b\|\leq\sqrt{\|e\|^2+2m\lambda(1\!+\!\varsigma_0)\|\overline{x}\|_1}$,
  it follows that
 \begin{align}\label{temp-ineq21}
  \|A_{I^c}x_{I^c}\|
   &\ge\|A_{\overline{S}} \overline{x}_{\overline{S}}-A_I x_I\|-\|b-A\overline{x}\|-\|Ax-b\|\nonumber\\
   &\ge\|(A_{\overline{S}\backslash I}\ \ A_{\overline{S}\cap I}\ \ A_{I\backslash\overline{S}})(\overline{x}_{\overline{S}\backslash I};\,(\overline{x}-x)_{\overline{S}\cap I};\, x_{I\backslash\overline{S}})\|-\|e\|-\|Ax-b\| \nonumber\\
   & \geq\sqrt{m}\,\sigma_{\!A}(r\!+\!k)\|\overline{x}_{\overline{S}\backslash I}\|-\|e\|-\sqrt{\|e\|^2+2m\lambda(1\!+\!\varsigma_0)\|\overline{x}\|_1}\nonumber\\
    & \geq\sqrt{m}\,\sigma_{\!A}(r\!+\!k)\sqrt{r-k}|\overline{x}|_r^\downarrow-\|e\|-\sqrt{\|e\|^2+2m\lambda(1\!+\!\varsigma_0)\|\overline{x}\|_1}>0
  \end{align}
  where 
  the third inequality is obtained by using $|\overline{S}\backslash I|+|\overline{S}\cap I|+|I\backslash\overline{S}|=|\overline{S}\cup I|\leq r+k$
  and the definition of $\sigma_{\!A}(r+k)$, the fourth one is due to $|\overline{S}\backslash I|\geq r-k$ and $\min_{i\in \overline{S}}|x_i|=|\overline{x}|_r^\downarrow$,
  and the last one is due to
$\sqrt{m}|\overline{x}|_{r}^{\downarrow}>
\frac{2\|e\|+\sqrt{2m\lambda(1+\varsigma_0)\|\overline{x}\|_1}}{\sigma_{\!A}(2r-1)}$ and $k\leq r-1$.
  Combining \eqref{temp-ineq21} with $\|A\|\|x_{I^c}\|\geq\|A_{I^c}x_{I^c}\|$ leads to 
  \(
  \|x_{I^c}\|\geq\frac{\sqrt{m}\sigma_{\!A}(r+k)\sqrt{r-k}\,|\overline{x}|_{r}^{\downarrow}-\|e\|-\sqrt{\|e\|^2+2m\lambda(1\!+\!\varsigma_0)\|\overline{x}\|_1}}{\|A\|}.
  \)
 Together with
$|x|_{k+1}^{\downarrow}\!=\!\left\|x_{I^c}\right\|_{\infty}\!\geq\! \frac{\left\|x_{I^c}\right\|}{\sqrt{n-k}}$, we have $|x|_{k+1}^{\downarrow}\ge\vartheta_k$ for all $k\in[r\!-\!1]_{+}$. This along with the definition of $\beta_k(\vartheta)$ implies that $\beta_k(\vartheta)\ge\vartheta_k$ for $k\in[r\!-\!1]$.
 Recall that \(
 \beta_{0}(\lambda)\geq\min\limits_{x\in\mathbb{R}^n}\big\{\|x\|_\infty\ \text { s.t.}\ \|Ax-b\|\le\!\sqrt{\|e\|^2+2m\lambda(1\!+\!\varsigma_0)\|\overline{x}\|_1}\big\}.
 \)
 Then $\beta_0(\vartheta)\ge\vartheta_0$ holds. 
 \end{proof}

 \subsection{Proof of Proposition \ref{REC-rNSP}}\label{app-3-3}

  \begin{proof}
  Suppose that \eqref{REC-ineq} holds with constant $c>1$.
  Fix any $S\subset[n]$ with $|S|=r$ and any $d\in\mathbb{R}^n$.
  If $\|d_{S^c}\|_1> c\|d_S\|_1$, we have $\sum_{i\in S}|d_i|\le \frac{1}{c}\,\sum_{i\in S^c}|d_i|$.
  This implies that \eqref{ineq1-NSP} holds with $\gamma=\frac{1}{c}$. Now, suppose that $\|d_{S^c}\|_1\leq c\|d_S\|_1$.
  From \eqref{REC-ineq}, we immediately obtain
  \[
  \frac{1}{\sqrt{m}}\|Ad\|\geq \chi(c)\|d\|\geq \chi(c)\|d_S\|\geq \frac{\chi(c)}{\sqrt{r}}\|d_S\|_1.
  \]
 This implies that $\|d_S\|_1\leq \frac{\sqrt{r}}{\chi(c)\sqrt{m}}\|Ad\|$, so \eqref{ineq1-NSP} holds with $\tau=\frac{1}{\chi(c)}$, and the robust NSP of order $r$ holds with $\gamma=\frac{1}{c}$ and $\tau=\frac{1}{\chi(c)}$. The second part follows Remarks \ref{RNSP-remark1} and \ref{SRNSP-remark1}.
 \end{proof}

\section{Proofs for the main results of Section \ref{sec4.2}}\label{secC}

\subsection{Proof of Lemma \ref{bound-xhat1}}\label{app-4-1}
\begin{proof} 
  Since $\sigma_{\!A}(r)>0$ and $|\overline{S}|=r$,
  the constants $\kappa$ and $\widehat{M}$ are well defined.
  Pick any optimal solution $\widehat{x}$ of problem \eqref{equa-JTL1}. 
 From the feasibility of $x=\textbf{0}$ to \eqref{equa-JTL1}, 
  \begin{align}\label{boundx0}
  \frac{1}{2m}\|A\widehat{x}-b\|^2+\lambda\sum_{i\in J^c}|\widehat{x}_i|-\lambda\langle \xi, \widehat{x}\rangle
  &\le\frac{1}{2m}\|b\|^2.
  \end{align}
  When $J=\emptyset$, from \eqref{boundx0} and $\|\xi\|_\infty<\frac{1}{5}$, we have
  \(
  \lambda(1-\|\xi\|_\infty)\|\widehat{x}\|_1\leq\frac{\|b\|^2}{2m},
  \)
  which implies that
  \(
  \|\widehat{x}\|_1\leq \frac{\|b\|^2}{2m\lambda(1-\|\xi\|_\infty)}
  < \frac{5\|b\|^2}{8m\lambda}<\widehat{M}.
  \)
 The rest only focuses on the case that $J\neq\emptyset$. Since $\sigma_{\!A}(r)>0$, we have ${\rm rank}(A_J)=|J|$.
 From $A_J\widehat{x}_J=A\widehat{x}-A_{J^c}\widehat{x}_{J^c}$, 
 it follows that   
 \(
  	\widehat{x}_J=(A_J^{\top}A_J)^{-1}A_J^{\top}(A\widehat{x}-A_{J^c}\widehat{x}_{J^c}).
   \)
  Together with the definition of $\kappa$ and $\|A\widehat{x}-b\|\leq\|b\|$ due to \eqref{boundx0}, it holds that
  \begin{align}\label{phatx2}
  \|\widehat{x}_J\|_1
  &\leq \|(A_J^{\top}A_J)^{-1}A_J^{\top} A_{J^c}\widehat{x}_{J^c}\|_1
  		+\|(A_J^{\top}A_J)^{-1}A_J^{\top}A\widehat{x}\|_1\nonumber\\
  &\le \sum_{j\in J^c}\| (A_J^{\top}A_J)^{-1}A_J^{\top}A_{j}\widehat{x}_{j}\|_1
  		+\sqrt{r}\|(A_J^{\top}A_J)^{-1}A_J^{\top}A\widehat{x}\|\nonumber\\
  &\le \kappa\sum_{j\in J^c}|\widehat{x}_{j}|
  		+\sqrt{r}\|(A_J^{\top}A_J)^{-1}A_J^{\top}\|\|A\widehat{x}\|\nonumber\\
  &\leq \kappa\|\widehat{x}_{J^c}\|_1+\frac{2\sqrt{r}\|b\|}{\sqrt{m}\sigma_{\!A}(r)},
  \end{align}
  where the fourth inequality is due to Lemma \ref{LemmaAJ} and $\|A\widehat{x}\|\le 2\|b\|$. Then, it holds that
  \begin{equation}\label{phatx2L1}
    \|\widehat{x}\|_1\!=\!\|\widehat{x}_J\|_1+\|\widehat{x}_{J^c}\|_1      \leq(1+\kappa)\|\widehat{x}_{J^c}\|_1+\frac{2\sqrt{r}\|b\|}{\sqrt{m}\sigma_{\!A}(r)}.
  \end{equation}
  In addition, using \eqref{boundx0} again yields that $\|\widehat{x}_{J^c}\|_1-\|\xi\|_\infty\|\widehat{x}_{J^c}\|_1-\|\xi\|_\infty\|\widehat{x}_{J}\|_1\leq \frac{\|b\|^2}{2m\lambda}$,
  which along with \eqref{phatx2} and  $(1+\kappa)\|\xi\|_\infty\leq\frac{1}{5}$ implies that
 \begin{equation*}
  \|\widehat{x}_{J^c}\|_1
 \le \frac{2\sqrt{r}\|b\|\|\xi\|_\infty}{\sqrt{m}\sigma_{\!A}(r)(1-\!\|\xi\|_\infty-\!\kappa\|\xi\|_\infty)}
         +\frac{5\|b\|^2}{8m\lambda}
  \leq\frac{\sqrt{r}\|b\|}{2\sqrt{m}\sigma_{\!A}(r)}+\frac{5\|b\|^2}{8m\lambda}.\nonumber
 \end{equation*}
 Together with the above \eqref{phatx2L1}, we obtain
  \(
  \|\widehat{x}\|_1\leq \widehat{M}.
  \)
  The proof is completed.
 \end{proof}

 \subsection{Proof of Proposition \ref{prop1-nearly}}\label{app-4-2}

 \begin{proof}
 It suffices to consider that $\{j\in J^c\ |\ |\widehat{x}_j|\ge\eta\}\ne\emptyset$.
 Let $d:=\widehat{x}-x^{\rm o}$ where $x^{\rm o}$ is the oracle estimator defined by \eqref{oracle-esti}.
 Then, by Lemma \ref{Lemma-oracle}, $b-Ax^{\rm o}=(E-A_{\overline{S}}(A_{\overline{S}}^\top A_{\overline{S}})^{-1}A_{\overline{S}}^\top) e$, $x^{\rm o}_{\overline{S}^c}=\textbf{0}$, $\|x^{\rm o}\|_1\leq 0.8\widehat{M}$ and $d_{\overline{S}^c}=\widehat{x}_{\overline{S}^c}$. By Lemma \ref{bound-xhat1}, $\|\widehat{x}\|_1\leq\widehat{M}$. From the optimality of $\widehat{x}$ to \eqref{equa-JTL1}, the feasibility of $x^{\rm o}$ to \eqref{equa-JTL1}, and $J^{c}=\overline{S}^c\cup[\overline{S}\backslash J]$, it  follows that
 \begin{align}
  &\frac{1}{2m\lambda}\|A\widehat{x}-b\|^2+\sum_{i\in\overline{S}\backslash J}|\widehat{x}_i|+\sum_{i \in \overline{S}^c}|\widehat{x}_i|
    -\langle \xi, \widehat{x}\rangle
   \leq \frac{1}{2m\lambda}\|Ax^{\rm o}-b\|^2+\sum_{i \in \overline{S}\backslash J}|x^{\rm o}_i| -\langle \xi, x^{\rm o}\rangle\nonumber\\
 \Rightarrow\
 &\sum_{i\in\overline{S}^c}|d_i|-\langle \xi_{\overline{S}^c}, d_{\overline{S}^c}\rangle
 \leq \sum_{i \in \overline{S}\backslash J}(|x^{\rm o}_i|-|\widehat{x}_i|)
   +\langle \xi_{\overline{S}}, d_{\overline{S}}\rangle +\frac{\|Ax^{\rm o}-b\|^2}{2m\lambda}-\frac{\|A\widehat{x}-b\|^2}{2m\lambda}\nonumber\\
 \Rightarrow\
 & (1-\|\xi\|_\infty)\sum_{i \in \overline{S}^c}|d_i|
   \le\sum_{i \in \overline{S}\backslash J}(|x^{\rm o}_i|-|\widehat{x}_i|)
    +\| \xi\|_\infty \|d_{\overline{S}}\|_1+\frac{1}{m\lambda}\langle Ad, b-Ax^{\rm o}\rangle-\frac{\|Ad\|^2}{2m\lambda}\nonumber\\
 \Rightarrow\
 & (1-\|\xi\|_\infty)\sum_{i \in \overline{S}^c}|d_i|
   \leq \sum_{i \in \overline{S}\backslash J}(|x^{\rm o}_i|-|\widehat{x}_i|)
   +1.8\widehat{M}\| \xi\|_\infty+\frac{1}{m\lambda}\langle d_{\overline{S}^c}, A_{\overline{S}^c}^\top(b-Ax^{\rm o})\rangle-\frac{\|Ad\|^2}{2m\lambda}\nonumber\\
    \Rightarrow\
 & (1-\|\xi\|_\infty)\sum_{i \in \overline{S}^c}|d_i|
   \leq \sum_{i \in \overline{S}\backslash J}(|x^{\rm o}_i|-|\widehat{x}_i|)
   +1.8\widehat{M}\| \xi\|_\infty+\frac{\|A_{\overline{S}^c}^\top(b-Ax^{\rm o})\|_\infty}{m\lambda}\| d_{\overline{S}^c}\|_1-\frac{\|Ad\|^2}{2m\lambda},\nonumber
  \end{align}
  where the third implication is due to $\|d_{\overline{S}}\|_1\leq \|\widehat{x}\|_1+\|x^{\rm o}\|_1\leq 1.8\widehat{M}$ and $A_{\overline{S}}^\top(b-Ax^{\rm o})=\textbf{0}$. Recall that $\lambda\geq \frac{2\|A_{\overline{S}^c}(I-A_{\overline{S}}(A_{\overline{S}}^\top A_{\overline{S}})^{-1}A_{\overline{S}}^\top) e\|_\infty}{m(1-\rho)}=\frac{2\|A_{\overline{S}^c}^\top(b-Ax^{\rm o})\|_\infty}{m(1-\rho)}$, where the equality is due to Lemma \ref{Lemma-oracle} (ii). Then, $\frac{\|A_{\overline{S}^c}^\top(b-Ax^{\rm o})\|_\infty}{m\lambda}\| d_{\overline{S}^c}\|_1\le\frac{1-\rho}{2}\| d_{\overline{S}^c}\|_1$. The above inequality implies that  
  \[
    (1\!-\!\|\xi\|_\infty-\frac{1-\rho}{2})\sum_{i\in\overline{S}^c}|d_i|
   \leq \sum_{i \in \overline{S}\backslash J}(|x^{\rm o}_i|-|\widehat{x}_i|)
   +1.8\widehat{M}\| \xi\|_\infty-\frac{\|Ad\|^2}{2m\lambda}.
  \]
  By Lemma \ref{Lemma-oracle} (i) and the definition of $M$, we have $M\ge\|x^{\rm o}\|_\infty$. Using Lemma \ref{PreLemma} with $a=\overline{x}_i$ and $\varpi=d_i$ for $i\in \overline{S}\backslash J$ and combining with the above inequality leads to
  \begin{align}\label{psa1-ine}
  (1\!-\!\|\xi\|_\infty-\frac{1-\rho}{2})\sum_{i\in\overline{S}^c}|d_i|
  \!\le\!\sum_{i \in \overline{S}\backslash J} \min\{|d_i|,\,2 M-|d_i|\}+1.8\widehat{M}\| \xi\|_\infty
  -\frac{\|Ad\|^2}{2m\lambda}.
  \end{align}
  
 Suppose on the contrary that $\{j\in J^c\,|\,|\widehat{x}_j|\ge\eta\}\subset\overline{S}$ does not hold. There exists at least one index $l\in J^c$ with $|\widehat{x}_l|\geq\eta$ but $l\notin\overline{S}$, so  $\|d_{\overline{S}^c}\|_\infty=\|\widehat{x}_{\overline{S}^c}\|_{\infty}\geq \eta> \frac{18\widehat{M}\| \xi\|_\infty+5\lambda r\nu^2}{3(1-\rho)}$.
 Using the $(r\!-\!|J|,\eta)$-rRNSP($M,\rho,\nu$) of $A$,
 i.e. inequality \eqref{ineq1-RNSP} with $I=\overline{S}\backslash J$ and $S=\overline{S}$ for $d$ yields that
 \(
   \sum_{i \in \overline{S}\backslash J} \min\{|d_i|,\,2 M-|d_i|\}
    \le\rho\sum_{i \in \overline{S}^c}|d_i|+\nu\sqrt{\frac{r}{m}}\|Ad\|.
 \)
 Along with \eqref{psa1-ine}, 
  \begin{align}
  \big(\frac{1\!-\!\rho}{2}\!-\!\|\xi\|_\infty\big)\sum_{i \in \overline{S}^c}|d_i|
   \le 1.8\widehat{M}\| \xi\|_\infty+\nu\sqrt{\frac{r}{m}}\|Ad\|-\frac{1}{2m\lambda}\|Ad\|^2\le 1.8\widehat{M}\| \xi\|_\infty+\frac{\lambda r\nu^2}{2}.\nonumber
  \end{align}
  Note that $\|\xi\|_\infty<\frac{1-\rho}{5}$ and $\sum_{i\in\overline{S}^c}\left|d_i\right|\geq\|d_{\overline{S}^c}\|_\infty
  \geq\eta>0$. Then, it holds that 
  \[
   \frac{3(1\!-\!\rho)}{10}\eta\leq \big(\frac{1\!-\!\rho}{2}\!-\!\|\xi\|_\infty\big)\eta\leq\big(\frac{1\!-\!\rho}{2}\!-\!\|\xi\|_\infty\big)\|d_{\overline{S}^c}\|_\infty
    \leq 1.8\widehat{M}\| \xi\|_\infty+\frac{\lambda r\nu^2}{2},
  \]
  which contradicts to $\eta>\frac{18\widehat{M}\| \xi\|_\infty+5\lambda r\nu^2}{3(1-\rho)}$. Thus, the desired inclusion holds.

 For the second part, by noting that $\min\{|d_i|,2M-|d_i|\}\le M$ for each $i\in[n]$, from the above \eqref{psa1-ine} and $\|\xi\|_\infty<\frac{1-\rho}{5}<\frac{1}{5}$, it immediately follows that
 \begin{equation}
  \sum_{i \in \overline{S}^c}\left|d_i\right|\le\frac{10}{3}\big[(r-|J|)M+1.8\widehat{M}\| \xi\|_\infty\big].\nonumber
 \end{equation}
  Noting that $J\subsetneqq\overline{S}$ and $\{j\in J^c\ |\ |\widehat{x}_j|\ge\eta\}\subset\overline{S}$,
 we have $|\widehat{x}_j|<\eta$ for all $j\in \overline{S}^c$, which implies that
 \(
 \sum_{i \in \overline{S}^c}|d_i|
 =\sum_{i \in \overline{S}^c}|\widehat{x}_i|
 \leq (n-r)\eta.
 \)
Along with the above inequality, we have
 \begin{equation}\label{temp-ineq51-noise}
  \sum_{i \in \overline{S}^c}\left|d_i\right|\le \min\Big(\frac{10}{3}\big[(r-|J|)M+1.8\widehat{M}\| \xi\|_\infty\big], \ (n-r)\eta\Big).
  \end{equation}
  Note that ${\rm rank}(A_{\overline{S}})=r$ because $\sigma_{\!A}(r)>0$ and $|\overline{S}|=r$. Using $A_{\overline{S}}\widehat{x}_{\overline{S}}-A_{\overline{S}}x^{\rm o}_{\overline{S}}
     =-A_{\overline{S}^c}\widehat{x}_{\overline{S}^c}
     +(A\widehat{x}-Ax^{\rm o})$ leads to $\widehat{x}_{\overline{S}}-x^{\rm o}_{\overline{S}}
   =-(A_{\overline{S}}^{\top}A_{\overline{S}})^{-1}A_{\overline{S}}^{\top}A_{\overline{S}^c}\widehat{x}_{\overline{S}^c} +(A_{\overline{S}}^{\top}A_{\overline{S}})^{-1}A_{\overline{S}}^{\top}(A\widehat{x}-Ax^{\rm o})$.
  Then, 
  \begin{align}\label{ine-Tc-nearly-error}
  &\|\widehat{x}_{\overline{S}}-x^{\rm o}_{\overline{S}}\|_1
   \le    \|(A_{\overline{S}}^{\top}A_{\overline{S}})^{-1}A_{\overline{S}}^{\top}A_{\overline{S}^c}\widehat{x}_{\overline{S}^c}\|_1
     +\|(A_{\overline{S}}^{\top}A_{\overline{S}})^{-1}A_{\overline{S}}^{\top}Ad\|_1\nonumber\\
  &\qquad\quad\ \ \ \ \ \
  \le
  \sum_{j\in \overline{S}^c}\|(A_{\overline{S}}^{\top}A_{\overline{S}})^{-1}A_{\overline{S}}^{\top}A_{j}\widehat{x}_{j}\|_1
     +\sqrt{r}\|(A_{\overline{S}}^{\top}A_{\overline{S}})^{-1}A_{\overline{S}}^{\top}Ad\|\nonumber\\
  &\qquad\quad\ \ \ \ \ \ \leq
   \sum_{j\in \overline{S}^c}\|(A_{\overline{S}}^{\top}A_{\overline{S}})^{-1}A_{\overline{S}}^{\top}A_{j}\|_1|\widehat{x}_{j}|
   +\sqrt{r}\|(A_{\overline{S}}^{\top}A_{\overline{S}})^{-1}A_{\overline{S}}^{\top}\|\|Ad\|\nonumber\\
  &\qquad\quad\ \ \ \ \ \ \leq \kappa\|d_{\overline{S}^c}\|_1+\frac{\sqrt{r}\|Ad\|}{\sqrt{m}\sigma_A(r)},
 \end{align}
 where the last inequality is by the definition of $\kappa$ and Lemma \ref{LemmaAJ}. While from inequality \eqref{psa1-ine}, $\frac{\|Ad\|^2}{m}\leq 2\lambda((r-|J|)M+1.8\widehat{M}\| \xi\|_\infty)$. Substituting this inequality into \eqref{ine-Tc-nearly-error} leads to
 \[
 \|\widehat{x}_{\overline{S}}-x^{\rm o}_{\overline{S}}\|_1\leq \kappa\|d_{\overline{S}^c}\|_1+[\sigma_{\!A}(r)]^{-1}\sqrt{2r\lambda}\sqrt{(r-|J|)M+1.8\widehat{M}\| \xi\|_\infty},
 \]
 so that $\|d\|_1  \!=\!\|d_{\overline{S}}\|_1+\|d_{\overline{S}^c}\|_1       \leq(1\!+\!\kappa)\|d_{\overline{S}^c}\|_1+[\sigma_{\!A}(r)]^{-1}\sqrt{2r\lambda}\sqrt{(r-|J|)M+1.8\widehat{M}\| \xi\|_\infty}$.
 This, along with inequality \eqref{temp-ineq51-noise}, implies the desired result. The proof is completed.
 \end{proof}

 \subsection{Proof of Theorem \ref{theorem1-nearly}}\label{app-4-3}

  \begin{proof}
  By Assumption \ref{ass1} and the nonincreasing of $\{\varsigma_{k}\}_{k\in [\overline{k}-1]_+}$,
  for each $k\in[\overline{k}\!-\!1]_+$,
  \[
    \beta_{j}(\lambda)\geq\beta_{r-1}(\lambda)>\frac{18\widehat{M}\varsigma_0+5\lambda r\tau^2}{3\varrho(1-\gamma)}
    \geq\frac{18\widehat{M}\|\xi^{k}\|_\infty+5\lambda r\tau^2}{3\varrho(1-\gamma)}\quad\forall j\in[r\!-\!1]_{+},
  \]
  where the third inequality is due to $\|\xi^{k}\|_\infty\leq \varsigma_{k}\leq \varsigma_{0}$ and $\gamma\in(0, 1)$. We first prove that $I^{k}\subset\overline{S}$ for each $k\in[\overline{k}\!-\!1]$ by induction. By the definition of $\beta_{r-1}(\lambda)$, $\|x^{1}\|_\infty\geq\beta_{r-1}(\lambda)>\epsilon$, so $I^1$ is well defined and $\overline{k}\geq 2$.
  Since the $(r,\varrho\beta(\lambda))$-rSRNSP($M,\gamma,\tau$) of $A$ implies its
  $(r,\varrho\beta_0(\lambda))$-rRNSP($M,\gamma,\tau$),
  using Corollary \ref{Corollary-nearly-rho} with $J\!=(T^{0})^c=\emptyset$ and $\xi=\xi^{0}$ leads to $I^{1}=\{j\in J^c\ |\ |x^{1}_j|\geq\varrho\|x^{1}_{J^c}\|_\infty\} \subset\overline{S}$ and inequality \eqref{result421-noise} for $k=1$. Now suppose that $I^{k-1}\subset\overline{S}$ for all $k=2,\ldots,l$ with $2\le l\le \overline{k}$. 
  If $\bigcup_{j=1}^{l-1}I^{j}\ne\overline{S}$, then $\bigcup_{j=1}^{l-1}I^{j}\subsetneqq\overline{S}$. Since the $(r,\varrho\beta(\lambda))$-rSRNSP($M,\gamma,\tau$) of $A$ implies its $(r\!-\!|J|,\varrho\beta_{|J|}(\lambda))$-rRNSP($M,\gamma,\tau$), invoking Corollary \ref{Corollary-nearly-rho} with $J=\bigcup_{j=1}^{l-1}I^{j}\subsetneqq\overline{S}$ and $\xi=\xi^{l-1}$ leads to
 $I^{l}=\{j\in J^c\,|\,|x^{l}_j|\geq\varrho\|x^{l}_{J^c}\|_\infty\} \subset\overline{S}$ and inequality \eqref{result421-noise} for $k=l$. Consequently, $I^{k}\subset\overline{S}$ for each $k\in[\overline{k}\!-\!1]$, and $\bigcup_{k=1}^lI^{k}\subset\overline{S}$ for all $l\in[\overline{k}\!-\!1]$.
  Thus, it suffices to argue by contradiction that $\bigcup_{j=1}^{l-1}I^{j}\ne\overline{S}$ for all $l\in[\overline{k}\!-\!1]$. Suppose on the contrary that $\bigcup_{j=1}^{l-1}I^{j}=\overline{S}$. Then, 
 $T^{l-1}=(\bigcup_{j=1}^{l-1}I^{j})^c=\overline{S}^c$. Invoking Lemma \ref{bound-xhat1} with $J=(T^{l-1})^c$ and $\xi=\xi^{l-1}$ leads to $\|x^{l}\|_1\leq \widehat{M}$. From the feasibility of $x^{\rm o}$ to the $l$th subproblem and $x^{\rm o}_{\overline{S}^c}=\textbf{0}$, it follows that 
 \begin{align*}%\label{ineq41-nearly}
  &\frac{1}{2m\lambda}\|Ax^l-b\|^2+\sum_{i\in\overline{S}^c}|x^{l}_i|-\langle \xi^{l-1}_{\overline{S}^c}, x^{l}_{\overline{S}^c}\rangle
   \le\frac{\|Ax^{\rm o}-b\|^2}{2m\lambda}
     -\langle \xi^{l-1}_{\overline{S}}, x^{\rm o}_{\overline{S}}-x^{l}_{\overline{S}}\rangle \nonumber\\
   \Rightarrow\
  &\sum_{i\in\overline{S}^c}|x^{l}_i|-\langle \xi^{l-1}_{\overline{S}^c}, x^{l}_{\overline{S}^c}\rangle
   \leq \|\xi^{l-1}\|_\infty\|x^{\rm o}_{\overline{S}}-x^{l}_{\overline{S}}\|_1
     +\frac{\|Ax^{\rm o}-b\|^2-\|Ax^l-b\|^2}{2m\lambda}\\
 % \Rightarrow\
%   &(1-\|\xi^{l-1}\|_\infty)\|x^{l}_{\overline{S}^c}\|_1\leq
%     (1+\|\xi^{l-1}\|_\infty)s_r(\overline{x})+\|\xi^{l-1}\|_\infty\|\overline{x}_{\overline{S}}-x^{l}_{\overline{S}}\|_1
%     +\|e\|^2/(2\lambda)\nonumber\\
 \Rightarrow\
 & (1\!-\!\|\xi^{l-1}\|_\infty)\|x^{l}_{\overline{S}^c}\|_1
   \le \|\xi^{l-1}\|_\infty(\widehat{M}+\|x^o\|_1)+\frac{1}{m\lambda}\langle A(x^l\!-\!x^{\rm o}), b\!-\!Ax^{\rm o}\rangle-\frac{\|A(x^l\!-\!x^{\rm o})\|^2}{2m\lambda}.
 %   \Rightarrow\
 %   & (1\!-\!\|\xi^{l-1}\|_\infty)\|x^{l}_{\overline{S}^c}\|_1
 %     \le \|\xi^{l-1}\|_\infty(\widehat{M}+\|\overline{x}\|_1)+\frac{1}{m\lambda}\langle A_{\overline{S}^c}x^l_{\overline{S}^c}, b\!-\!Ax^{\rm o}\rangle-\frac{\|A(x^l\!-\!x^{\rm o})\|^2}{2m\lambda}\nonumber\\
 %       \Rightarrow\
 %   & (1\!-\!\|\xi^{l-1}\|_\infty)\|x^{l}_{\overline{S}^c}\|_1
 %     \le \|\xi^{l-1}\|_\infty(\widehat{M}+\|\overline{x}\|_1)+\frac{\|A_{\overline{S}^c}^\top(b\!-\!Ax^{\rm o})\|_\infty}{m\lambda}\|x^l_{\overline{S}^c}\|_1-\frac{\|A(x^l\!-\!x^{\rm o})\|^2}{2m\lambda}\nonumber\\
 %      \Rightarrow\
 %   & (1\!-\!\|\xi^{l-1}\|_\infty-\frac{1-\gamma}{2})\|x^{l}_{\overline{S}^c}\|_1
 %     \le \|\xi^{l-1}\|_\infty(\widehat{M}+\|\overline{x}\|_1)\nonumber\\
 %  \Rightarrow\
 %  &\|x^{l}_{\overline{S}^c}\|_\infty\le\|x^{l}_{\overline{S}^c}\|_1\le \frac{1.8\widehat{M}\|\xi^{l-1}\|_\infty}{1\!-\!\|\xi^{l-1}\|_\infty-(1\!-\!\gamma)/2}
 %   \leq 6\widehat{M}\varsigma_{l-1}
 %   \leq\epsilon,
  \end{align*} 
 Recall that $A_{\overline{S}}^\top(A_{\overline{S}}x^{\rm o}_{\overline{S}}-b)=\textbf{0}$ by Lemma \ref{Lemma-oracle}. Then, it holds that  
 \begin{align*}
  \frac{1}{m\lambda}\langle A(x^l\!-\!x^{\rm o}), b\!-\!Ax^{\rm o}\rangle=\frac{1}{m\lambda}\langle A_{\overline{S}^c}x^l_{\overline{S}^c}, b\!-\!Ax^{\rm o}\rangle
  \le\frac{\|A_{\overline{S}^c}^\top(b\!-\!Ax^{\rm o})\|_\infty}{m\lambda}\|x^l_{\overline{S}^c}\|_1\le \frac{1-\gamma}{2}\|x^l_{\overline{S}^c}\|_1
 \end{align*}
 where the last inequality is due to $\lambda \geq \frac{2\|A_{\overline{S}^c}^\top(I-A_{\overline{S}}(A_{\overline{S}}^\top A_{\overline{S}})^{-1}A_{\overline{S}}^\top) e\|_\infty}{m(1-\gamma)}=\frac{2\|A_{\overline{S}^c}^\top(b-Ax^{\rm o})\|_\infty}{m(1-\gamma)}$. From the above two inequalities, we obtain 
 $(1\!-\|\xi^{l-1}\|_\infty-\frac{1-\gamma}{2})\|x^{l}_{\overline{S}^c}\|_1\le \|\xi^{l-1}\|_\infty(\widehat{M}+\|x^o\|_1)$. Together with $\|\xi^{l-1}\|_{\infty}\le\varsigma_{l-1}\le\varsigma_0<\frac{1-\gamma}{5}$ and $0.8\widehat{M}\geq \|x^{\rm o}\|_1$, it follows that 
 \begin{equation}\label{ineq41-nearly}
\|x^{l}_{\overline{S}^c}\|_\infty\le\|x^{l}_{\overline{S}^c}\|_1\le \frac{1.8\widehat{M}\|\xi^{l-1}\|_\infty}{1\!-\!\|\xi^{l-1}\|_\infty-(1\!-\!\gamma)/2}
   \leq 6\widehat{M}\varsigma_{l-1}
   \leq\epsilon.
 \end{equation}
 This shows that Algorithm \ref{Alg1} stops at the $l$th iterate, so $l=\overline{k}$. On the other hand, $l<\overline{k}$ since $l\in[\overline{k}\!-\!1]$. Thus, we obtain a contradiction.
 
 Now we prove the first part. From \eqref{boundx0x}, $\|Ax^{l}-b\|\leq \sqrt{\|e\|^2+2m\lambda(1+\varsigma_0)\|\overline{x}\|_1}$ for each $l\in[\overline{k}]$. Along with the definition of $\beta_{r-1}(\lambda)$, for each $l\in[\overline{k}]$, $\|x^{l}\|_r^\downarrow\geq\beta_{r-1}(\lambda)>\epsilon$. Thus, for each $l\in[\overline{k}]$, if $(T^{l-1})^c=\bigcup_{k=1}^{l-1}I^{k}\subsetneqq\overline{S}$, then $\|x^{l}_{T^{l-1}}\|_\infty\geq\|x^{l}\|_r^\downarrow>\epsilon$. That is, for each $l\in[\overline{k}]$, Algorithm \ref{Alg1}
 does not stop at $l$th iteration if $(T^{l-1})^c=\bigcup_{k=1}^{l-1}I^{k}\neq\overline{S}$. This implies that  $\overline{k}$ is the smallest positive integer such that $(T^{\overline{k}-1})^c=\bigcup_{k=1}^{\overline{k}-1}I^{k}=\overline{S}$. Since $\|Ax^{\overline{k}}-b\|\leq \sqrt{\|e\|^2+2m\lambda(1+\varsigma_0)\|\overline{x}\|_1}$, for each $i\in {\rm supp}(x^{\overline{k}}, r)$,  $|x^{\overline{k}}_i|\ge|x^{\overline{k}}|_r^{\downarrow}\ge\beta_{r-1}(\lambda)>0$. In addition, using the same arguments as those for inequality \eqref{ineq41-nearly} with $l=\overline{k}$ yields that
 \begin{equation}\label{ineq52-nearly}
  \|x^{\overline{k}}_{T^{\overline{k}-1}}\|_\infty
  \le\|x^{\overline{k}}_{T^{\overline{k}-1}}\|_1
  \le 6\widehat{M}\varsigma_{\overline{k}-1}\le\epsilon<\beta_{r-1}(\lambda).
 \end{equation}
 The two sides imply that ${\rm supp}(x^{\overline{k}}, r)\subset(T^{\overline{k}-1})^c=\overline{S}$. Along with $|{\rm supp}(x^{\overline{k}}, r)|=r=|\overline{S}|$, we obtain that ${\rm supp}(x^{\overline{k}}, r)=T^{\overline{k}-1})^c=\overline{S}$.

 From Proposition \ref{finite-Alg1}, for each $k\in[\overline{k}]$, $I^{k}\ne\emptyset$ and $I^{k}\cap I^{l}=\emptyset$ for any $l\in[k\!-\!1]$. Along with $|\overline{S}|=r$ and $\bigcup_{k=1}^{\overline{k}-1}I^{k}=\overline{S}$, we deduce that $\overline{k}-1\leq r$, so $1<\overline{k}\leq r+1$, and \eqref{result421-noise} holds with $1\le k\le \overline{k}-1$ by the above arguments. The rest only needs to derive the error bound of $x^{\overline{k}}$ from $x^{\rm o}$.
 Using inequality \eqref{psa1-ine} with $J=(T^{\overline{k}-1})^c=\overline{S}$ yields that $\frac{\|A(x^{\overline{k}}-x^{\rm o})\|}{\sqrt{m}}\leq \sqrt{3.6\lambda\widehat{M}\varsigma_{\overline{k}-1}}$. Note that ${\rm rank}(A_{\overline{S}})=r$ due to $\sigma_{\!A}(r)>0$ and $|\overline{S}|=r$. Using the same arguments as for \eqref{ine-Tc-nearly-error}, we have
 \(
  \|x^{\rm o}_{\overline{S}}-x^{\overline{k}}_{\overline{S}}\|_1
  \leq \kappa\|x^{\overline{k}}_{\overline{S}^c}\|_1+\frac{\sqrt{r}\|A(x^{\overline{k}}-x^{\rm o})\|}{\sqrt{m}\sigma_{\!A}(r)}.
 \)
 Thus, 
 \begin{equation*}
\|x^{\overline{k}}-x^{\rm o}\|_1  \leq(1\!+\!\kappa)\|x^{\overline{k}}_{\overline{S}^c}\|_1+\frac{\sqrt{r}\|A(x^{\overline{k}}-x^{\rm o})\|}{\sqrt{m}\sigma_{\!A}(r)}\leq 6(1\!+\!\kappa)\widehat{M}\varsigma_{\overline{k}-1}+\frac{\sqrt{3.6\lambda r\widehat{M}\varsigma_{\overline{k}-1}}}{\sigma_{\!A}(r)},\nonumber
 \end{equation*}
 where the second inequality is due to \eqref{ineq52-nearly}.
  The desired inequality holds.
 \end{proof}

 \subsection{Proof of Theorem \ref{theorem2-nearly}}\label{app-4-4}

 \begin{proof}
 Let $\beta(\lambda)=\big\{\beta_0(\lambda),\ldots,\beta_{r-1}(\lambda)\big\}$. By Remark \ref{Rem-rSRNSP}, the $(r, \varrho\beta'(\lambda))$-rSRNSP($M,\gamma,\tau$) of $A$ implies its $(r,\varrho\beta(\lambda))$-rSRNSP($M,\gamma, \tau$). Then, the conditions of Theorem \ref{theorem1-nearly} hold, so $\overline{k}\in\{2,\ldots,r\!+\!1\}$ is the smallest positive integer such that $\bigcup_{k=1}^{\overline{k}-1}I^k=\overline{S}$. If $\bigcup_{j=1}^{\widehat{k}-1}I^{j}=\overline{S}$, it is not hard to deduce that $\widehat{k}=\overline{k}$, and the conclusions follow by Theorem \ref{theorem1-nearly}. Now suppose that $J:=\bigcup_{j=1}^{\widehat{k}-1}I^{j}\subsetneqq\overline{S}$. From the proof of Theorem \ref{theorem1-nearly}, Algorithm \ref{Alg1} does not stop at $(\widehat{k}\!-\!1)$th step. By using $\widetilde{r}\le|J|\le r\!-\!1$ and $\beta'_{|J|}(\lambda)=\beta_{r-1}(\lambda)$, it follows that the $(r,\varrho\beta'(\lambda))$-rSRNSP($M,\gamma, \tau$) of $A$
  implies its $(r\!-|J|,\varrho\beta_{r-1}(\lambda))$-rRNSP($M,\gamma,\tau$).
  Thus, by using Corollary \ref{corollary1-nearly} and the definition of $x^{\widehat{k}}$, we obtain ${\rm supp}(x^{\widehat{k}},r)=\overline{S}$ and \eqref{result422-noise}.

  For the second part, as $\beta'(\lambda)=\big\{\beta_0(\lambda),\beta_1(\lambda),\ldots, \beta_{\widetilde{r}-1}(\lambda),\underbrace{0,\ldots,0}_{r-\widetilde{r}}\big\}$, invoking Remark \ref{Rem-rSRNSP} and the first part of conclusions results in ${\rm supp}(x^{\widehat{k}},r)=\overline{S}$ and
  $J:=\bigcup_{j=1}^{\widehat{k}-1}I^{j}\subset\overline{S}$. Let $d^{\widehat{k}}=x^{\widehat{k}}\!-\!x^{\rm o}$. By the definition of $x^{\widehat{k}}$, using the same arguments as those for \eqref{psa1-ine} yields that
 \begin{align}\label{psa1-ine-nearly}
   (1-\|\xi^{\widehat{k}-1}\|_\infty-\frac{1-\gamma}{2})\sum_{i \in \overline{S}^c}|d_i^{\widehat{k}}|
  \le&\sum_{i \in \overline{S}\backslash J} \min\big\{|d_i^{\widehat{k}}|, 2M-|d_i^{\widehat{k}}|\big\}
  +1.8\widehat{M}\varsigma_{\widehat{k}-1}-\frac{\|Ad^{\widehat{k}}\|^2}{2m\lambda}.
  \end{align}
  If $J\subsetneqq\overline{S}$, we have $\widetilde{r}\le|J|\le r-1$ and $\beta'_{|J|}(\lambda)=0$. Then, the $(r,\varrho\beta'(\lambda))$-rSRNSP($M,\gamma,\tau$) of $A$ implies its $(r\!-|J|, 0)$-rRNSP($M,\gamma,\tau$). From \eqref{ineq1-RNSP} with $I=\overline{S}\backslash J,S=\overline{S}$ and \eqref{psa1-ine-nearly},
  \[
   (1-\|\xi^{\widehat{k}-1}\|_\infty-\frac{1-\gamma}{2})\sum_{i \in \overline{S}^c}|d^{\widehat{k}}_i|
     \leq \gamma\sum_{i \in \overline{S}^c}|d^{\widehat{k}}_i|+1.8\widehat{M}\varsigma_{\widehat{k}-1}+\tau\sqrt{\frac{r}{m}}\|Ad^{\widehat{k}}\|-\frac{1}{2m\lambda}\|Ad^{\widehat{k}}\|^2.
  \]
  Combining the above inequality with $\tau\sqrt{\frac{r}{m}}\|Ad^{\widehat{k}}\|-\frac{1}{2m\lambda}\|Ad^{\widehat{k}}\|^2\leq \frac{\lambda r\tau^2}{2}$ yields that
  \begin{equation}\label{ineq42-index-error-nearly}
  \sum_{i \in \overline{S}^c}|d^{\widehat{k}}_i|\leq \frac{1.8\widehat{M} \varsigma_{\widehat{k}-1}+\lambda r\tau^2/2}{(1-\gamma)/2-\|\xi^{\widehat{k}-1}\|_\infty}
  \leq\frac{18\widehat{M} \varsigma_{\widehat{k}-1}+5\lambda r\tau^2}{3(1-\gamma)},
  \end{equation}
  where the second inequality is due to $\|\xi^{\widehat{k}-1}\|_\infty\leq\varsigma_{\widehat{k}-1}\leq\varsigma_{0}<\frac{1-\gamma}{5}$.
  Following the same argument as those for $\|d\|_1    \leq(1+\kappa)\|d_{\overline{S}^c}\|_1+\frac{\sqrt{2r\lambda}}{\sigma_{\!A}(r)}\sqrt{(r\!-\!|J|)M+1.8\widehat{M}\| \xi\|_\infty}$
 in the proof of Proposition \ref{prop1-nearly}, we have 
  \(
  \|d^{\widehat{k}}\|_1
    \leq(1+\kappa)\|d^{\widehat{k}}_{\overline{S}^c}\|_1 +\frac{\sqrt{2r\lambda}}{\sigma_{\!A}(r)}\sqrt{(r-\widetilde{r})M+1.8\widehat{M}\varsigma_{\widehat{k}-1}}.
  \)
 Along with \eqref{ineq42-index-error-nearly}, we get the desired inequality. The second part follows.
 \end{proof}

\vskip 0.2in
\bibliography{iSCRA_JMLR}

\end{document}